\documentclass{article}
\usepackage{blindtext}
\usepackage{amsfonts,amsthm}
\usepackage{mathtools}
\usepackage{amsmath} 
\usepackage{amssymb}
\usepackage{authblk}
\usepackage{natbib}
\usepackage{babel}
\usepackage{graphicx}
\usepackage{graphics}
\usepackage[dvipsnames]{xcolor} 
\usepackage{bm}
\usepackage[margin = 1.2in]{geometry}
\usepackage{hyperref}
\usepackage{pgfplots}
\pgfplotsset{width = 5cm, compat = 1.17}
\usepackage{algpseudocode}
\usepackage{algorithm}

\theoremstyle{theorem}
\newtheorem{proposition}{Proposition}[section]
\newtheorem{theorem}{Theorem}[section]
\newtheorem{lemma}{Lemma}[section]

\theoremstyle{definition}
\newtheorem{definition}{Definition}
\newtheorem{example}{Example}
\newtheorem{remark}{Remark}

\renewcommand{\Pr}{\mathrm{P}}
\renewcommand{\P}{\mathbb{P}}

\newcommand{\R}{\mathbb{R}}
\newcommand{\E}{\mathbb{E}}

\newcommand{\var}{\text{Var}}
\newcommand{\indpt}{\perp \!\!\! \perp}
\newcommand{\tikzcircle}[2][red,fill=red]{\tikz[baseline=-0.5ex]\draw[#1,radius=#2] (0,0) circle ;}

\definecolor{darkForestGreen}{rgb}{0.0,0.5,0.0}
\definecolor{darkyellow}{rgb}{0.255,0.2,0.0}
\definecolor{darkgreen}{rgb}{0.0,0.5,0.0}

\newcommand{\xvec}{\mbox{\bf x}}

\newcommand{\sigmat}{\mbox{\boldmath $\Sigma$}}

\allowdisplaybreaks 

\title{Conditional Independence Testing Using Exchangeable Pairs}
\author{Bilol Banerjee}
\affil{Department of Statistics and Data Science\\
National University of Singapore\\
E-mail: bilol@nus.edu.sg}
\date{\null}

\begin{document}

\maketitle
\begin{abstract}

This article considers the problem of testing conditional independence between two random vectors \(bm X\) and \(\bm Y\) given a confounding random vector \(\bm Z\). An exchangeable-pairs framework is introduced through which the conditional independence testing problem is reformulated as a two-sample testing problem. The framework is motivated by ideas from the model-X literature and is based on a fundamental exchangeability property that holds under the null hypothesis of conditional independence. An energy-distance/maximum mean discrepancy type measure is employed on the resulting exchangeable pairs to quantify departures from conditional independence. A consistent estimator of the proposed discrepancy measure is constructed and its theoretical properties are established under general assumptions. A conditional independence test is then developed using this estimator as a test statistic and is calibrated through a suitable resampling procedure. It is shown that the proposed test is consistent against fixed alternatives, possesses nontrivial asymptotic power against local contiguous alternatives, attains the minimax separation rate for detecting alternatives characterized by the proposed discrepancy measure, and remains consistent when the data dimension diverges with the sample size. The effect of estimating the conditional distribution used to generate the exchangeable pairs is also investigated, and condition under which validity and power properties are preserved is established. Extensive simulation studies demonstrate that the proposed procedure performs competitively with some state-of-the-art methods.

\vspace{0.1in}

\end{abstract}

\section{Introduction}

The notion of conditional dependence, formalized by \cite{dawid1980}, plays a significant role in the field of statistics. For example, it frequently arises in graphical modeling \citep{lauritzen1996graphical,maathuis2018handbook}, causal discovery \citep{pearl1988graphicalmodelbook,peters2017elements}, variable selection \citep{candes2018panning,azadkia2021simple,deb2020measuring}, dimension reduction \citep{cook2002dimension,li2018sufficient}, and computational biology \citep{markowetz2007inferring}. In this article, for a given random vector $(\bm X, \bm Y, \bm Z)$, the notation $\bm X\indpt\bm Y\mid\bm Z$ is used to signify that $\bm X$ is independent of $\bm Y$ given $\bm Z$ and $\bm X\not\indpt\bm Y\mid\bm Z$ is used to denote that $\bm X$ and $\bm Y$ are conditionally dependent given $\bm Z$. The hypotheses of interest are:
\begin{align}
    H_0: \bm X\indpt\bm Y\mid\bm Z~~~~~~ \text{against}~~~~~~ H_1:\bm X\not\indpt\bm Y\mid \bm Z.
    \label{eq:fundamental-problem}
\end{align}
Based on an i.i.d. sample $\mathcal{D} = \{(\bm X_i,\bm Y_i,\bm Z_i)\}_{1\leq i\leq n}$ on the random vector $(\bm X,\bm Y,\bm Z)\sim\Pr$, a test for conditional independence is developed that has some nice theoretical properties. Initial developments in this direction can be traced back to the work in \cite{fisher1924}, where the author studied the theoretical distribution of the partial correlation coefficient under the Gaussianity assumption. Though the partial correlation has desirable local asymptotic properties, it is highly sensitive to the underlying Gaussianity assumption. However, modern computational and analytic tools enabled researchers to store and analyze data from more complicated distributions. This necessitates the development of new methodologies that are applicable in general. A potential alternative to tackle the problems is by taking a nonparametric approach. However, in the context of nonparametric conditional independence testing \cite{shah2020hardness,neykov2021minimax,kim2021local} established the `no-free lunch theorems'. It is now accepted that it is impossible to construct a test for \eqref{eq:fundamental-problem} which simultaneously keeps the type I error rate controlled at a desired level $\alpha\in (0,1)$ and is consistent against any general alternative, unless one assumes some structural conditions on the underlying distribution. Numerous authors have tried to address this issue by proposing characterizing measures under suitable moment and smoothness assumptions on the underlying distribution. Some notable ones include measures based on conditional cumulative distribution functions \citep{linton1996conditional,patra2016nonparametric}, conditional characteristic functions \citep{su2007}, conditional probability density functions \citep{su2008}, empirical likelihood \citep{su2014}, kernel methods \citep{fukumizu2007kernel,zhang2012kernel,doran2014permutation,scetbon2022asymptotic,strobl2019approximate,sheng2023distance,kernel2024practical}, mutual information and entropy \citep{runge2018conditional,li2024k}, Rosenblatt transformations  \citep{cai2022,song2009}, copulas \citep{Bergsma2004TestingCI,veraverbeke2011estimation}, maximal nonlinear conditional correlation \citep{huang2010}, distance correlation \citep{fan2020,szekely2014,wang2015}, regression based approaches \citep{shah2020hardness,petersen2021,Scheidegger2022}, geometric graph-based measures \citep{azadkia2021simple,huang2020kernel,shi2024}, among several others. The readers are referred to \cite{li2020} for a detailed review of the existing literature.

It is widely known that for any nonparametric conditional independence test, the main difficulty lies in the test calibration step. It is important to ensure type I and type II error rate control without losing statistical efficiency. Achieving this for conditional independence tests is challenging in general. Recently, \cite{candes2018panning} proposed the conditional randomization test (CRT) framework. They designed a bootstrap approach to calibrate conditional independence tests assuming the distribution of $\bm X\mid \bm Z$ is known. This can control the type I error rate at any desired level $\alpha\in(0,1)$, but at the expense of restricting the class of distributions to those where $\bm X\mid \bm Z$ is some pre-specified distribution. \cite{berrett2019} also proposed the conditional permutation test that operates under the same assumption as the CRT and has the same desired level properties. However, it is computationally more challenging than the CRT. The authors also studied the robustness of their method when the distribution of $\bm X\mid \bm Z$ is approximated by some other distribution. In the literature, this class of approaches is popularly termed the model-X approach. This framework was popularized by the work of \cite{candes2018panning,candes2019selectiveknockoff} in the context of controlled variable selection and selective inference. \cite{candes2019genehunting, katsevich2019knockoff} provided applications of the framework for genetic datasets.

The model-X approach became popular in the literature of conditional independence testing due to its applicability with any test statistic yielding a test having the desired type I error rate control. Among such tests, \cite{shi2024} proposed a test based on the Azadkia--Chatterjee partial correlation coefficient \citep[][]{azadkia2021simple}. The Azadkia--Chatterjee partial correlation coefficient assumes $Y$ to be univariate, but $\bm X$ and $\bm Z$ can be multi-dimensional. It takes values in $[0,1]$; it is zero if and only if $\bm X\indpt Y\mid \bm Z$ and it is one if and only if $Y$ is a measurable function of $(\bm X,\bm Z)$. \cite{huang2020kernel} generalized this measure for metric-valued data sets, which is termed the kernel partial correlation coefficient, and proposed a test for \eqref{eq:fundamental-problem} using the CRT approach. However, the test based on the Azadkia--Chatterjee partial correlation coefficient has poor performance against local alternatives \citep[see][]{shi2024}. The test proposed by \cite{huang2020kernel} is expected to share similar properties. %There are several approaches to mitigate the issues with these measures. Some notable recent contributions include \cite{azadkia2025bias, tran2024rank}.

Moreover, implementing the CRT is computationally expensive as it requires computing the test statistic every time a new sample is generated in the bootstrap iteration. Also, these partial correlation coefficients are estimated using the local neighbourhood averaging method based on the confounding variable $\bm Z$. Therefore, when $\bm Z$ is high dimensional, as is common in graphical models or the variable selection literature, the local neighbourhood structure of $\bm Z$ will be affected due to the curse of dimensionality. Since in $d$ dimensions the expected distance between two random points has a lower bound of order $O(n^{-1/d})$ \citep[see Chapter 2 from][]{gyorfi2002distribution}, the local averaging based tests will suffer from poor performance when $\bm Z$ is high dimensional. \cite{liu2021dcrt} considered this problem and proposed a fast CRT via distilling the information of $\bm Z$ from $\bm X$ and $\bm Y$ separately and using the residuals to construct a test. However, their method measures only linear dependence among the variables, which is strongly dependent on the distillation method. Therefore, their approach may not have satisfactory performance if the distillation step is poorly specified.

\subsection{Article Contributions}

Considering these limitations, in this article, a test of conditional independence is developed that can measure non-linear conditional dependence, computationally fast, efficient against local alternatives, and is applicable even when the dimension of the random vector $(\bm X,\bm Y,\bm Z)$ is high compared to the available sample size. The article is organized as follows.

\begin{itemize}
    \item In Section \ref{sec:methodology}, the model-X framework is introduced and a motivation is given about how it can be used for constructing a pair of random vectors is defined which is exchangeable under $H_0$. Using this pair of random vectors a measure of conditional dependence is defined. A consistent estimator of the measure is proposed and some of its desired theoretical properties are established. The estimator is used to construct a test and a new resampling algorithm is designed to calibrate the test utilizing the exchangeability under $H_0$, without generating further new samples. The proposed resampling algorithm has the desired type I error rate control and consistency properties.

    \item In Section \ref{sec:theory}, it is shown that the proposed resampling test is efficient against local contiguous alternatives, minimax rate optimal against a suitable class of alternatives and is consistent against high-dimensional alternatives.

    \item  The empirical performance of the test is evaluated against some state-of-the-art methods using simulated data sets in Section \ref{sec:simulation}.

    \item In Section \ref{sec:estimation-effect}, the effect of using an approximation of $P_{\bm X|\bm Z}$ on the level and power properties of the test is studied for multivariate observations. The effect of high-dimensionality is left as a potential future research direction.
\end{itemize}

\noindent Proofs of all theorems, lemmas, and propositions are provided in the supplementary material.

\section{Methodology}
\label{sec:methodology}

A review of the model-X framework and the notion of null exchangeable pairs is introduced in Section \ref{subsec:mx-civar}. A new measure of conditional dependence based on null exchangeable pairs is then proposed in Section \ref{subsec:the-measure}. A consistent estimator of this measure is developed in Section \ref{subsec:estimation}, and its theoretical properties are studied. Finally, a test of conditional independence based on this estimator is proposed in Section \ref{subsec:test}.

\subsection{The Model-X Framework and The Null Exchangeable Pairs}
\label{subsec:mx-civar}

The model-X framework was popularized by the seminal work of \cite{candes2018panning} in the context of controlled variable selection, where the goal is to identify important covariates $\bm X$ from observations on $(Y, \bm X)$ while controlling the false discovery rate. In this framework, the distribution of $\bm X$ is assumed to be known. The authors argue that in many practical settings (e.g., genetic experiments and clinical trials), obtaining knowledge about the distribution of the covariates $\bm X$ is often easier than specifying the joint distribution of $(Y, \bm X)$. Building on this idea, they proposed the conditional randomization test (CRT) for testing conditional independence under the assumption that the distribution of $\bm X \mid \bm Z$ is known, while leaving the distribution of $Y \mid \bm X, \bm Z$ completely unspecified. Their approach constructs a bootstrap procedure based on the known distribution of $\bm X \mid \bm Z$, which can be combined with an arbitrary test statistic within this restricted model class. The CRT has subsequently facilitated several recent advances in the literature on conditional independence testing; see \cite{liu2021dcrt, niu2022reconciling, katsevich2022UMPMX}.

\begin{figure}[t]
    \centering
    \includegraphics[width=0.45\linewidth]{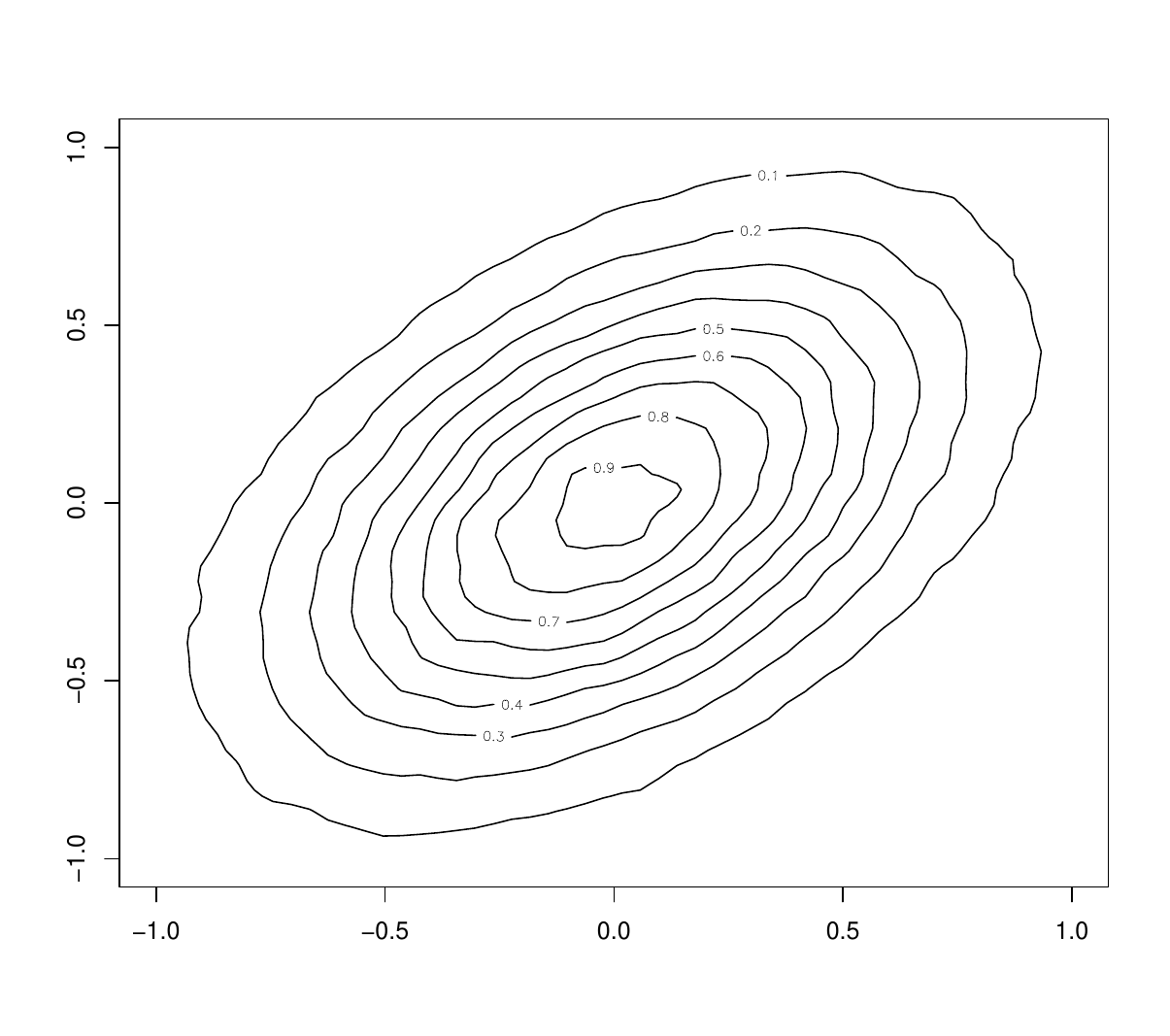}
    \includegraphics[width=0.45\linewidth]{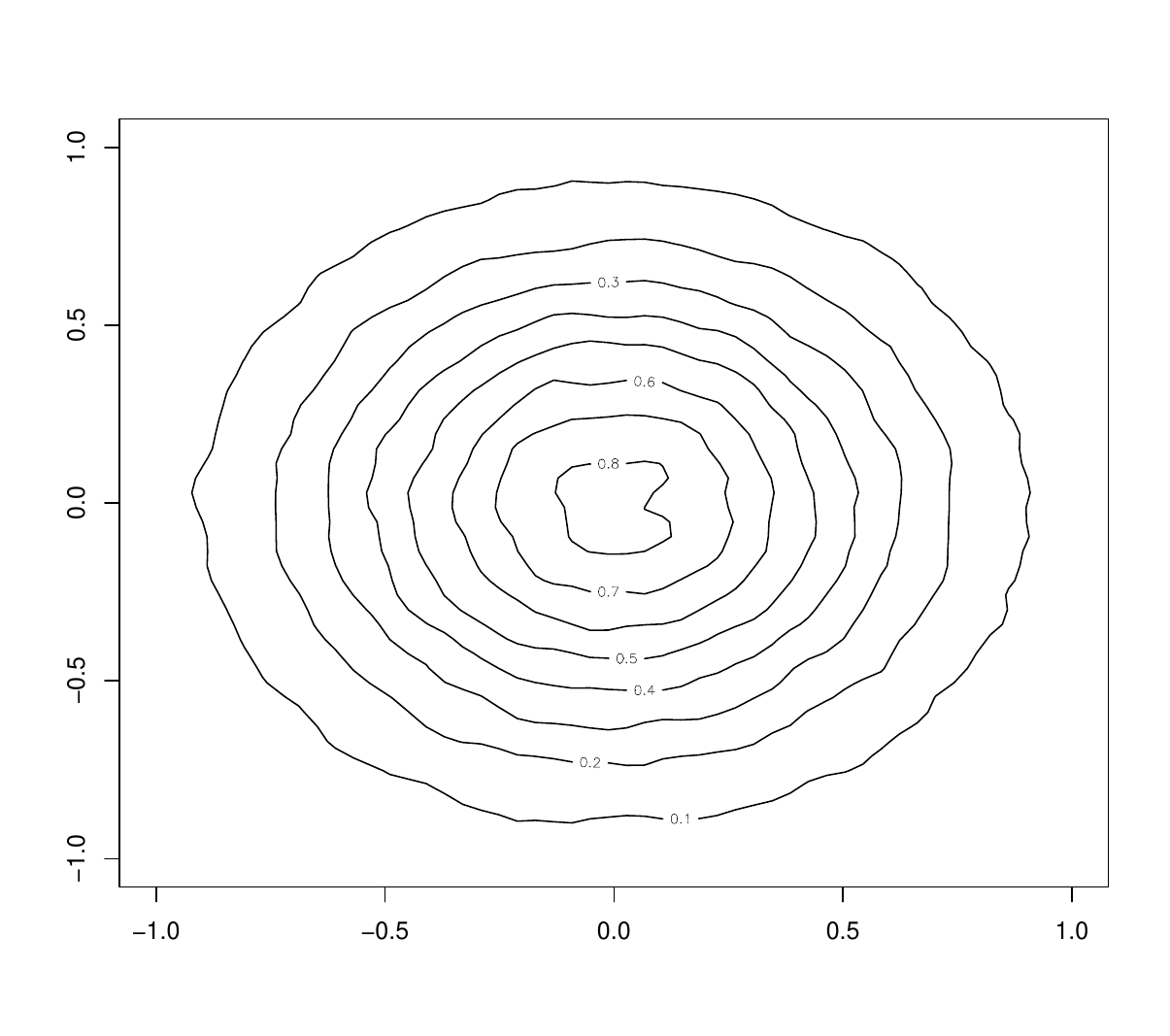}
    \caption{The contour of the probability density function of $(X,Y)|Z=0$ and $(X^\prime, Y)|Z=0$ when $(X,Y,Z)$ follows a trivariate normal distribution with mean at the origin and variance-covariance matrix being the equicorrelation matrix with correlation $\rho_0 = 0.8$.}
    \label{fig:density-contour-plot}
\end{figure}

%The foundation of the CRT framework is to capture the strength of conditional dependence by the change in the distribution of the corresponding test statistic under $H_0$ and $H_1$. Therefore, the performance of the CRT based tests depends heavily on the underlying test statistic. If the test statistic fails to discriminate between $H_0$ and $H_1$ (even under restricted setup), the CRT will also fail to reject $H_0$. This can be avoided if a test statistic is constructed using the knowledge of $\bm X|\bm Z$.

The foundation of the CRT framework lies in quantifying the strength of conditional dependence through the change in the distribution of the chosen test statistic under $H_0$ and $H_1$. Consequently, the performance of CRT-based tests depends critically on the choice of the test statistic. If the test statistic fails to adequately discriminate between $H_0$ and $H_1$ (even under the restricted model assumptions), the CRT will also fail to reject $H_0$. This limitation can be mitigated by constructing test statistics that explicitly exploit the knowledge of the distribution of $\bm X \mid \bm Z$.

%The CRT allows us to generate new observations $\bm X_i^\prime\sim \P_{\bm X|\bm Z}$ such that $\bm X^\prime\indpt\bm Y|\bm Z$. So, when $H_1$ is true, there is a clear distributional difference between the random vectors $(\bm X,\bm Y,\bm Z)$ and $(\bm X^\prime, \bm Y, \bm Z)$. Consider a motivating example. Assume $(X,Y,Z)$ \footnote{Throughout this article, boldface capital letters $\bm X, \bm Y, \bm Z$ are used to denote multivariate random vectors and $X,Y,Z$ are used to denote univariate random variables.} follows a trivariate normal distribution with mean at the origin, and the variance-covariance matrix being the equicorrelation matrix with correlation $\rho_0 = 0.8$. Note that $X\not\indpt Y|Z$ and $X|Z\sim \mathcal{N}_1(Z, 1-\rho_0^2)$. Figure \ref{fig:density-contour-plot}, displays the contour plots of the probability density function of $(X,Y)|Z=0$ and $(X^\prime, Y)|Z = 0$, where $X^\prime\sim \mathcal{N}_1(Z, 1-\rho_0^2)\stackrel{D}{=}X|Z$. Clearly, the joint distributions of $(X,Y)|Z$ and $(X^\prime, Y)|Z$ are different when $X\not\indpt Y|Z$. On the other hand, $(X,Y,Z)\stackrel{D}{=}(X^\prime, Y,Z)$ under $H_0$. In this article, the random vector $(X^\prime, Y, Z)$ is referred to as the conditionally independent variant of $(X, Y, Z)$. The notion of conditionally independent variant is formally defined as follows.

The CRT allows us to generate new observations $\bm X_i' \sim \mathbb{P}_{\bm X \mid \bm Z}$ such that $\bm X' \indpt \bm Y \mid \bm Z$. Therefore, when $H_1$ holds, there is a clear distributional difference between the random vectors $(\bm X, \bm Y, \bm Z)$ and $(\bm X', \bm Y, \bm Z)$. To illustrate this idea, consider the following motivating example. Suppose $(X,Y,Z)$\footnote{Throughout this article, boldface capital letters $\bm X, \bm Y, \bm Z$ denote multivariate random vectors, whereas $X,Y,Z$ denote univariate random variables.} follows a trivariate normal distribution with mean zero and variance--covariance matrix given by the equicorrelation matrix with correlation $\rho_0 = 0.8$. In this case, $X \not\indpt Y \mid Z$ and $X \mid Z \sim \mathcal{N}_1(Z,\, 1-\rho_0^2)$. 

Figure \ref{fig:density-contour-plot} displays the contour plots of the joint density of $(X,Y)\mid Z=0$ and $(X',Y)\mid Z=0$, where $X' \sim \mathcal{N}_1(Z,\, 1-\rho_0^2) \stackrel{D}{=} X \mid Z$. Clearly, the joint distributions of $(X,Y)\mid Z$ and $(X',Y)\mid Z$ differ when $X \not\indpt Y \mid Z$. On the other hand, under $H_0$ we have $(X,Y,Z) \stackrel{D}{=} (X',Y,Z)$. In this article, the random vector $(X',Y,Z)$ is referred to as the \emph{null exchangeable pair} of $(X,Y,Z)$. The notion of a conditionally independent variant is formally defined below.

\begin{definition}
    Let $(\bm X,\bm Y, \bm Z)$ be a random vector. A random vector $(\bm X', \bm Y,\bm Z)$ is said to be a null exchangeable pair of $(\bm X, \bm Y, \bm Z)$ if $\bm X' \mid \bm Z \stackrel{D}{=} \bm X \mid \bm Z$ and $\bm X' \indpt \bm Y \mid \bm Z$.
\end{definition}

The notion of null exchangeable pairs facilitates measuring the strength of conditional dependence by borrowing ideas from the two-sample testing literature \citep[see][]{zekely2023book}. In particular, conditional dependence between $\bm X$ and $\bm Y$ given $\bm Z$ can be assessed by comparing the distributions of $(\bm X,\bm Y,\bm Z)$ and its null exchangeable pair $(\bm X',\bm Y,\bm Z)$. This idea forms the basis of the framework developed in the following section.

\subsection{The Measure of Dependence}
\label{subsec:the-measure}
Let $\varphi_1$ and $\varphi_2$ denote the characteristic functions of $\bm V = (\bm X,\bm Y,\bm Z)$ and $\bm V' = (\bm X',\bm Y,\bm Z)$—the null exchangeable pair of $\bm V$—respectively. It can be verified that $\bm X \indpt \bm Y \mid \bm Z$ if and only if $\varphi_1 = \varphi_2$. Motivated by this observation, a measure of conditional dependence is defined as
\begin{align}
    \zeta_W(\Pr) = \int \Big|\varphi_1(\bm t,\bm u, \bm v)-\varphi_2(\bm t,\bm u, \bm v)\Big|^2 
    W(\bm t, \bm u, \bm v)\, \mathrm d \bm t \, \mathrm d \bm u \, \mathrm d \bm v,
    \label{eq:definition-of-measure}
\end{align}
where $W(\bm t, \bm u, \bm v)$ is a probability density function. Alternative choices of $W$ similar to those considered in \cite{zekely2023book} are also possible. However, such choices typically require additional moment conditions on the distribution of $(\bm X, \bm Y, \bm Z)$ to ensure that $\zeta_W(\cdot)$ is finite. For this reason, $W$ is assumed to be a probability density function in this article. 

Clearly, $\zeta_W(\Pr)$ is non-negative and equals zero if and only if $\varphi_1 = \varphi_2$ under suitable choices of $W$. This result is formally stated in the following theorem.

\begin{proposition}
    $\zeta_W(\mathrm P)$ is non-negative, and if $W$ has support on the entire space, then $\zeta_W(\mathrm P)=0$ if and only if $\bm X \indpt \bm Y \mid \bm Z$.
    \label{thm:characterization}
\end{proposition}
Note that $\zeta_W(\mathrm P)$ is the weighted average of 
$\big|\varphi_1(\bm t,\bm u, \bm v)-\varphi_2(\bm t,\bm u, \bm v)\big|^2$ 
with respect to the weight function $W$. In this article, $W$ is taken to be the density of the multivariate normal distribution $\mathcal{N}(\bm 0, \sigma^2 \mathrm I_d)$, that is, the normal distribution with mean $\bm 0$ and variance--covariance matrix $\sigma^2 \mathrm I_d$ for some $\sigma^2>0$. For this choice of $W$, a closed-form expression for $\zeta_W(\mathrm P)$ can be obtained as follows.

%Therefore, it can be applied even when the variables $\bm X, \bm Y$ and $\bm Z$ are multi-dimensional. However, estimation of $\zeta(\mathrm P)$ in the above form might be difficult. To simplify our measure, we derive an alternative form of $\zeta(\mathrm P)$ using the following lemma:

%\begin{lemma}
%    \label{lemma:key-integral}
%\end{lemma}

%\noindent Using Lemma \ref{lemma:key-integral}, we get 
\begin{proposition}
  For any $\bm v_1,\bm v_2\in\R^d$,
\[
\int \exp\left\{i\langle\bm s, \bm v_1-\bm v_2 \rangle\right\}
\frac{1}{(2\pi\sigma^2)^{d/2}}
\exp\left\{-\frac{\|\bm s\|^2}{2\sigma^2}\right\}
\, \mathrm d \bm s
=
\exp\left\{-\frac{\sigma^2}{2}\|\bm v_1-\bm v_2\|^2\right\}.
\]

Using the above identity, the measure can be expressed as
\begin{align}
\label{eq:measure}
\zeta_\sigma(\mathrm P)
&:= \zeta_{\mathcal{N}(\bm 0, \sigma^2 \mathrm I_d)}(\mathrm P) = \E\big[K(\bm V_1, \bm V_2)\big]+\E\big[K(\bm V_1^\prime, \bm V_2^\prime)\big] - 2 \E\big[K(\bm V_1, \bm V_2^\prime)\big],
\end{align}
where $K(\bm x, \bm y) = \exp\{-\sigma^2\|\bm x- \bm y\|^2/2\}$ $\bm V_1 = (\bm X_1,\bm Y_1,\bm Z_1)$ and $\bm V_2 = (\bm X_2,\bm Y_2,\bm Z_2)$ are independent copies of $(\bm X,\bm Y,\bm Z)\sim \mathrm P$, and $\bm V_i'$ denote their respective null exchangeable pairs.
    \label{thm:alt-representation}
\end{proposition}

Note that $\zeta_\sigma(\mathrm P)$ can also be interpreted as the kernel Maximum Mean Discrepancy \citep[see][]{gretton2012kernel} between the distribution of $(\bm X, \bm Y, \bm Z)$ and its null exchangeable pair $(\bm X',\bm Y, \bm Z)$ with respect to the Gaussian kernel 
\(
K(\bm x, \bm y) = \exp\left\{-\sigma^2\|\bm x-\bm y\|^2/2\right\}.
\)
Other choices of the weight function $W$, such as the Cauchy or Laplace distributions, may also be considered. These choices lead to analogous representations involving the Laplace kernel or the inverse quadratic kernel. In the following sections, all theoretical results are presented for $\zeta_\sigma(\mathrm P)$, as defined in \eqref{eq:measure}. However, most of these results admit straightforward extensions when the Gaussian kernel is replaced by other bounded kernels.

It is important to mention that the main conceptual contribution of this article is not the particular choice of the discrepancy measure, but rather the reformulation of conditional independence testing as a two-sample testing problem through the notion of a null exchangeable pairs. Specifically, conditional dependence between $\bm X$ and $\bm Y$ given $\bm Z$ is quantified by comparing the distribution of $(\bm X,\bm Y,\bm Z)$ with that of its null exchangeable pair $(\bm X',\bm Y,\bm Z)$. Once this reformulation is available, a broad range of two-sample discrepancy measures may be employed. Examples include energy distance,  Wasserstein distances, graph-based two-sample statistics, and classifier-based discrepancies. In the present work, we focus on $\zeta_\sigma(\mathrm P)$ because of its popularity in the two-sample testing literature, its close connection with energy-distance-type measures, and the availability of a simple bounded-kernel representation that facilitates both theoretical analysis and efficient computation. 
The resulting methodology should therefore be viewed as one particular instance of a more general null exchangeable pairs framework.

\subsection{Estimation of $\zeta_\sigma(\mathrm P)$}
\label{subsec:estimation}
Before estimating $\zeta_\sigma(\mathrm P)$, note that it can be expressed as a combination of three functionals. Among these, the terms $\E[K(\bm V_1', \bm V_2')]$ and $\E[K(\bm V_1, \bm V_2')]$ are not directly estimable based solely on an i.i.d. sample $\mathcal{D} = \{\bm V_i = (\bm X_i, \bm Y_i, \bm Z_i)\}_{1\leq i\leq n}$ drawn from the distribution $\mathrm P$. However, the knowledge of the conditional distribution $\bm X \mid \bm Z$ can be exploited to adopt a data augmentation approach for estimating $\zeta_\sigma(\mathrm P)$, as described below.

\begin{itemize}
    \item Generate $\bm X_i' \sim \bm X \mid \bm Z_i$ for $i = 1,2,\ldots, n$ to obtain the augmented dataset
    \(
    \mathcal{D}' = \{(\bm X_i, \bm X_i', \bm Y_i, \bm Z_i)\}_{1\leq i\leq n}.
    \)

    \item Using the observations from $\mathcal{D}'$, define
    \begin{align}
    \label{eq:estimator}
        \hat\zeta_{n,\sigma}
        = {n\choose 2}^{-1}
        \sum_{i = 1}^n \sum_{j=i+1}^n
        \Big\{
        K(\bm V_i,\bm V_j)
        + K(\bm V_i',\bm V_j')
        - K(\bm V_i,\bm V_j')
        - K(\bm V_i',\bm V_j)
        \Big\},
    \end{align}
    where $\bm V_i = (\bm X_i, \bm Y_i, \bm Z_i)$ and $\bm V_i' = (\bm X_i', \bm Y_i, \bm Z_i)$ for $i=1,\ldots,n$, and $K(\bm x, \bm y)$ denotes the Gaussian kernel
    \(
    K(\bm x, \bm y) = \exp\!\left(-\sigma^2\|\bm x-\bm y\|^2/2\right).
    \)
\end{itemize}
Clearly, $\hat\zeta_{n,\sigma}$ is a U-statistic with core function
\begin{align}
\label{eq:core}
g^*\big((\bm X_1, \bm X_1', \bm Y_1, \bm Z_1),(\bm X_2, \bm X_2', \bm Y_2, \bm Z_2)\big)
= K(\bm V_1, \bm V_2)
+ K(\bm V_1', \bm V_2')
- K(\bm V_1, \bm V_2')
- K(\bm V_1', \bm V_2).
\end{align}
Here $\bm V_i = (\bm X_i,\bm Y_i,\bm Z_i)$ and $\bm V_i'=(\bm X_i',\bm Y_i,\bm Z_i)$ for $i=1,2$. From the theory of U-statistics, it follows that $\hat\zeta_{n,\sigma}$ is an unbiased estimator of $\zeta_{\sigma}(\mathrm P)$. Moreover, the kernel $g$ is bounded irrespective of the dimension of the data. Therefore, by applying the bounded difference inequality, the following concentration inequality for the proposed estimator can be obtained.
\begin{proposition}
    \label{thm:concentration}
    Let $\{(\bm X_i,\bm Y_i, \bm Z_i)\}_{1\leq i\leq n}$ be an i.i.d. sample from the distribution $\mathrm P$, and let $\bm X_i' \sim \mathrm P_{\bm X \mid \bm Z_i}$ for each $i=1,\ldots,n$. Then, for any $\epsilon>0$,
\begin{align*}
\P\Big(\big|\hat\zeta_{n,\sigma} - \zeta_\sigma(\mathrm P)\big|>\epsilon\Big)
\leq
2\exp\Big(-\frac{n\epsilon^2}{32}\Big).
\end{align*}
This bound holds irrespective of the dimensionality of the data.
\end{proposition}
Using the theory of U-statistics \citep[see][]{lee2019u}, the asymptotic distribution of $\hat\zeta_{n,\sigma}$ can also be derived as the sample size tends to infinity for multivariate observations.

\begin{proposition}(Asymptotic null distribution)
    Let $\{(\bm X_i,\bm Y_i, \bm Z_i)\}_{i=1}^n$ be an i.i.d. sample from the distribution $\mathrm P$, and let $\bm X_i' \sim \mathrm P_{\bm X \mid \bm Z_i}$ for each $i=1,\ldots,n$. Let $\{\lambda_k\}_{k\ge 1}$ denote the eigenvalues of the integral equation
\[
\E\left\{
g^*\big((\bm x, \bm x', \bm y, \bm z),(\bm X, \bm X', \bm Y, \bm Z)\big)
\psi(\bm X, \bm X', \bm Y, \bm Z)
\right\}
=
\lambda \, \psi(\bm x, \bm x', \bm y, \bm z),
\]
where $g^*(\cdot,\cdot))$ is the kernel defined in \eqref{eq:core}. Then, under $H_0$, as $n \to \infty$,
\[
n\hat\zeta_{n,\sigma}
\;\xrightarrow{D}\;
\sum_{i=1}^{\infty} \lambda_i \big(U_i^2 - 1\big),
\]
where $\{U_i\}_{i\ge1}$ are i.i.d. $\mathcal N(0,1)$ random variables.
    \label{thm:asymp-null}
\end{proposition}

Under $H_0$, the kernel $g^*(\cdot,\cdot))$ is completely degenerate. Therefore, the limiting null distribution of $\hat\zeta_{n,\sigma}$ is given by an infinite weighted sum of chi-square random variables. In contrast, under $H_1$, the kernel $g^*(\cdot,\cdot))$ is non-degenerate, and consequently $\hat\zeta_{n,\sigma}$ is asymptotically normal. These results are formally stated in the following theorem.

\begin{proposition}(Asymptotic alternate distribution)
    Let $\{(\bm X_i,\bm Y_i, \bm Z_i)\}_{i = 1}^n$ be an i.i.d. sample from the distribution $\mathrm P$, and let $\bm X_i' \sim \mathrm P_{\bm X \mid \bm Z_i}$ for each $i=1,\ldots,n$. Then, under $H_1$, as $n \to \infty$,
\[
\sqrt{n}\big(\hat\zeta_{n,\sigma}-\zeta_{\sigma}(\mathrm P)\big)
\xrightarrow{D} U,
\]
where $U \sim \mathcal N(0,4\sigma_1^2)$ and 
\[
\sigma_1^2 = \var\!\big(g_1^*(\bm X_1, \bm X_1', \bm Y_1, \bm Z_1)\big),
\]
with $g_1^*(\cdot)$ denoting the first-order Hoeffding projection of the kernel $g^*(\cdot,\cdot)$.
    \label{thm:asymp-alt}
\end{proposition}

Note that $\hat\zeta_{n,\sigma}$ is a randomized estimator, as different values may arise across different realizations of the data augmentation step. Nevertheless, Proposition \ref{thm:concentration} shows that $\hat\zeta_{n,\sigma}$ concentrates exponentially fast around its population counterpart $\zeta_\sigma(\mathrm P)$. Consequently, the effect of this randomization diminishes exponentially fast as the sample size increases.

%the estimator may exhibit randomness due to the data augmentation step. But with increasing sample size, the effect of this added randomness diminishes due to averaging in \eqref{eq:estimator}. 

%\begin{remark}
%    Note that the proposed methodology relies on the practitioner's ability to generate $\bm X^\prime$ from the distribution of $\bm X|\bm Z$ and not knowing the conditional distribution in its context. Due to the modern advancement of AI technologies, it is possible to generate observations without the explicit knowledge of $\bm X|\bm Z$. So, in practice, one can use the GenAI methods to generate the $\bm X$ vector given observations on $(\bm Y, \bm Z)$ to check if $\bm Y$ is relevant or not in generating the $\bm X$ observations. 
%\end{remark}

\subsection{Test of Conditional Dependence}
\label{subsec:test}
Note that $\zeta_{\sigma}$ characterizes $H_0$ and $H_1$ when the conditional distribution $\bm X \mid \bm Z$ is known, and it can be consistently estimated by $\hat\zeta_{n,\sigma}$, which possesses several desirable large-sample properties. Consequently, under $H_1$, $\hat\zeta_{n,\sigma}$ tends to take larger values than under $H_0$. This observation suggests rejecting $H_0$ for large values of $\hat\zeta_{n,\sigma}$. 

However, the theoretical quantiles of the asymptotic distribution of $n\hat\zeta_{n,\sigma}$ given in Proposition \ref{thm:asymp-null} are not analytically tractable. One might be tempted to use a permutation procedure to calibrate the test due to its similarity to the MMD test. However, such a permutation approach is not appropriate in the present setting, since the observations $\{(\bm X_i,\bm Y_i,\bm Z_i)\}_{1\leq i\leq n}$ and $\{(\bm X_i', \bm Y_i,\bm Z_i)\}_{1\leq i\leq n}$ are dependent. 

Nevertheless, under $H_0$, the joint distribution satisfies
\[
(\bm X_i, \bm X_i', \bm Y_i,\bm Z_i) \stackrel{D}{=} (\bm X_i', \bm X_i, \bm Y_i,\bm Z_i),
\quad i=1,2,\ldots,n.
\]
Therefore, exploiting this coordinate exchangeability property, the following resampling algorithm is proposed.
%\vspace{0.1in}

%\noindent\fbox{%
%    \parbox{1\textwidth}{%
%\newpage
\paragraph{Resampling algorithm}
\begin{enumerate}
    \item[A.] Given the augmented data $\mathcal{D}^\prime$ compute the test statistic $\hat\zeta_{n,\sigma}$.
    \item[B.] Let $\pi = (\pi(1),\ldots, \pi(n))$ be an element in $\{0,~1\}^n$. Define ${{\bf U}_i} = \pi(i) {\bf X}_i + (1-\pi(i)) {{\bf X}_i^\prime}$ and ${{\bf U}_i^\prime } = (1-\pi(i)) {{\bf X}_i}+\pi(i) {{\bf X}_i^\prime}$  for $i=1,2,\ldots,n$. Use $\{ ({\bf U}_i,{\bf U}_i^\prime, \bm Y_i,\bm Z_i)\}_{1\leq i\leq n}$ to compute $\hat{\zeta}_{n,\sigma}(\pi)$, the resampling analogue of $\hat\zeta_{n,\sigma}$.
    \item[C.] Repeat step B for all possible $\pi$ to get the critical value for a level $\alpha$ ($0<\alpha<1$) test as
    $$\hat c_{1-\alpha} = \inf\Big\{t\in\R: \frac{1}{2^n}\sum_{\pi\in \{0,~1\}^n}\text{I}[\hat\zeta_{n,\sigma}(\pi)\leq t]\geq 1-\alpha\Big\}.$$
\end{enumerate}
%   }
%}

Reject $H_0$ if $\hat\zeta_{n,\sigma} > \hat c_{1-\alpha}$ at nominal level $\alpha$. The corresponding conditional $p$-value is given by
\begin{align}
p_n = \frac{1}{2^n}\sum_{\pi\in\{0,1\}^n} \mathrm{I}\big[\hat\zeta_{n,\sigma}(\pi) \geq \hat\zeta_{n,\sigma}\big].
\label{eq:cond-pval}
\end{align}
Equivalently, $H_0$ can be rejected if $p_n < \alpha$. The resulting procedure defines a valid level $\alpha$ test, as formally stated in the following theorem.
\vspace{0.05in}
\begin{proposition}
    Let $\hat\zeta_{n,\sigma}$ be the estimator of $\zeta_\sigma(\Pr)$ as defined in \eqref{eq:measure} and $p_n$ be the conditional p-value as defined in \eqref{eq:cond-pval}. Then, under $H_0$, we have $\P\{p_n<\alpha\}\leq \alpha$ irrespective of $n$ and $d$.
    \label{prop:level-property}
\end{proposition}

Interestingly, the threshold $\hat c_{1-\alpha}$ can also be bounded by a deterministic sequence that does not depend on the dimension $d$ and converges to zero as $n \to \infty$. This result is established below.

\vspace{0.05in}
\begin{proposition}
Let $\{(\bm X_i,\bm Y_i,\bm Z_i)\}_{1\leq i\leq n}$ be i.i.d. observations from a $d$-dimensional distribution $\mathrm P$. Then, for any $\alpha$ with $0<\alpha<1$, the inequality
\(\
\hat c_{1-\alpha} \le \frac{2}{\alpha (n-1)}
\)
holds with probability one.
\label{prop:cutoff-bound}
\end{proposition}

Therefore, irrespective of the value of $d$, $\hat c_{1-\alpha}$ is of order $O_P(n^{-1})$ and converges to zero almost surely as $n \to \infty$. Under $H_1$, $\hat\zeta_{n,\sigma}$ converges to a positive constant. Consequently, the conditional $p$-value $p_n$ converges to zero as $n \to \infty$, and hence the power of the resulting conditional test converges to one. This result is formally stated in the following proposition.

\vspace{0.05in}
\begin{proposition}
Let $(\bm X_1,\bm Y_1,\bm Z_1),\ldots,(\bm X_n,\bm Y_n,\bm Z_n)$ be i.i.d. observations from a distribution $\mathrm P$ such that $\bm X \not\indpt \bm Y \mid \bm Z$, and suppose that the conditional distribution $\bm X \mid \bm Z$ is known. Then the power of the conditional test based on $p_n$ converges to one as $n \to \infty$.
\label{prop:test-consistency}
\end{proposition}

%\vspace{0.05in}
Although the above resampling algorithm leads to a consistent level $\alpha$ test, computing $\hat c_{1-\alpha}$ exactly can be computationally expensive. Therefore, an approximate $p$-value is proposed using a randomization technique. Generate $\pi_1,\pi_2,\ldots,\pi_B$ independently and uniformly from $\{0,1\}^n$ and compute the randomized $p$-value
\begin{align}
p_{n,B} =
\frac{1}{B+1}
\Bigg(
\sum_{i=1}^B 
\mathrm{I}\{\hat\zeta_{n,\sigma}(\pi_i)\ge \hat\zeta_{n,\sigma}\}
+1
\Bigg).
\end{align}
Then $H_0$ is rejected if $p_{n,B}<\alpha$. The following theorem shows that $p_{n,B}$ provides a close approximation to $p_n$ for large $B$, thereby justifying the use of $p_{n,B}$ for the practical implementation of the proposed test.

\begin{proposition}
    Given the augmented data $\mathcal{D}'$, for any $\epsilon>0$ and $B\in\mathbb{N}$,
\[
\P\!\left(\big|p_{n,B}-p_n\big|>\epsilon+\frac{1}{B+1}\right)
\le
2\exp\{-2B\epsilon^2\}.
\]
Consequently, $|p_{n,B}-p_n|\xrightarrow{\text{a.s.}}0$ as $B\to\infty$.
    \label{thm:monte-carlo-consistency}
\end{proposition}

\begin{remark}
  Note that Proposition \ref{thm:monte-carlo-consistency} ensures that the randomized $p$-value $p_{n,B}$ closely approximates the conditional resampling $p$-value $p_n$ even for moderately large values of $B$. Therefore, in practice, choosing a moderately large $B$ ensures that the test based on $p_{n,B}$ yields essentially the same conclusion as the test based on $p_n$.
\end{remark}

\begin{remark}
{The computational complexity of the proposed test is of order $O(Bn^2)$, which is easily manageable on modern computing systems. In addition, the test statistic is based on pairwise distances between the observed data and their null exchangeable pairs. The pairwise distance matrix needs to be computed only once and can be reused during the resampling step to efficiently compute $p_{n,B}$. This reuse significantly reduces the computational burden compared to the CRT framework. }
\end{remark}

Note that the proposed measure is designed to detect differences in the characteristic functions of the random vector and its null exchangeable pair. This allows it to capture both linear and non-linear conditional dependence between the random vectors. It is worth emphasizing that a related line of work was recently considered by \cite{zhang2025jrssb}. Their contribution is centered on constructing a new conditional independence test by generating a synthetic conditionally independent counterpart $(\Tilde{\bm X},\Tilde{Y},\bm Z)$ through sampling from $P_{\bm X|\bm Z}$ and $P_{\bm Y|\bm Z}$ and then measuring the discrepancy between the resulting distributions using an MMD criterion. In contrast, the primary contribution of the present work is not a new discrepancy measure, but the observation that under the model-X assumption the conditional independence testing problem can be reformulated as a two-sample testing problem through the notion of a null exchangeable pair. This reformulation yields the exchangeability property \((\bm X,\bm X^\prime,\bm Y,\bm Z)\overset{D}{=}(\bm X^\prime,X,Y,Z)\) under the null hypothesis, which is subsequently exploited to develop a resampling procedure that avoids repeatedly generating fresh conditional samples. Consequently, the proposed method can be viewed as a computationally efficient alternative to a direct CRT implementation, whereas \cite{zhang2025jrssb} calibrate their test through a wild bootstrap procedure and do not exploit this exchangeability structure for test calibration. 

We now proceed to study the theoretical properties of the proposed test, including its efficiency against local contiguous alternatives and its performance against high-dimensional alternatives.

\section{Theoretical Analysis}
\label{sec:theory}
Section \ref{subsec:local-asymp} deals with establishing the efficiency of the proposed test against multivariate local contiguous alternatives. The minimax rate optimality of the proposed test is presented in Section \ref{sec:minimax}. In Section \ref{subsec:hd-consistency}, it is established that the proposed test is consistent for high-dimensional alternatives when the dimension of the data is comparable to the sample size.  

\subsection{Efficiency against local alternatives}
\label{subsec:local-asymp}

In this section, firstly, a brief description of the asymptotic framework of interest is provided, then the behavior of the proposed test under this framework is studied under suitable sequences of local contiguous
alternatives. 

Let $\{\mathbb{P}_{n,\theta} : \theta \in \Theta\}$ be a sequence of probability distributions indexed by the sample size $n$, and suppose that the null hypothesis
corresponds to $\theta = 0$. Consider a sequence $\{\theta_n\}$ such that $\mathbb{P}_{n,\theta_n}$ becomes a sequence of probability measures that is contiguous with respect to $\mathbb{P}_{n,0}$. If $\phi_n$ is a sequence of tests at the nominal level
$\alpha \in (0,1)$, i.e., \(\mathbb{E}_{n,0}[\phi_n] \le \alpha\) for all $n\geq 1$,
the quantity of interest is the limiting power of $\phi_n$ under $\mathbb{P}_{n,\theta_n}$, i.e.,
\(\
\lim_{n \to \infty} \mathbb{E}_{n,\theta_n}[\phi_n],
\)
whenever this limit exists. Under contiguity, Le Cam's third lemma ensures that the joint asymptotic behavior
of the test statistic under $\mathbb{P}_{n,0}$ and $\mathbb{P}_{n,\theta_n}$ is
governed by a limiting experiment \citep[see][]{van2000asymptotic, le2000asymptotics}. In classical regular parametric models, this
limiting experiment is Gaussian and the resulting limiting power function admits
a closed-form expression in terms of the Gaussian distribution function. For instance, under the local alternatives $\theta_n = n^{-1/2} \beta$, the power of a test often converges to a
non-degenerate limit
\[
\lim_{n \to \infty} \mathbb{E}_{n,\theta_n}[\phi_n] = \mathcal{L}(\beta),
\]
where $\mathcal{L}(\beta) \in (\alpha,1)$ is a well-defined function of the local parameter
$\beta$. The function $\mathcal{L}(\beta)$ characterizes the asymptotic performance of the test
under contiguous alternatives. For instance, for asymptotically normal tests that rejects the null hypothesis for large values of the test statistic often has
$$\mathcal{L}(\beta)= 1 - \Phi(z_{1-\alpha} - \gamma(\beta)),$$
where $z_{1-\alpha}$ is the $(1-\alpha)$-th quantile of the standard Gaussian distribution and $\gamma(\beta)$ is a non-decreasing function in $\beta$ \citep{lehmann2005testing}. $\gamma(\beta)$ represents the local asymptotic slope of the test. If $\gamma(\beta)$ is a non-zero function, the test is called efficient in the Pitman sense \citep{bhattacharya2019general}.

But more generally, Pitman efficiency compares competing tests through
their limiting power functions under contiguous alternatives. A test
is said to be locally asymptotically efficient if its limiting power
function is non-trivial, that is, if $\mathcal{L}(\beta)>\alpha$ for some
$\beta\neq0$. Many popular tests in multivariate statistics are efficient in the Pitman sense. However, there also exists some tests which are inefficient, i.e., $\gamma(\beta) = 0$ for all $\beta\neq 0$. Some popular contributions in this direction include the graph based two sample tests \cite{bhattacharya2019general}, the Azadkia-Chatterjee's correlation based CRT \cite{shi2022power}.

In the language of Le Cam's theory of statistical
experiments, $\mathcal{L}(\beta)$ corresponds to the power of the induced test
in the limiting experiment \citep[see][]{le2000asymptotics}, and efficiency comparisons may therefore be carried out directly at that level \cite{van2000asymptotic}. Importantly, the absence of an
explicit Gaussian representation for $\mathcal{L}(\beta)$ does not affect the
validity of a Pitman-type efficiency analysis. What is essential is
the existence of a non-degenerate limiting power function under
contiguous alternatives. Two different tests can also be compared using this approach. Between two different tests, the test with a dominant $\mathcal{L}(\beta)$ will be more efficient than the other in the Pitman sense. A test with nontrivial limiting $\mathcal{L}(\beta)$ will be referred to as an efficient test in the Pitman sense. Such an approach has been recently considered in \cite{chatterjee2024boostingmmd}. Motivated by their approach, the performance of the proposed test is evaluated against alternatives of the form $$F_{1-\beta_n n^{-1/2}} = (1-\beta_n n^{-1/2}) G + (\beta_n n^{-1/2}) F,$$ where $\beta_n$ is a sequence of positive numbers converging to $\beta\in[0,\infty)$ as $n\rightarrow \infty$, $\zeta_\sigma(G)=0$, but $\zeta_\sigma(F)>0$. Assume that $F$ and $G$ admit probability density functions $f$ and $g$, respectively. The following result shows that under suitable assumptions on $F$ and $G$, $F_{1-\beta_n n^{-1/2}}$ turns out to be a contiguous and locally asymptotically normal alternative to $G$.

\begin{proposition}
   Let $\{(\bm X_i,\bm Y_i,\bm Z_i)\}_{1\le i\le n}$ be i.i.d. observations from the distribution $G$ with probability density function $g$ such that $\zeta_\sigma(G)=0$. Let
\[
F_{1-\beta_n n^{-1/2}}
=
\big(1-\beta_n n^{-1/2}\big)G + \beta_n n^{-1/2} F,
\]
where $\zeta_\sigma(F)>0$ and $F$ has probability density function $f$. Assume that
\[
\E\!\left[\frac{f(\bm X_1,\bm Y_1,\bm Z_1)}{g(\bm X_1,\bm Y_1,\bm Z_1)}-1\right]^2 < \infty
\]
and $\beta_n \to \beta \in [0,\infty)$ as $n\rightarrow\infty$. Then
\begin{align*}
&\log \prod_{i=1}^n
\left(
1+\frac{\beta_n}{\sqrt n}
\left\{
\frac{f(\bm X_i,\bm Y_i,\bm Z_i)}{g(\bm X_i,\bm Y_i,\bm Z_i)}-1
\right\}
\right) \\
&=
\frac{\beta}{\sqrt n}
\sum_{i=1}^n
\left\{
\frac{f(\bm X_i,\bm Y_i,\bm Z_i)}{g(\bm X_i,\bm Y_i,\bm Z_i)}-1
\right\}
-\frac{\beta^2}{2}
\E\!\left[
\frac{f(\bm X_1,\bm Y_1,\bm Z_1)}{g(\bm X_1,\bm Y_1,\bm Z_1)}-1
\right]^2
+o_P(1).
\end{align*}
    \label{prop:lanfd}
\end{proposition}

Now, to study the asymptotic behaviour of the proposed test under such alternatives, the asymptotic behaviour of $n\hat \zeta_{n,\sigma}$ is evaluated when the observations are drawn from $F_{1-\beta_n n^{-1/2}}$ in the following theorem.

\vspace{0.05in}
\begin{theorem}(Local limit distribution)
Let $\{(\bm X_i,\bm Y_i, \bm Z_i)\}_{1\le i\le n}$ be i.i.d. observations from $F_{1-\beta_n n^{-1/2}}$ satisfying the conditions of Proposition \ref{prop:lanfd}, and let $\bm X_i' \sim \mathrm P_{\bm X \mid \bm Z_i}$ for each $i=1,\ldots,n$. Then, as $n\to\infty$,
\[
n\hat\zeta_{n,\sigma}
\xrightarrow{D}
\sum_{k=1}^{\infty}
\lambda_k
\left(
\big(U_k + \beta\, \E_F[\psi_k(\bm X_1,\bm X_1',\bm Y_1,\bm Z_1)]\big)^2 - 1
\right),
\]
where $\{U_k\}_{k\ge1}$ are i.i.d. standard normal random variables, and $\{\lambda_k\}_{k\ge1}$ and $\{\psi_k\}_{k\ge1}$ are the eigenvalues and eigenfunctions of the integral equation
\[
\E\!\left\{
g^*\big((\bm x, \bm x', \bm y, \bm z),(\bm X, \bm X', \bm Y, \bm Z)\big)
\psi(\bm X, \bm X', \bm Y, \bm Z)
\right\}
=
\lambda \psi(\bm x, \bm x', \bm y, \bm z),
\]
where $g^*(\cdot,\cdot))$ is as in \eqref{eq:core}, $(\bm X,\bm Y,\bm Z)\sim G$, and $(\bm X',\bm Y,\bm Z)$ is the null exchangeable pair of $(\bm X,\bm Y,\bm Z)$.
\label{thm:local-limit}
\end{theorem}

Theorem \ref{thm:local-limit} shows that, under $F_{1-\beta_n n^{-1/2}}$, the limiting distribution of $n\hat\zeta_{n,\sigma}$ is stochastically larger than its null limiting distribution given in Proposition \ref{thm:asymp-null} when $\bm V\sim G$. However, since the test is calibrated using a resampling algorithm, this result alone is not sufficient to establish the efficiency of the proposed test. In particular, the local asymptotic behaviour of the resampled test statistic $n\hat\zeta_{n,\sigma}(\pi)$ must also be analyzed. The following theorem establishes the Pitman efficiency of the proposed test against $F_{1-\beta_n n^{-1/2}}$.
\begin{theorem}
 Let $\bm V_1=(\bm X_1,\bm Y_1,\bm Z_1),\ldots,\bm V_n=(\bm X_n,\bm Y_n,\bm Z_n)\stackrel{i.i.d.}{\sim}F_{1-\beta_n n^{-1/2}}$ satisfying the conditions of Proposition \ref{prop:lanfd}. Then the following statements hold.

\begin{itemize}

\item[(a)] If $\beta_n\to 0$, the power of the resampling test converges to the nominal level $\alpha$.

\item[(b)] If $\beta_n\to\beta\in(0,\infty)$, the power of the resampling test converges to
\[
\mathcal L(\beta)
=
\P\!\left[
\sum_{i=1}^{\infty}
\lambda_i
\Big(
(Z_i+\beta\,\E_F[\psi_i(\bm X_1,\bm X_1',\bm Y_1,\bm Z_1)])^2-1
\Big)
\ge
z_{1-\alpha}
\right],
\]
where $\{Z_i\}_{i\ge1}$ are i.i.d. $\mathcal N(0,1)$ random variables, $\{\lambda_i\}_{i\ge1}$ and $\{\psi_i\}_{i\ge1}$ are as defined in Theorem \ref{thm:local-limit}, and $z_{1-\alpha}$ is the $(1-\alpha)$ quantile of the distribution of $\sum_{i=1}^{\infty}\lambda_i(Z_i^2-1)$.

\item[(c)] If $\beta_n\to\infty$, the power of the resampling test converges to one.

\end{itemize}
 \label{thm:local-limit-per}
\end{theorem}

Theorem \ref{thm:local-limit-per} shows that as $\beta$ increases from zero to infinity, the power of the proposed test increases from $\alpha$ to one. Hence, the proposed test is efficient in the Pitman sense. However, the exact expression of the limiting power function $\mathcal{L}(\beta)$ is not analytically tractable, as discussed earlier. This result establishes the Pitman efficiency of the test against local contiguous alternatives of the form $F_{1-\beta_n n^{-1/2}}$. Similar results can also be obtained for the alternatives considered in \cite{shi2024}. The particular choice of $F_{1-\beta_n n^{-1/2}}$ is motivated by its analytical convenience and ease of sample generation.

\begin{figure}[h!]
\centering
    
\begin{tikzpicture}[scale = 1.2]
\begin{axis}[xmin = 100, xmax = 500, ymin = 0, ymax = 1, xlabel = {Sample Size}, ylabel = {Power Estimates}, title = {\bf $\mathcal{L}(\beta)$},
    grid=major,
    major grid style={gray!30},
    tick style={black},
    axis line style={black},
    mark size=1.2pt
]
\addplot[color = red,   mark = *, step = 1cm,very thin]coordinates{(100, 0.043)(200,0.042)(300,0.046)(400,0.054)(500,0.061)}; 

\addplot[color = blue,   mark = diamond*, step = 1cm,very thin]coordinates{(100,0.071)(200,0.077)(300,0.064)(400,0.071)(500,0.076)};

\addplot[color = purple, mark = square*, step = 1cm, very thin]coordinates{(100,0.141)(200,0.137)(300,0.132)(400,0.128)(500,0.134)};

\addplot[color = magenta,   mark = *, step = 1cm,very thin]coordinates{(100,0.310)(200,0.300)(300,0.283)(400,0.312)(500,0.304)};

\addplot[color = ForestGreen,   mark = diamond*, step = 1cm,very thin]coordinates{(100,0.647)(200,0.598)(300,0.559)(400,0.575)(500,0.584)};

\end{axis}
\end{tikzpicture}
\caption{%Power of the proposed test for different values of $\beta$ ( $1$(\textcolor{red}{$\tikzcircle{2pt}$}), $3$ (\textcolor{blue}{$\blacklozenge$}), $5$ (\textcolor{purple}{$\blacksquare$}), $7$ (\textcolor{magenta}{$\tikzcircle{2pt}$}), $9$ (\textcolor{ForestGreen}{$\blacklozenge$})), and $n$ ($100,\ldots, 500$) when the samples are generated from a mixture of (a) $X=Z+\eta_1$, $Y = Z+\eta_2$, $Z,\eta_1,\eta_2\stackrel{i.i.d.}{\sim} N(0,1)$ and (b) $X=Z+\eta_1$, $Y = Z+X+\eta_2$, $Z,\eta_1,\eta_2\stackrel{i.i.d.}{\sim} N(0,1)$ with mixing proportion $(1-\beta/\sqrt{n})$ and $\beta/\sqrt{n}$, respectively. %This clearly exhibits that the proposed test is Pitman efficient in accordance with Theorem \ref{thm:local-limit-per}.
Power of the proposed test for different values of $\beta$ ($1$ (\textcolor{red}{$\tikzcircle{2pt}$}), $3$ (\textcolor{blue}{$\blacklozenge$}), $5$ (\textcolor{purple}{$\blacksquare$}), $7$ (\textcolor{magenta}{$\tikzcircle{2pt}$}), and $9$ (\textcolor{ForestGreen}{$\blacklozenge$})) and sample sizes $n=100,\ldots,500$. The samples are generated from a mixture of two distributions: (a) $X=Z+\eta_1$, $Y=Z+\eta_2$, and (b) $X=Z+\eta_1$, $Y=Z+X+\eta_2$, where $Z,\eta_1,\eta_2\stackrel{i.i.d.}{\sim}\mathcal{N}(0,1)$. The mixing proportions are $(1-\beta/\sqrt{n})$ and $\beta/\sqrt{n}$, respectively.
}
    \label{fig:Pitman}
\end{figure}

%Consider an example where $\bm V_1 = (X_1, Y_1, Z_1),\ldots, \bm V_n = (X_n,Y_n,Z_n)$ are are randomly sampled from the mixture of distributions where observations are sampled using (a) $X=Z+\eta_1$, $Y = Z+\eta_2$, $Z,\eta_1,\eta_2\stackrel{i.i.d.}{\sim} N(0,1)$ and (b) $X=Z+\eta_1$, $Y = Z+X+\eta_2$, $Z,\eta_1,\eta_2\stackrel{i.i.d.}{\sim} N(0,1)$ with mixing proportion $(1-\beta/\sqrt{n})$ and $\beta/\sqrt{n}$, respectively. The power of the test is calculated based on the proportion of times the test rejects $H_0$ as in \eqref{eq:fundamental-problem} on $1000$ replications of the experiment. Our results are reported in Figure \ref{fig:Pitman} for different values of $\beta$ and sample size $n$. Clearly, the power of our test converges to a nontrivial limit with increasing sample size as presented in Theorem \ref{thm:local-limit-per}, confirming the efficiency of our test. Several tests available in the literature for testing conditional dependence does not have the efficiency property, which makes the proposed test particularly desirable. 

Consider an example where $\bm V_1=(X_1,Y_1,Z_1),\ldots,\bm V_n=(X_n,Y_n,Z_n)$ are randomly sampled from a mixture of two distributions with mixing proportions $(1-\beta/\sqrt{n})$ and $\beta/\sqrt{n}$. Under the first component,
\(
X = Z+\eta_1\) and \( Y = Z+\eta_2,
\)
where $Z,\eta_1,\eta_2 \stackrel{i.i.d.}{\sim} \mathcal N(0,1)$. Under the second component,
\(
X = Z+\eta_1\) and \( Y = Z+X+\eta_2,
\)
with $Z,\eta_1,\eta_2 \stackrel{i.i.d.}{\sim} \mathcal N(0,1)$.

The power of the test is estimated as the proportion of times the null hypothesis in \eqref{eq:fundamental-problem} is rejected over $1000$ replications of the experiment. The results are reported in Figure \ref{fig:Pitman} for different values of $\beta$ and sample sizes $n$. As predicted by Theorem \ref{thm:local-limit-per}, the power converges to a nontrivial limit as the sample size increases, confirming the Pitman efficiency of the proposed test. Several existing tests for conditional independence do not possess this efficiency property, which highlights the advantage of the proposed method.

\subsection{Minimax rate optimality}
\label{sec:minimax}

Let us consider a testing problem involving a pair of hypotheses 
$$H_0:\zeta_\sigma(\Pr)=0\text{ and }H_1^\prime: \zeta_\sigma(\Pr)>\epsilon(n),$$ 
where $\epsilon(n)$ is a positive number that depends on the sample size $n$. Let $\mathcal{F}(\epsilon(n)):=\{\Pr\mid \zeta_\sigma(\Pr)>\epsilon(n)\}$ be the class of alternatives under $H_1^\prime$ and $\mathbb{T}_{n}(\alpha)$ be the class of all nonrandomized level $\alpha$ test. The minimax type II error rate for this problem is defined as
$$R_{n}\big(\epsilon(n)\big) = \inf_{\phi\in\mathbb{T}_{n}(\alpha)}\sup_{F\in\mathcal{F}(\epsilon(n))} \P_{F^n}\{\phi = 0\},$$
where $\P_{F^n}$ denotes the probability corresponding to the joint distribution of the augmented data $\{(\bm X_i,\bm X_i^\prime, \bm Y_i, \bm Z_i)\}_{i=1}^n$. Here we want to find an optimum choice of $\epsilon(n)$ (call it $\epsilon_{0}(n)$) such that the following two conditions hold.
\begin{enumerate}
    \item[(a)] \label{cond_a} 
    For any $0<\beta<1-\alpha$, there exists a constant $c(\alpha,\beta)>0$ such that for all $0 < c < c(\alpha,\beta)$, we have 
    $$\liminf\limits_{n \rightarrow \infty} R_{n}(c~\epsilon_0({n})) \geq \beta.$$
    
    \item[(b)] \label{cond_b} There exists a level $\alpha$ test $\phi_0$ such that for any $0<\beta<1-\alpha$, there exists $C(\alpha,\beta) > 0$ for which $$\limsup\limits_{n \rightarrow \infty} \sup\limits_{F\in\mathcal{F}(c~\epsilon_0({n}))}\P_{F^n}\{\phi_0=0\}\leq \beta$$ for all $c>C(\alpha,\beta)$, or in other words, $$\limsup\limits_{n \rightarrow \infty} R_{n}(c~\epsilon_0({n})) \leq \beta$$ for all $c > C(\alpha,\beta)$.
  
\end{enumerate}

This optimal rate $\epsilon_0({n})$ is called the minimax rate of separation for the above problem, and the test $\phi_0$ is called the minimax rate optimal test. Theorem \ref{thm:minimax-lower-bound} below shows that here $\epsilon_0(n)$ cannot be of order smaller than $O(n^{-1})$. So, for any $0<\beta<1-\alpha$ and any $\phi\in \mathbb{T}_n(\alpha)$, we can always find a distribution $F$ with $\zeta(F)$ of the order $O(n^{-1})$ or smaller such that the type II error rate of the test $\phi$, i.e., $\P_{F^n}\{\phi=0\}$ is larger than $\beta$.

\begin{theorem}
For $0<\beta<1-\alpha$, there exists a constant $c_0(\alpha,\beta)$ such that the minimax type II error rate $R_{n}(cn^{-1})$ is lower bounded by $\beta$ for all $n$ and all $c\in(0,c_0(\alpha, \beta))$.
\label{thm:minimax-lower-bound}
\end{theorem}

\begin{remark}
    In particular, consider the distribution $F_{\delta_n} = (1-\delta_n)F+\delta_n G$ where $\zeta(G)=0$ and $\zeta(F)>0$. Using Lemma \ref{lem:relation-contamination} from the supplementary materials we get $\zeta(F_{\delta_n})=(1-\delta_n)^2 \zeta(F)$. 
    Hence, if $\delta_n$ is such that $n(1-\delta_n)^2\rightarrow 0$ as $n \rightarrow \infty$ (i.e. $\zeta(F_{\delta_n})$ is of smaller asymptotic order than $O(n^{-1})$), then the power of any level $\alpha$ test will fall below the nominal level $\alpha$. 
    \label{remark4}
\end{remark}

In the next theorem, we establish that in the case of $\epsilon_0(n)=n^{-1}$, our test based on ${\hat \zeta}_n$ satisfies the condition (b) stated above. Therefore, these two theorems (Theorem \ref{thm:minimax-lower-bound} and \ref{thm:minimax-upper-bound}) together show that the minimax rate of separation is $\epsilon_0({n})=n^{-1}$, and our proposed test has the minimax rate optimality for the class of alternatives $\mathcal{F}(\epsilon(n))$. 

\begin{theorem}
For any $\beta \in (0,1-\alpha)$, there exists a constant $C_0(\alpha,\beta)$ (independent of $d$) such that asymptotically the maximum type II error of the test based on $\hat\zeta_{n,\sigma}$ over $\mathcal{F}(cn^{-1})$  is uniformly bounded above by $\beta$ for all $c>C_0(\alpha,\beta)$, i.e.,
$$ \limsup\limits_{n \rightarrow \infty}\sup_{F\in \mathcal{F}(c\lambda({n}))}\P_{F^n}(\hat\zeta_{n,\sigma}\leq c_{1-\alpha})\leq \beta$$ for all $c>C_0(\alpha,\beta).$
\label{thm:minimax-upper-bound}
\end{theorem}

{ \begin{remark}
    Consider the same example as in Remark \ref{remark4}. Since $\zeta(F_{\delta_n}) = (1-\delta_n)^2 \zeta(F)$, from the proof of Theorem \ref{thm:minimax-upper-bound}, it can be shown that if $n(1-\delta_n)^2 \rightarrow \infty$ as $n \rightarrow \infty$
    (i.e. $\zeta(F_{\delta_n})$ is of higher asymptotic order than $O(n^{-1})$), the power of our test converges to one. 
\end{remark}}

The proposed conditional independence test may be implemented within the conditional randomization test (CRT) framework by combining the discrepancy measure \(\zeta_\sigma(\mathrm P)\) with repeated conditional resampling from \(P_{\bm X|\bm Z}\). Such a procedure is viewed as the natural CRT implementation of the proposed discrepancy measure. In contrast, a single augmented sample is reused by the resampling procedure developed in Section \ref{sec:methodology}, and the exchangeability property \( (\bm X,\bm X^\prime,\bm Y,\bm Z)\overset{D}{=}(\bm X^\prime,\bm X,\bm Y,\bm Z)\) under \(H_0\) is exploited, thereby avoiding the repeated generation of new conditional samples.

At first sight, it may be expected that this computational simplification would come at the expense of statistical efficiency. It is demonstrated by the minimax analysis that this is not the case. In Theorem \ref{thm:minimax-lower-bound}, it is established that no level-\(\alpha\) test can reliably distinguish alternatives satisfying \(\zeta_\sigma(\mathrm P)=o(n^{-1})\), whereas in Theorem \ref{thm:minimax-upper-bound}, it is shown that alternatives of order \(n^{-1}\) are consistently detected by the proposed exchangeability-based procedure. Consequently, the minimax separation rate associated with the discrepancy measure \(\zeta_\sigma(\mathrm P)\) is attained by the proposed test.

An important implication is carried by this observation. Since reliance on the same discrepancy measure \(\zeta_\sigma(\mathrm P)\) is made by a CRT-based implementation, a fundamentally smaller minimax detection boundary cannot be achieved. Therefore, it is suggested by the minimax analysis that the limiting factor is the discrepancy measure itself rather than the calibration mechanism used to construct the test. In particular, although finite-sample performance may differ between the two approaches, no loss of statistical efficiency occurs at the level of minimax rate optimality when the CRT calibration step is replaced by the proposed exchangeability-based resampling scheme.

The distinction between the two procedures is therefore primarily computational rather than statistical. Repeated conditional resampling and repeated evaluations of the test statistic are avoided by the proposed approach, thereby substantially reducing the computational burden while retaining the same asymptotic detection rate. Consequently, a computationally attractive alternative to the CRT implementation is provided by the exchangeability-based procedure, particularly in large-scale applications where computational cost is a primary concern.

\subsection{High-dimensional consistency}
\label{subsec:hd-consistency}

Recent advances in computational technology have enabled researchers across various fields to collect and analyze large-scale datasets. These datasets often contain observations on a large number of features, which may be comparable to or even exceed the sample size. Such settings are commonly referred to as high-dimensional data regimes. The emergence of high-dimensional datasets has motivated the development of statistical procedures that remain scalable with increasing dimensionality. A widely used approach in the literature is to construct algorithms based on pairwise distances. Since the proposed test statistic is also based on pairwise distances, it naturally extends to such settings. However, to understand its behaviour in high-dimensional regimes, it is necessary to study how the dimensionality of the data affects the fundamental properties of the test statistic. The following result provides insight into this direction.

\begin{proposition}
    Let $\{(\bm X_i,\bm Y_i,\bm Z_i)\}_{1\le i\le n}$ be i.i.d. observations from a $d$-dimensional distribution $\mathrm P$. Then, irrespective of the dimension $d$,
\begin{align*}
\E[\hat\zeta_{n,\sigma}] &= \zeta_\sigma(\mathrm P), \hspace{0.1in}\text{and}\hspace{0.1in}
\var[\hat\zeta_{n,\sigma}] \le {n\choose 2}^{-1}\Big(4(n-1)\zeta_\sigma(\mathrm P) + 4\Big).
\end{align*}
    \label{prop:baic-prop}
\end{proposition}
Note that the effect of dimensionality on the proposed estimator primarily arises through the dependence of $\zeta_\sigma(\mathrm P)$ on the dimension $d$. In high-dimensional settings, if $\zeta_\sigma(\mathrm P)$ remains bounded away from zero (that is, the measure continues to detect conditional dependence between the random variables), then the estimator will still concentrate around $\zeta_\sigma(\mathrm P)$ provided that the sample size is sufficiently large. A stronger result can be established by characterizing the rate at which $\zeta_\sigma(\mathrm P)$ may converge to zero. This is formally stated in the following theorem.

\begin{proposition}
  Let $\{(\bm X_i,\bm Y_i,\bm Z_i)\}_{1\le i\le n}$ be i.i.d. observations from a $d$-dimensional distribution $\mathrm P$, where the dimension $d=d(n)$ is allowed to grow with the sample size $n$. If $n\,\zeta_\sigma(\mathrm P)\to\infty$ as $n\to\infty$, then the power of the proposed test converges to one as $n,d\to\infty$.
    \label{thm:high-dim-consistency}
\end{proposition}

{\begin{remark}
Let \((\bm X, \bm Y, \bm Z) \sim \mathrm{P} = \mathcal{N}(\bm \mu,\bm \Sigma_1)\). Then, it turns out that $(\bm X^\prime, \bm Y, \bm Z)\sim \mathcal{N}(\bm \mu,\bm \Sigma_2)$ with $\bm \Sigma_1 = \bm \Sigma_2 + \bm \Delta$, where 
$$\Delta = \begin{pmatrix}
    0 & \bm \Sigma_{XY|Z} & 0\\
    \bm \Sigma_{YX|Z} & 0 & 0 \\
    0 & 0 & 0
\end{pmatrix},$$
$\bm \Sigma_{XY|Z}$ being the partial cross-covariance matrix between $\bm X$ and $\bm Y$ given $\bm Z$, defined as 
$$\bm \Sigma_{XY|Z} = \E\left[(\bm X - \E[\bm X|\bm Z])(\bm Y-\E[\bm Y|\bm Z])^\top\right].$$
Then, $\zeta_\sigma(\mathrm{P})$ depends only on the variance-covariance matrices and is given by
\begin{align*}
    \zeta_{\sigma}(\mathrm{P})
& =
\frac{1}{\sqrt{\det(I+2\sigma^2\bm \Sigma_1)}}
+
\frac{1}{\sqrt{\det(I+2\sigma^2\bm \Sigma_2)}}
-
\frac{2}{\sqrt{\det(I+\sigma^2(\bm \Sigma_1+\bm \Sigma_2))}}\\
& = \frac{1}{\sqrt{\det(I+2 \sigma^2\bm \Sigma_2+2\sigma^2 \bm \Delta)}}
+
\frac{1}{\sqrt{\det(I+2\sigma^2\bm \Sigma_2)}}
-
\frac{2}{\sqrt{\det(I+2\sigma^2\bm \Sigma_2+\sigma^2 \bm\Delta))}}
\end{align*}
Therefore, for Gaussian distributions, the magnitude of $\zeta_{\sigma}(\mathrm P)$ depends only on $\bm \Sigma_{XY\mid Z}$. In particular, $\zeta_{\sigma}(\mathrm P)=0$ if $\bm \Sigma_{XY\mid Z}=\bm 0$, and it is positive otherwise. Under suitable regularity conditions it can also be established that when $\|\Sigma_{XY|Z}\|_F^2\rightarrow 0$, $\zeta_\sigma(\mathrm P) \gtrsim \|\Sigma_{XY|Z}\|_F^2$ and hence the high-dimensional consistency holds when $n\|\Sigma_{XY|Z}\|_F^2$ diverges to infinity as $n$ and $d$ diverges to infinity simultaneously. Here $\|\cdot\|_F$ denotes the Frobenius norm. Interested readers are referred to Section A of the supplementary material for a more detailed discussion.
\end{remark}
}

\begin{remark}
 In practice, if $\liminf_{d\to\infty}\zeta_\sigma(\mathrm P)>0$, the power of the proposed test converges to one even when the sample size increases at a very slow rate. On the other hand, if $\liminf_{d\to\infty}\zeta_\sigma(\mathrm P)=0$, the sample size must increase appropriately to ensure satisfactory performance of the test, which can be controlled by the practitioner. Otherwise, the power of the test may fall below the level $\alpha$. To the best of our knowledge, results of this type have not been explored previously in the literature on conditional independence testing. 
\end{remark}

 {
\section{Empirical Analysis}

\label{sec:simulation}
In this section, the empirical performance of the proposed test is evaluated using synthetic experiments where it is assumed that practitioner has access to the conditional distribution $\bm X \mid \bm Z$ and can generate new observations from that distribution. 

\subsection{Synthetic data analysis}
\label{sec:synthetic-data-analysis}
%The empirical performance of the test is evaluated against the test studied in \cite{shi2022power} based on Azadkia-Chatterjee's partial correlation coefficient, the conditional distance correlation test proposed in \cite{wang2015} under the CRT framework, the Generalized Covariance Measure proposed in \citep{shah2020hardness}, and the weighted Generalized Covariance Measure as in \citep{Scheidegger2022}. These tests are referred to as the AC test, the cDC.CRT test, the GCM test, and the wGCM test, respectively. There are multiple ways to carry out the GCM test depending upon the method used to estimate the regression function. Here, the generalized additive model is used for that purpose. \citet{Scheidegger2022} proposed two different tests using the wGCM where one can either choose a weight function apriori and use it for the entire data or estimate the optimal weight function based on a training data set and apply the test on the validation data set. Here, also the generalized additive model is used to estimate the regression functions and the corresponding tests are referred to as wGCM.fix test and wGCM.est test, respectively. For the proposed test two different results are presented, one based on the proposed resampling algorithm and another based on the CRT framework. These tests are referred to as the AUG test and the AUG.CRT test, respectively. 

The empirical performance of the proposed test is compared with several existing methods in the literature, including the test based on Azadkia–Chatterjee's partial correlation coefficient studied in \cite{shi2022power}, the conditional distance correlation test proposed in \cite{wang2015} under the CRT framework, the Generalized Covariance Measure (GCM) test of \cite{shah2020hardness}, and the weighted Generalized Covariance Measure (wGCM) test of \cite{Scheidegger2022}. These methods are referred to as the AC test, the cDC.CRT test, the GCM test, and the wGCM test, respectively.

The GCM test can be implemented using different approaches for estimating the regression functions. In the present study, generalized additive models are used for this purpose. The wGCM framework proposed in \cite{Scheidegger2022} provides two variants: one that uses a fixed weight function specified a priori, and another that estimates an optimal weight function using a training dataset and applies the test on a validation dataset. In both cases, generalized additive models are used to estimate the regression functions. These variants are referred to as the wGCM.fix test and the wGCM.est test, respectively.

For the proposed method, two versions are considered: one based on the proposed resampling algorithm and another based on the CRT framework. These are referred to as the AUG test and the AUG.CRT test, respectively. While the theoretical results for the test hold for any fixed $\sigma^2>0$, practical performance depends on its suitable choice. Here it is recommended to use $\sigma^2 = \frac{1}{d}$, where $d$ is the dimension of the vector $(\bm X,\bm Y,\bm Z)$, which balances sensitivity to local and global dependence structures. Alternatively, $\sigma^2$ may be selected via cross-validation on a held-out portion of the data or via the median heuristic approach \citep{gretton2009fast}, though both approaches incur additional computational cost. A systematic study of the optimal choice of $\sigma^2$ is beyond the scope of this article and is left for future work.

Throughout the remainder of the article, all tests are conducted at the $5\%$ nominal level. For the resampling-based procedures, the critical values are computed using $500$ resampling iterations. Each experiment is repeated $1000$ times, and the empirical power is estimated as the proportion of replications in which $H_0$ is rejected. First, consider an example involving unidimensional observations in a simple regression framework.
\begin{example}
\label{exa:ex1}
Let $X$, $Y$, and $Z$ be univariate random variables such that $Z\sim F$, $X=Z+\varepsilon_1$, and $Y=rX+Z+\varepsilon_2$, where $\varepsilon_1,\varepsilon_2\stackrel{i.i.d.}{\sim}F$ are independent of $Z$, $r\in\mathbb{R}$, and $F$ is chosen to be (a) the standard normal distribution, (b) the standard Student's $t$-distribution with $4$ degrees of freedom, or (c) the standard Cauchy distribution. It is easy to verify that $X\indpt Y\mid Z$ when $r=0$, whereas $X\not\indpt Y\mid Z$ otherwise. Moreover, the strength of the conditional dependence between $X$ and $Y$ given $Z$ increases as $|r|$ increases.

The power of the tests is evaluated for different values of $r$ with sample size fixed at $n=50$. The results are reported in Figure \ref{fig:power-1}. It can be observed that several competing tests are sensitive to the presence of outliers in the data. Note that Examples \ref{exa:ex1}(a)–\ref{exa:ex1}(c) represent distributions with progressively heavier tails. This effect is reflected in the power of the competing tests.
\end{example}

%Note that the AUG test exhibits a significantly better performance that the AC and HDS tests for all the examples. 

\begin{figure}[h!]
\centering
    
\begin{tikzpicture}[scale = 1.05]
\begin{axis}[xmin = -2, xmax = 2, ymin = 0, ymax = 1, xlabel = {$r$}, ylabel = {Estimates}, title = {\bf Example 1 (a)},
 grid=major,
    major grid style={gray!30},
    tick style={black},
    axis line style={black},
    mark size=1.2pt
    ]
\addplot[color = red,   mark = *, step = 1cm,very thin]coordinates{(-2,0.983)(-1.5,0.865)(-1.2,0.636)(-0.9,0.334)(-0.6,0.125)(-0.3,0.056)(0,0.041)(0.3,0.060)(0.6,0.118)(0.9,0.257)(1.2,0.478)(1.5,0.695)(2,0.884)};

\addplot[color = blue,   mark = diamond*, step = 1cm,very thin]coordinates{(-2,0.930)(-1.5,0.826)(-1.2, 0.666)(-.9, 0.433)(-.6, 0.165)(-.3, 0.067)(0, 0.036)(.3, 0.071)(.6, 0.166)(.9,  0.375)(1.2,  0.584)(1.5,  0.716)(2,  0.850)};

\addplot[color = purple, mark = square*, step = 1cm, very thin]coordinates{(-2,1.000)(-1.5,1.000)(-1.2, 1.000)(-.9, 0.955)(-.6, 0.563)(-.3, 0.045)(0, 0.058)(.3, 0.472)(.6, 0.913)(.9,  0.995)(1.2,  0.999)(1.5,  1.000)(2,  1.000)};

\addplot[color = magenta, mark = diamond*, step = 1cm,very thin]coordinates{(-2,0.994) (-1.5,0.940) (-1.2,0.782) (-.9,0.508) (-.6,0.212) (-.3,0.064) (0,0.040) (.3,0.091) (.6,0.254)  (.9,0.536)  (1.2,0.799)  (1.5,0.937)  (2,0.996)};

\addplot[color = ForestGreen, mark = x, step = 1cm,very thin]coordinates{(-2,1.000) (-1.5,1.000) (-1.2,1.000) (-.9,0.999) (-.6,0.961) (-.3,0.527) (0,0.064) (0.3,0.489) (0.6,0.953)  (0.9,0.998)  (1.2,1.000)  (1.5,1.000)  (2,1.000)};

\addplot[color = black, mark = star, step = 1cm,very thin]coordinates{(-2,1.000) (-1.5,1.000) (-1.2,1.000) (-.9,0.988) (-.6,0.863) (-.3,0.354) (0,0.062) (0.3,0.315) (.6,0.853)  (.9,0.991)  (1.2,0.998)  (1.5,0.999)  (2,1.000)};

\addplot[color = blue, mark = triangle*, step = 1cm,very thin]coordinates{(-2,0.992) (-1.5,0.985) (-1.2,0.972) (-.9,0.923) (-.6,0.739) (-.3,0.306) (0,0.062) (0.3,0.293) (.6,0.695)  (.9,0.921)  (1.2,0.970)  (1.5,0.983)  (2,0.985)};

\end{axis}
\end{tikzpicture}
\begin{tikzpicture}[scale = 1.05]
\begin{axis}[xmin = -2, xmax = 2, ymin = 0, ymax = 1, xlabel = {$r$}, ylabel = {Estimates}, title = {\bf Example 1 (b)},
 grid=major,
    major grid style={gray!30},
    tick style={black},
    axis line style={black},
    mark size=1.2pt
    ]
\addplot[color = red,   mark = *, step = 1cm,very thin]coordinates{(-2,0.961) (-1.5,0.841) (-1.2,0.612) (-.9,0.321) (-.6,0.138) (-.3,0.063) (0,0.041) (0.3,0.060) (.6,0.139)  (.9,0.333)  (1.2,0.566)  (1.5,0.770)  (2,0.916)};

\addplot[color = blue,   mark = diamond*, step = 1cm,very thin]coordinates{(-2,0.899) (-1.5,0.778) (-1.2,0.610) (-.9,0.387) (-.6,0.173) (-.3,0.080) (0,0.048) (.3,0.078) (.6,0.208)  (.9,0.446)  (1.2,0.626)  (1.5,0.750) (2,0.854)};

\addplot[color = purple, mark = square*, step = 1cm, very thin]coordinates{(-2,1.000) (-1.5,1.000) (-1.2,0.992) (-.9,0.940) (-.6,0.652) (-.3,0.130) (0,0.061) (.3,0.504) (.6,0.907)  (.9,0.981)  (1.2,0.998)  (1.5,1.000)  (2,1.000)};

\addplot[color = magenta, mark = diamond*, step = 1cm,very thin]coordinates{(-2,0.986) (-1.5,0.931) (-1.2,0.793) (-.9,0.539) (-.6,0.233) (-.3,0.076) (0,0.045) (.3,0.126) (.6,0.323)  (.9,0.589)  (1.2,0.822)  (1.5,0.932)  (2,0.985)};

\addplot[color = ForestGreen, mark = x, step = 1cm,very thin]coordinates{(-2,0.957) (-1.5,0.957) (-1.2,0.955) (-.9,0.948) (-.6,0.859) (-.3,0.470) (0,0.067) (.3,0.461) (.6,0.843)  (.9,0.929)  (1.2,0.955)  (1.5,0.959)  (2,0.962)};

\addplot[color = black, mark = star, step = 1cm,very thin]coordinates{(-2,0.910) (-1.5,0.899) (-1.2,0.884) (-.9,0.827) (-.6,0.674) (-.3,0.272) (0,0.060) (.3,0.275) (.6,0.657)  (.9,0.833)  (1.2,0.878)  (1.5,0.893)  (2,0.902)};

\addplot[color = blue, mark = triangle*, step = 1cm,very thin]coordinates{(-2,0.924) (-1.5,0.911) (-1.2,0.889) (-.9,0.804) (-.6,0.637) (-.3,0.294) (0,0.078) (.3,0.285) (.6,0.660)  (.9,0.818)  (1.2,0.890)  (1.5,0.909)  (2,0.934)};

\end{axis}
\end{tikzpicture}
\begin{tikzpicture}[scale = 1.05]
\begin{axis}[xmin = -2, xmax = 2, ymin = 0, ymax = 1, xlabel = {$r$}, ylabel = {Estimates}, title = {\bf Example 1 (c)},
 grid=major,
    major grid style={gray!30},
    tick style={black},
    axis line style={black},
    mark size=1.2pt
    ]
\addplot[color = red,   mark = *, step = 1cm,very thin]coordinates{(-2,0.868) (-1.5,0.731) (-1.2,0.560) (-.9,0.343) (-.6,0.162) (-.3,0.073) (0,0.048) (0.3,0.087) (0.6,0.230) (0.9, 0.449)  (1.2,0.631)  (1.5,0.736)  (2,0.864)};

\addplot[color = blue,   mark = diamond*, step = 1cm,very thin]coordinates{(-2,0.785) (-1.5,0.684) (-1.2,0.565) (-.9,0.404) (-.6,0.204) (-.3,0.073) (0,0.041) (0.3,0.105) (0.6,0.293)  (.9,0.484)  (1.2,0.616)  (1.5,0.693)  (2,0.754)};

\addplot[color = purple, mark = square*, step = 1cm, very thin]coordinates{(-2,0.533) (-1.5,0.494) (-1.2,0.456) (-.9,0.416) (-.6,0.356) (-.3,0.279) (0,0.044) (.3,0.280) (.6,0.355)  (.9,0.427)  (1.2,0.461)  (1.5,0.488)  (2,0.535)};

\addplot[color = magenta, mark = diamond*, step = 1cm,very thin]coordinates{(-2,0.937) (-1.5,0.874) (-1.2,0.777) (-.9,0.534) (-.6,0.331) (-.3,0.154) (0,0.051) (.3,0.308) (.6,0.575)  (.9,0.734)  (1.2,0.826)  (1.5,0.892)  (2,0.949)};

\addplot[color = ForestGreen, mark = x, step = 1cm,very thin]coordinates{(-2,0.264) (-1.5,0.252) (-1.2,0.239) (-.9,0.217) (-.6,0.183) (-.3,0.123) (0,0.016) (.3,0.131) (.6,0.207)  (.9,0.239)  (1.2,0.255)  (1.5,0.267)  (2,0.271)};

\addplot[color = black, mark = star, step = 1cm,very thin]coordinates{(-2,0.145) (-1.5,0.135) (-1.2,0.119) (-.9,0.094) (-0.6,0.072) (-.3,0.043) (0,0.006) (.3,0.036) (.6,0.078)  (.9,0.101)  (1.2,0.113)  (1.5,0.124)  (2,0.150)};

\addplot[color = blue, mark = triangle*, step = 1cm,very thin]coordinates{(-2,0.389) (-1.5,0.369) (-1.2,0.348) (-.9,0.320) (-.6,0.272) (-.3,0.167) (0,0.031) (.3,0.153) (0.6,0.253)  (.9,0.321)  (1.2,0.350)  (1.5,0.370)  (2,0.402)};

\end{axis}
\end{tikzpicture}
\caption{Powers of the AUG test (\textcolor{red}{$\tikzcircle{2pt}$}), the AUG.CRT test (\textcolor{blue}{$\blacklozenge$}), the cDC.CRT test (\textcolor{purple}{$\blacksquare$}), the AC test (\textcolor{magenta}{$\blacklozenge$}), the GCM test (\textcolor{ForestGreen}{$\times$}), the GCM.fix test (\textcolor{blue}{$\blacktriangle$}) and the wGCM.est test (\textcolor{black}{$\star$}) in Examples 1 (a)-(c).}
    \label{fig:power-1}
\end{figure}

\begin{example}\label{exa:ex2}
   Here, a bivariate noise distributions with geometric dependence structures is considered. Specifically, suppose $Z\sim \text{Unif}(0,1)$, $X = Z + \eta_1$, and $Y = Z + \eta_2$, where $(\eta_1,\eta_2)$ is generated from an equal mixture of two bivariate normal distributions with mean zero and variance–covariance matrices $\bm \Sigma_1$ and $\bm \Sigma_2$, respectively. The power of the tests is computed for the following choices:
\[
\text{(a)}\;
\bm \Sigma_1 =
\begin{pmatrix}
1 & 0\\
0 & 1
\end{pmatrix},
\quad
\bm \Sigma_2 =
\begin{pmatrix}
10 & 0\\
0 & 10
\end{pmatrix},
\qquad
\text{(b)}\;
\bm \Sigma_1 =
\begin{pmatrix}
1 & 0\\
0 & 10
\end{pmatrix},
\quad
\bm \Sigma_2 =
\begin{pmatrix}
10 & 0\\
0 & 1
\end{pmatrix}.
\]

Figure \ref{fig:power-2} reports the power of the tests for varying sample sizes. In this setting, regression-based tests perform poorly because the covariance between $\eta_1$ and $\eta_2$ is zero. In Example \ref{exa:ex2}(a), the AUG and AUG.CRT tests and the cDC.CRT test exhibit comparable performance. However, in Example \ref{exa:ex2}(b), the AUG and AUG.CRT tests significantly outperform their competitors.

Note that in case (b) the correlation between the pairwise distances of $\eta_1$ and $\eta_2$ is negative, a setting in which distance-correlation-based methods are known to perform poorly. For example, in the context of independence testing, \citet{sarkar2018} observed that when the correlation between pairwise distances is negative, the distance correlation test may have low power even when the sample size is large. A similar phenomenon was also reported in the conditional independence setting by \citet{banerjee2024ball}.
\end{example}

\begin{figure}[h!]
\centering

\begin{tikzpicture}[scale = 1.05]
\begin{axis}[xmin = 10, xmax = 100, ymin = 0, ymax = 1, xlabel = {Sample Size}, ylabel = {Estimates}, title = {\bf Example 2 (a)},
 grid=major,
    major grid style={gray!30},
    tick style={black},
    axis line style={black},
    mark size=1.2pt
    ]
\addplot[color = red,   mark = *, step = 1cm,very thin]coordinates{(10,0.090) (20,0.219) (30,0.357) (40,0.476) (50,0.583) (60,0.673) (70,0.746) (80,0.814) (90,0.879)  (100,0.881)};

\addplot[color = blue,   mark = diamond*, step = 1cm,very thin]coordinates{(10,0.142) (20,0.308) (30,0.434) (40,0.562) (50,0.651) (60,0.723) (70,0.776) (80,0.804) (90,0.856)  (100,0.868)};

\addplot[color = purple, mark = square*, step = 1cm, very thin]coordinates{(10,0.167) (20,0.298) (30,0.424) (40,0.562) (50,0.676) (60,0.744) (70,0.828) (80,0.898) (90,0.915)  (100,0.948)};

\addplot[color = magenta, mark = diamond*, step = 1cm,very thin]coordinates{(10,0.095) (20,0.108) (30,0.118) (40,0.144) (50,0.158) (60,0.151) (70,0.173) (80,0.185) (90,0.185)  (100,0.184)};

\addplot[color = ForestGreen, mark = x, step = 1cm,very thin]coordinates{(10,0.132) (20,0.083) (30,0.067) (40,0.064) (50,0.055) (60,0.052) (70,0.057) (80,0.055) (90,0.054)  (100,0.055)};

\addplot[color = black, mark = star, step = 1cm,very thin]coordinates{(10,0.051) (20,0.061) (30,0.055) (40,0.047) (50,0.064) (60,0.058) (70,0.052) (80,0.052) (90,0.051)  (100,0.038)};

\addplot[color = blue, mark = triangle*, step = 1cm,very thin]coordinates{(10,0.191) (20,0.096) (30,0.086) (40,0.069) (50,0.072) (60,0.071) (70,0.061) (80,0.064) (90,0.063)  (100,0.048)};

\end{axis}
\end{tikzpicture}
\begin{tikzpicture}[scale = 1.05]
\begin{axis}[xmin = 10, xmax = 100, ymin = 0, ymax = 1, xlabel = {Sample Size}, ylabel = {Estimates}, title = {\bf Example 2 (b)},
 grid=major,
    major grid style={gray!30},
    tick style={black},
    axis line style={black},
    mark size=1.2pt
    ]

\addplot[color = red,   mark = *, step = 1cm,very thin]coordinates{(10,0.096) (20,0.210) (30,0.314) (40,0.417) (50,0.534) (60,0.612) (70,0.669) (80,0.742) (90,0.788)  (100,0.856)};

\addplot[color = blue,   mark = diamond*, step = 1cm,very thin]coordinates{(10,0.137) (20,0.270) (30,0.381) (40,0.504) (50,0.607) (60,0.681) (70,0.742) (80,0.805) (90,0.827)  (100,0.869)};

\addplot[color = purple, mark = square*, step = 1cm, very thin]coordinates{(10,0.005) (20,0.006)(30, 0.006)(40, 0.025)(50, 0.051)(60, 0.087)(70, 0.177)(80, 0.290) (90,0.450)  (100,0.586)};

\addplot[color = magenta, mark = diamond*, step = 1cm,very thin]coordinates{(10,0.059)(20, 0.094)(30, 0.100)(40, 0.103)(50, 0.134)(60, 0.118)(70, 0.160)(80, 0.143)(90, 0.150)(100,  0.154)};

\addplot[color = ForestGreen, mark = x, step = 1cm,very thin]coordinates{(10,0.066)(20, 0.046)(30, 0.032)(40, 0.028)(50, 0.033)(60, 0.031)(70, 0.028)(80, 0.031)(90, 0.033)(100,  0.030)};

\addplot[color = black, mark = star, step = 1cm,very thin]coordinates{(10,0.041)(20, 0.030)(30, 0.024)(40, 0.029)(50, 0.027)(60, 0.029)(70, 0.020)(80, 0.023)(90, 0.028)(100,  0.027)};

\addplot[color = blue, mark = triangle*, step = 1cm,very thin]coordinates{(10,0.152)(20, 0.073)(30, 0.038)(40, 0.063)(50, 0.051)(60, 0.057)(70, 0.039)(80, 0.053)(90, 0.051)(100,  0.040)};

\end{axis}
\end{tikzpicture}

\caption{Powers of the AUG test (\textcolor{red}{$\tikzcircle{2pt}$}), the AUG.CRT test (\textcolor{blue}{$\blacklozenge$}), the cDC.CRT test (\textcolor{purple}{$\blacksquare$}), the AC test (\textcolor{magenta}{$\blacklozenge$}), the GCM test (\textcolor{ForestGreen}{$\times$}), the GCM.fix test (\textcolor{blue}{$\blacktriangle$}) and the wGCM.est test (\textcolor{black}{$\star$}) in Examples 2 (a)-(b).}
 
    \label{fig:power-2}
\end{figure}

Now a high-dimensional alternative of Example \ref{exa:ex2} is considered. In this setting, $X$ and $Y$ are univariate, while the dimension of the confounding random vector $\bm Z$ is allowed to increase. The power of the tests is evaluated as a function of the dimension of $\bm Z$. In this scenario, the GCM and wGCM tests are not applicable, and the cDC.CRT test becomes computationally expensive. Therefore, the performance of the AUG tests is compared only with the AC test.
\begin{figure}[h!]
\centering
    
\begin{tikzpicture}[scale = 1.05]
\begin{axis}[xmin = 1, xmax = 10, ymin = 0, ymax = 1, xlabel = {$\log_2(d)$}, ylabel = {Estimates}, title = {\bf Example 3 (a)},
 grid=major,
    major grid style={gray!30},
    tick style={black},
    axis line style={black},
    mark size=1.2pt
    ]

\addplot[color = red,   mark = *, step = 1cm, very thin]coordinates{(1,0.557)(2,0.546)(3,0.445)(4,0.328)(5,0.207)(6,0.095)(7,0.061)(8,0.052)(9,0.048)(10,0.044)};

\addplot[color = blue,   mark = diamond*, step = 1cm,very thin]coordinates{
(1,0.609)(2,0.557)
(3,0.49)
(4,0.362)
(5,0.225)
(6,0.092)
(7,0.055)
(8,0.046)
(9,0.037)
(10,0.043)};

\addplot[color = purple, mark = diamond*, step = 1cm, very thin]coordinates{(1,0.162)
(2,0.142)
(3,0.139)
(4,0.132)
(5,0.134)
(6,0.141)
(7,0.138)
(8,0.123)
(9,0.115)
(10,0.105)
};

\end{axis}
\end{tikzpicture}
\begin{tikzpicture}[scale = 1.05]
\begin{axis}[xmin = 1, xmax = 10, ymin = 0, ymax = 1, xlabel = {$\log_2(d)$}, ylabel = {Estimates}, title = {\bf Example 3 (b)},
 grid=major,
    major grid style={gray!30},
    tick style={black},
    axis line style={black},
    mark size=1.2pt
    ]
\addplot[color = red,   mark = *, step = 1cm,very thin]coordinates{(1,0.488)
(2,0.439)
(3,0.361)
(4,0.241)
(5,0.145)
(6,0.104)
(7,0.063)
(8,0.053)
(9,0.048)
(10,0.048)};

\addplot[color = blue,   mark = diamond*, step = 1cm,very thin]coordinates{
(1,0.566)
(2,0.522)
(3,0.422)
(4,0.325)
(5,0.182)
(6,0.125)
(7,0.082)
(8,0.056)
(9,0.052)
(10,0.041)
};

\addplot[color = magenta, mark = diamond*, step = 1cm, very thin]coordinates{
(1,0.132)
(2,0.121)
(3,0.099)
(4,0.111)
(5,0.095)
(6,0.079)
(7,0.078)
(8,0.066)
(9,0.06)
(10,0.053)};

\end{axis}
\end{tikzpicture}
\caption{Powers of the AUG test (\textcolor{red}{$\tikzcircle{2pt}$}), the AUG.CRT test (\textcolor{blue}{$\blacklozenge$}) and the AC test (\textcolor{magenta}{$\blacklozenge$}) in Examples 3 (a) and (b).}
    \label{fig:power-3}
\end{figure}

\begin{example}\label{exa:ex3}
 Let $X = f(\bm Z)+\eta_1$ and $Y = f(\bm Z)+\eta_2$, where $\eta_1$ and $\eta_2$ are i.i.d. noise variables generated as in Example \ref{exa:ex2}, and $\bm Z$ is sampled independently from a $d$-dimensional standard normal distribution. Here,
\[
f(\bm u)=\frac{\sum_{i=1}^{d}|u_i|}{\sqrt{\sum_{i=1}^{d}u_i^2}}.
\]

The power of the tests is evaluated for different choices of the dimension $d$ ($2^i$, $i=1,2,\ldots$) while the sample size is fixed at $n=50$. The results are reported in Figure \ref{fig:power-3}. It can be observed that the tests exhibit moderate power in lower dimensions. However, as the dimension increases, the power of the tests drops below the nominal level $\alpha$.

Note that these procedures are distance-based tests. In high-dimensional settings, pairwise distances tend to concentrate around certain values. In this example, the concentration of the pairwise distances is determined primarily by the distribution of $\bm Z$. Consequently, as the dimension increases, the information about the conditional dependence contained in the pairwise distances diminishes, leading to the decline in power observed in Figure \ref{fig:power-3}.

According to Proposition \ref{thm:high-dim-consistency}, this limitation can be alleviated if the sample size increases with the dimension at an appropriate rate. This phenomenon is explored in the following example.
\end{example}

\begin{figure}[h!]
\centering
    
\begin{tikzpicture}[scale = 1.05]
\begin{axis}[xmin = 1, xmax = 5, ymin = 0, ymax = 1, xlabel = {$\log_2(d)$}, ylabel = {Estimates}, title = {\bf Example 4 (a)},
 grid=major,
    major grid style={gray!30},
    tick style={black},
    axis line style={black},
    mark size=1.2pt
    ]

\addplot[color = red,   mark = *, step = 1cm, very thin]coordinates{(1,0.242)
(2,0.4)
(3,0.719)
(4,0.998)
(5,1)};

\addplot[color = blue,   mark = diamond*, step = 1cm,very thin]coordinates{
(1,0.285)
(2,0.455)
(3,0.733)
(4,0.962)
(5,1)};

\addplot[color = purple, mark = diamond*, step = 1cm, very thin]coordinates{
(1,0.103)
(2,0.126)
(3,0.154)
(4,0.306)
(5,0.667)};

\end{axis}
\end{tikzpicture}
\begin{tikzpicture}[scale = 1.05]
\begin{axis}[xmin = 1, xmax = 5, ymin = 0, ymax = 1, xlabel = {$\log_2(d)$}, ylabel = {Estimates}, title = {\bf Example 4 (b)},
 grid=major,
    major grid style={gray!30},
    tick style={black},
    axis line style={black},
    mark size=1.2pt
    ]

\addplot[color = red,   mark = *, step = 1cm,very thin]coordinates{
(1,0.236)
(2,0.293)
(3,0.598)
(4,0.965)
(5,1)};

\addplot[color = blue,   mark = diamond*, step = 1cm,very thin]coordinates{
(1,0.281)
(2,0.373)
(3,0.693)
(4,0.958)
(5,1)};

\addplot[color = magenta, mark = diamond*, step = 1cm, very thin]coordinates{
(1,0.082)
(2,0.087)
(3,0.133)
(4,0.19)
(5,0.408)};

\end{axis}
\end{tikzpicture}
\caption{Powers of the AUG test (\textcolor{red}{$\tikzcircle{2pt}$}), the AUG.CRT test (\textcolor{blue}{$\blacklozenge$}) and the AC test (\textcolor{magenta}{$\blacklozenge$}) in Examples 4 (a) and (b).}
    \label{fig:power-4}
\end{figure}
\begin{example}\label{exa:ex4}
    Consider the same data-generating mechanism as in Example \ref{exa:ex3}, but now generate $n=d^2+20$ observations, where $d$ denotes the dimension of the random vector $\bm Z$. The results are reported in Figure \ref{fig:power-4}. It can be observed that the behaviour of the proposed test is consistent with the theoretical result stated in Proposition \ref{thm:high-dim-consistency}. 

Interestingly, a similar trend is observed for the AC test, suggesting that it may also exhibit high-dimensional consistency in this setting. However, the proposed test demonstrates substantially better performance compared to the AC test.   
\end{example}

\begin{example}\label{exa:ex5}
Let \(Z \sim \mathcal{N}(0,1)\) and suppose \(\bm X=(X_1,\ldots,X_d)^\top\) and \( \bm Y=(Y_1,\ldots,Y_d)^\top\)
are $d$-dimensional random vectors. Set \(X_j = Z+\varepsilon_j, j=1,\ldots,d,\) where \(\varepsilon_1,\ldots,\varepsilon_d\) are i.i.d. \(\mathcal{N}(0,1)\) random variables independent of \(Z\). Conditional on \(X\) and \(Z\), generate \(Y_k = Z+r~X_k+\eta_k, k=1,\ldots,d,\)
where \(\eta_1,\ldots,\eta_d\) are i.i.d. \(\mathcal{N}(0,1)\) random variables independent of \((X,Z)\), and \(r\in\mathbb R\).
It is easy to verify that \(\bm X \indpt \bm Y \mid Z\) when \(r=0\), whereas \(\bm X \not\indpt \bm Y \mid Z\) whenever \(r\neq 0\). Moreover, the strength of the conditional dependence between \(\bm X\) and \(\bm Y\) given \(Z\) increases as \(|r|\) increases.

Here only the of the performances of the AUG test and AUG.CRT test are evaluated for different values of \(d\) and \(r\) based on $50$ random observations. The other competing tests are not applicable in this scenario due to the high-dimensionality of $\bm X$ and $\bm Y$. Throughout the experiment, the confounding random variable \(Z\) remains one-dimensional, while both \(\bm X\) and \(\bm Y\) are allowed to be high-dimensional. The power of the test are computed by the proportion of times the tests rejected $H_0$ based on $500$ replications of the experiments. 
\end{example}

\begin{figure}[h!]
\centering
    
\begin{tikzpicture}[scale = 1.05]
\begin{axis}[xmin = 1, xmax = 5, ymin = 0, ymax = 1, xlabel = {$\log_2(d)$}, ylabel = {Estimates}, title = {(a)},
 grid=major,
    major grid style={gray!30},
    tick style={black},
    axis line style={black},
    mark size=1.2pt
    ]

\addplot[color = red,   mark = *, step = 1cm, very thin]coordinates{(1,0.046)(2,0.041)(3,0.054)(4,0.044)(5,0.054)};

\addplot[color = blue,   mark = diamond*, step = 1cm,very thin]coordinates{(1,0.042)(2,0.05)(3,0.044)(4,0.048)(5,0.049)};

\end{axis}
\end{tikzpicture}
\begin{tikzpicture}[scale = 1.05]
\begin{axis}[xmin = 1, xmax = 5, ymin = 0, ymax = 1, xlabel = {$\log_2(d)$}, ylabel = {Estimates}, title = {(b)},
 grid=major,
    major grid style={gray!30},
    tick style={black},
    axis line style={black},
    mark size=1.2pt
    ]

\addplot[color = red,   mark = *, step = 1cm, very thin]coordinates{(1,0.925)(2,0.646)(3,0.63)(4,0.619)(5,0.64)};

\addplot[color = blue,   mark = diamond*, step = 1cm,very thin]coordinates{(1,0.401)(2,0.065)(3,0.055)(4,0.043)(5,0.041)};

\end{axis}
\end{tikzpicture}
\begin{tikzpicture}[scale = 1.05]
\begin{axis}[xmin = 1, xmax = 5, ymin = 0, ymax = 1, xlabel = {$\log_2(d)$}, ylabel = {Estimates}, title = {(c)},
 grid=major,
    major grid style={gray!30},
    tick style={black},
    axis line style={black},
    mark size=1.2pt
    ]

\addplot[color = red,   mark = *, step = 1cm, very thin]coordinates{(1,0.425)(2,0.522)(3,0.761)(4,0.935)(5,0.989)};

\addplot[color = blue,   mark = diamond*, step = 1cm,very thin]coordinates{(1,0.132)(2,0.047)(3,0.051)(4,0.054)(5,0.039)};

\end{axis}
\end{tikzpicture}
\caption{Powers of the AUG test (\textcolor{red}{$\tikzcircle{2pt}$}), the AUG.CRT test (\textcolor{blue}{$\blacklozenge$}) in Examples \ref{exa:ex5} when (a) $r = 0$ and $n=50$, (b) $r=10$ and $n=50$ and (c) $r=10$ and $n = d^2+20$.}
\label{fig:power-5}
\end{figure}

In Figure \ref{fig:power-5}, the AUG test exhibited better performance than the AUG.CRT test especially when the sample size diverges to infinity along with the dimension of the data. In case (b), the effect of increasing dimentionality effected both of these tests adversely. However, with increasing sample size in case (c), the AUG test exhibited satisfactory performance. Moreover, the AUG test is computationally much faster than the AUG.CRT test. In particular, the proposed method requires generating augmented samples only once, which makes it computationally efficient in practice. These advantages, together with the guaranteed power properties of the test, make the AUG test more appealing than the AUG.CRT test from both theoretical and practical perspectives.

\section{Approximate null exchangeable pairs}
\label{sec:estimation-effect}

In practice, the conditional distribution $P_{\bm X|\bm Z}$ is rarely known and must be estimated from data. Therefore, it is important to understand whether the theoretical guarantees established in the previous sections persist when the null exchangeable pairs are generated from an estimated conditional distribution. This section investigates the effect of using approximate null exchangeable pairs on the level and power of the proposed test.

Following \cite{berrett2019}, the behaviour of the test is studied when, instead of assuming knowledge of \(P_{\bm X|\bm Z}\), an approximation \(Q_{\bm X|\bm Z}^{(n)}\) is used to generate the null exchangeable pairs. For example, \(Q_{\bm X|\bm Z}^{(n)}\) may be an estimate of \(P_{\bm X|\bm Z}\) constructed either from the same data or from an auxiliary unlabeled sample \(\{(\bm X_i^*, \bm Z_i^*)\}_{i=1}^N\) drawn from \(P_{(\bm X,\bm Z)}\). In the rest of the article, the process of generating null exchangeable pairs constructed using $P_{\bm X|\bm Z}$ and $Q_{\bm X|\bm Z}^{(n)}$ will be referred to as the exact and approximate sampling schemes, respectively. The following result establishes an upper bound on the Type~I error rate and shows that the test remains consistent provided that the total variation distance between \(\prod_{i=1}^n Q_{\bm X|\bm Z_i}^{(n)}\) and \(\prod_{i=1}^n P_{\bm X|\bm Z_i}\) converges to zero, i.e., the corresponding sequences of distributions are contiguous.

\begin{theorem}
Assume that $H_0 : \bm X \indpt \bm Y \mid \bm Z$ is true, and that the conditional distribution of $\bm X \mid \bm Z$ is given by $P_{\bm X|\bm Z}$. Let $\bm X^{(1)}$ be copies of $\bm X$ generated using an estimate $Q_{\bm X|\bm Z}^{(n)}$ of the true conditional distribution $P_{\bm X|\bm Z}$. Then, for any desired Type~I error rate $\alpha \in [0,1]$,

\[
\mathbb{P}\{p_{n,B} \le \alpha \mid (\bm Y_i, \bm Z_i)_{i=1}^n\} \le \frac{\lfloor (B+1)\alpha \rfloor}{B+1} + d_{\mathrm{TV}}\Bigg(\prod_{i=1}^n Q_{\bm X|\bm Z_i}^{(n)}, \prod_{i=1}^n P_{\bm X|\bm Z_i}\Bigg),
\]
where $d_{\mathrm{TV}}(Q_1,Q_2) = \sup_A \big|Q_1(A)-Q_2(A)\big|$, $p_{n,B}$ is as in \eqref{eq:cond-pval}, and the probability is taken with respect to the distribution of $\bm X$ and $\bm X^{(1)}$ conditional on $\bm Y,\bm Z$.
\label{thm:excess-type-I}
\end{theorem}

Note that by Theorem~\ref{thm:excess-type-I}, the proposed test will have asymptotic Type~I error rate control if

\[
d_{\mathrm{TV}}\Bigg(\prod_{i=1}^n Q_{\bm X|\bm Z_i}^{(n)}, \prod_{i=1}^n P_{\bm X|\bm Z_i}\Bigg) = o_{P}(1).
\]

It is also interesting to observe that a simple application of the dominated convergence theorem with the above assumption ensures that the joint distributions of $\{(\bm X_i, \bm X_i^\prime, \bm Y_i, \bm Z_i)\}_{i=1}^n$ and \linebreak$\{(\bm X_i, \bm X_i^{(1)}, \bm Y_i, \bm Z_i)\}_{i=1}^n$ are contiguous. The following proposition shows that contiguity is sufficient to ensure the consistency of the proposed test.

\begin{proposition}
Assume that $H_1 : \bm X \not\indpt \bm Y \mid \bm Z$ is true, and that the conditional distribution of $\bm X \mid \bm Z$ is given by $P_{\bm X|\bm Z}$. Let $\bm X^{(1)}$ be copies of $\bm X$ generated using an estimate $Q_{\bm X|\bm Z}^{(n)}$ of the true conditional distribution $P_{\bm X|\bm Z}$. Also assume that

\[
d_{\mathrm{TV}}\Bigg(\prod_{i=1}^n Q_{\bm X|\bm Z_i}^{(n)}, \prod_{i=1}^n P_{\bm X|\bm Z_i}\Bigg) \stackrel{\P}{\longrightarrow} 0,
\]
as $n \to \infty$. Then, $\hat\zeta_{n,\sigma}^{(1)}$, the statistic computed based on $\{(\bm X_i,\bm X_i^{(1)}, \bm Y_i,\bm Z_i)\}_{i=1}^n$, converges in probability to $\zeta_\sigma(P)$ (as defined in \eqref{eq:measure}, with $\bm V^\prime$ being the exact null exchangeable pair) as $n \to \infty$.
\label{prop:contiguity-consistency}
\end{proposition}

Hence, the proposed test is robust against model misspecification or when an approximation is used to generate the null exchangeable pairs, as long as the joint distributions of the augmented data under the exact sampling scheme and the approximate sampling scheme are contiguous in the sense of total variation distance. Interested readers are referred to Section~5.1 of \cite{berrett2019} for a discussion of when this assumption is satisfied.

\subsection{Empirical analysis}

The performance of the AUG test is investigated when the null exchangeable pairs are generated using an estimated conditional distribution rather than the true distribution $P_{\bm X|\bm Z}$. Throughout, the focus is on Example~\ref{exa:ex1}.
To approximate $P_{\bm X|\bm Z}$, a kernel-based nonparametric estimator is constructed from an independent training sample of size $m_n$. Specifically, let \(\{(\widetilde{\bm X}_i,\widetilde{\bm Z}_i)\}_{i=1}^{m_n}\) be an independent sample drawn from the joint distribution of $(\bm X,\bm Z)$. Using this sample, the conditional distribution of $\bm X$ given $\bm Z$ is estimated through a plug-in kernel smoothing approach, which is given by 
\begin{align*}
		\hat p_{\bm X| \bm Z=\bm z}({\bm x| \bm z}) = \frac{\hat p_{\bm X,\bm Z}({\bm x,\bm z})}{\hat p_{\bm Z}({\bm z})} & = \frac{\frac{1}{m_nh_0^{d_X}h_{2}^{\prime d_Z}}\sum_{s=1}^{m_n} K\left(\frac{ \| \bm{z} - \bm{Z}_s \| }{h_{2}'}\right) \phi\left(\frac{\bm x-\bm X_s}{h_0}\right)}{\frac{1}{m_nh_{2}^{\prime d_Z}}\sum_{s=1}^{m_n} K\left(\frac{ \| \bm{z} - \bm{Z}_s \| }{h_{2}'}\right)} \nonumber \\ 
		& = \sum_{s=1}^{m_n} \kappa_{s}({\bm z}) \frac{1}{h_0^{d_X}} \phi\left(\frac{\bm x-\bm X_s}{h_0}\right) , 
	\end{align*}
    where $\phi$ denotes the probability density function of the standard normal distribution and 
	\begin{align*}
		\kappa_{s}({\bm z}) := \frac{ K\left(\frac{ \| \bm{z} - \bm{Z}_s \| }{h_{2}'}\right)}{\sum_{j=1}^{m_n} K\left(\frac{\bm z-\bm Z_ j}{h_{2}'}\right) } , 
	\end{align*}
	for $1 \leq s \leq m_n$. The resulting estimator is denoted by $Q_{\bm X|\bm Z}^{(n)}$. The estimated conditional distribution $Q_{\bm X|\bm Z}^{(n)}$ is then used to generate approximate null exchangeable pairs $(\bm X_1^{(1)},\ldots,\bm X_n^{(1)})$ corresponding to the observations in an independent testing sample \(\{(\bm X_i,\bm Y_i,\bm Z_i)\}_{i=1}^{n}.\) Since the training and testing samples are independent, the effect of estimating $P_{X|\bm Z}$ can be studied separately from the sampling variability of the test statistic. The augmented sample \(\{(\bm X_i,\bm X_i^{(1)},\bm Y_i,\bm Z_i)\}_{i=1}^{n}\) is subsequently used to construct the AUG test.

To assess the impact of estimating the conditional distribution, the performance of the AUG test based on the approximate null exchangeable pairs generated from $\{(\bm X_i,\bm X_i^{(1)},\bm Y_i,\bm Z_i)\}_{i=1}^{n}$ is compared with that of the oracle AUG test, where the null exchangeable pairs are generated from the true conditional distribution $P_{X|\bm Z}$. Throughout the experiments, the training sample size is taken to satisfy \(m_n \asymp n^\gamma,\)
and the empirical level and power of the two procedures are evaluated for several values of $\gamma$. The experiments are carried out under both the null hypothesis ($r=0$) and the alternative hypothesis ($r=0.5,1$).

\begin{figure}[h!]
\centering
    
\begin{tikzpicture}[scale = 1.05]
\begin{axis}[xmin = 1, xmax = 3.1, ymin = 0, ymax = 1, xlabel = {$\gamma$}, ylabel = {Estimates}, title = {$r=0$},
 grid=major,
    major grid style={gray!30},
    tick style={black},
    axis line style={black},
    mark size=1.2pt
    ]

\addplot[color = red,   mark = *, step = 1cm, very thin]coordinates{(1,0.046)(1.5,0.046)(2,0.046)(2.5,0.046)(3,0.046)};

\addplot[color = magenta,   mark = *, step = 1cm, very thin]coordinates{(1,0.167)(1.5,0.074)(2,0.054)(2.5,0.048)(3,0.05)};

\end{axis}
\end{tikzpicture}
\begin{tikzpicture}[scale = 1.05]
\begin{axis}[xmin = 1, xmax = 3.1, ymin = 0, ymax = 1, xlabel = {$\gamma$}, ylabel = {Estimates}, title = {$r=0.5$},
 grid=major,
    major grid style={gray!30},
    tick style={black},
    axis line style={black},
    mark size=1.2pt]

\addplot[color = red,   mark = *, step = 1cm, very thin]coordinates{(1,0.167)(1.5,0.167)(2,0.167)(2.5,0.167)(3,0.167)};

\addplot[color = magenta,   mark = *, step = 1cm, very thin]coordinates{(1,0.555)(1.5,0.292)(2,0.191)(2.5,0.197)(3,0.186)};

\end{axis}
\end{tikzpicture}
\begin{tikzpicture}[scale = 1.05]
\begin{axis}[xmin = 1, xmax = 3.1, ymin = 0, ymax = 1, xlabel = {$\gamma$}, ylabel = {Estimates}, title = {$r=1$},
 grid=major,
    major grid style={gray!30},
    tick style={black},
    axis line style={black},
    mark size=1.2pt]

\addplot[color = red,   mark = *, step = 1cm,very thin]coordinates{(1,0.762)(1.5,0.762)(2,0.762)(2.5,0.762)(3,0.762)};

\addplot[color = magenta,   mark = *, step = 1cm, very thin]coordinates{(1,0.975)(1.5,0.902)(2,0.842)(2.5,0.814)(3,0.807)};

\end{axis}
\end{tikzpicture}
\caption{Powers of the AUG test using the exact null exchangeable pairs (\textcolor{red}{$\tikzcircle{2pt}$}) and approximate null exchangeable pairs (\textcolor{magenta}{$\tikzcircle{2pt}$}) using Example 1 for $r=0,0.5$ and $1$.}
    \label{fig:power-4}
\end{figure}

Note that as $\gamma$ increases, i.e., as the training sample size grows, the power of the AUG test based on approximate null exchangeable pairs approaches that of the AUG test with the exact null exchangeable pairs. Therefore, to ensure validity and desirable power properties when the sampling model is trained on an external dataset, one requires a sufficiently large training sample size. Otherwise, following \cite{banerjee2024ball}, the sampling scheme may be designed appropriately to guarantee asymptotic validity and consistency.

}

\section*{Acknowledgments}
The author gratefully acknowledges Anil K. Ghosh, Paromita Dubey, Samriddha Lahiry, Bhaswar B. Bhattacharya, and Somabha Mukherjee for their thoughtful comments, which helped improve the quality of this article.

\bibliographystyle{apalike}
\bibliography{refs}

\section{High Dimensional Consistency: The Gaussianity Case}

{
%\subsection{Gaussian Construction of Conditionally Independent Variants}

In this section, a simple intuition behind the distribution of null exchangeable pairs of the multivariate Gaussian random vectors is described. Then the exact expression of $\zeta_\sigma(\mathrm P)$ is derived when $\mathrm P$ is a Gaussian distribution. 
%\subsubsection{Conditional independence under multivariate normality}

\subsection{null exchangeable pairs for Gaussian random vectors}
Suppose that the random vector $(\bm X,\bm Y,\bm Z)$ follows a multivariate normal distribution, i.e.,
\(
(\bm X,\bm Y,\bm Z) \sim \mathcal{N}(\bm \mu, \bm \Sigma),
\)
where $\bm \mu$ is the mean vector and $\bm \Sigma$ is the variance-covariance matrix. Without loss of generality, one can write the variance-covariance matrix in block form as
\[
\bm \Sigma =
\begin{pmatrix}
\bm \Sigma_{XX} & \bm \Sigma_{XY} & \bm \Sigma_{XZ} \\
\bm \Sigma_{YX} & \bm \Sigma_{YY} & \bm \Sigma_{YZ} \\
\bm \Sigma_{ZX} & \bm \Sigma_{ZY} & \bm \Sigma_{ZZ}
\end{pmatrix}.
\]
A well-known property of the multivariate normal distribution is that conditional independence is equivalent to a null partial cross-covariance matrix, i.e.,
\[
\bm X \indpt \bm Y \mid \bm Z
\quad \Longleftrightarrow \quad
\bm \Sigma_{XY|Z} = 0,
\]
where the partial cross-covariance matrix is defined as
\[
\bm \Sigma_{XY|Z}
= \E\big[(\bm X- \E[\bm X|\bm Z])(\bm Y - \E[\bm Y|\bm Z])^\top\big] = 
\bm \Sigma_{XY}
-
\bm \Sigma_{XZ}\bm \Sigma_{ZZ}^{-1}\bm \Sigma_{ZY}.
\]
Equivalently, the cross-covariance matrix can be decomposed as
\[
\bm \Sigma_{XY}
=
\bm \Sigma_{XZ}\bm \Sigma_{ZZ}^{-1}\bm \Sigma_{ZY}
+
\bm \Sigma_{XY|Z}.
\]
The first term represents the component of the covariance between $\bm X$ and $\bm Y$ that is induced through their mutual dependence on $\bm Z$, while $\bm \Sigma_{XY|Z}$ measures the residual dependence between $\bm X$ and $\bm Y$ after conditioning on $\bm Z$.

%\subsubsection{Construction of a conditionally independent variant}

Consider now the construction of a modified random vector $(\bm X',\bm Y, \bm Z)$ that has the properties $ \bm X'| \bm Z \stackrel{D}{=} \bm X| \bm Z$ and $\bm X'\indpt \bm Y|\bm Z$. In the Gaussian setting this can be achieved by modifying the covariance structure. These conditions ensure that $(\bm X',\bm Y,\bm Z)$ will have a multivariate Gaussian distribution with mean $\mu$ and variance-covariance matrix
\[
\bm \Sigma' =
\begin{pmatrix}
\bm \Sigma_{XX} & \bm \Sigma_{XY}' & \bm \Sigma_{XZ} \\
\bm \Sigma_{YX}' & \bm \Sigma_{YY} & \bm \Sigma_{YZ} \\
\bm \Sigma_{ZX} & \bm \Sigma_{ZY} & \bm \Sigma_{ZZ}
\end{pmatrix},
\]
where the modified cross-covariance matrix is \(\bm \Sigma_{XY}' =
\bm \Sigma_{XZ}\bm \Sigma_{ZZ}^{-1}\bm \Sigma_{ZY}.
\)
The partial cross-covariance between $\bm X'$ and $\bm Y$ given $\bm Z$ is
\[
\bm \Sigma_{X'Y|Z}
=
\bm \Sigma_{XY}'
-
\bm \Sigma_{XZ}\bm \Sigma_{ZZ}^{-1}\bm \Sigma_{ZY} = 0.
\]
Therefore,
\( \bm X' \indpt \bm Y \mid \bm Z\). This shows that replacing the original cross-covariance $\bm \Sigma_{XY}$ with \(
\bm \Sigma_{XZ}\bm \Sigma_{ZZ}^{-1}\bm \Sigma_{ZY} \)
produces a Gaussian vector in which the conditional dependence between $\bm X$ and $\bm Y$ has been removed while preserving their marginal relationships with $\bm Z$.

%The quantity
%\[
%\Sigma_{XY|Z}
%=
%\Sigma_{XY}
%-
%Sigma_{XZ}\Sigma_{ZZ}^{-1}\Sigma_{ZY}
%\]
%is the partial covariance between $X$ and $Y$ given $Z$. The above construction can therefore be interpreted as removing the partial covariance component from the original covariance structure.

%In other words, the covariance between $X$ and $Y$ can be decomposed into two parts: the component induced by their mutual dependence on $Z$, and the residual component representing direct dependence. The modifiedvariance-covariance matrix $\Sigma'$ retains only the former component, thereby enforcing conditional independence.

%This Gaussian construction provides useful intuition for the conditional independence variants introduced in Section~2 of the main text.

\subsection{Expression of $\zeta_\sigma(\mathrm{P})$}

In this section, the exact expression of $\zeta_\sigma(\mathrm{P})$ is evaluated when $(\bm X,\bm Y,\bm Z)\sim\mathcal{N}(\bm \mu,\bm \Sigma)$. But before that, consider the following lemma.
\begin{lemma}
    If ${\bf X}_1\sim \mathcal{N}({\bf 0},\sigmat_1)$ and ${\bf X}_2\sim {\mathcal N}({\bf 0},\sigmat_2)$ are independent $d$-dimensional random vectors,
    $$\E\Big\{\exp\big(-\frac{\sigma^2}{2}\|{\bf X}_1-{\bf X}_2\|^2\big)\Big\} = \Big|\sigma^2(\sigmat_1+\sigmat_2)+{\bf I}_d\Big|^{-1/2},$$
    where $\big| {\bf A}\big|$ denotes the determinant of a matrix ${\bf A}$, and ${\bf I}_d$ is the $d\times d$ identity matrix. 
    \label{aux-1}
\end{lemma}

\begin{proof}[\bf Proof of Lemma \ref{aux-1}]
    Here ${\bf N} = {\bf X}_1-{\bf X}_2$ follows $\mathcal{N}( {\bf 0}, \sigmat_1+\sigmat_2)$. So, we have %Therefore, it is enough to find the expectation of $\E\Big\{\exp\big(-\frac{1}{2d}\|N\|^2\big)\Big\}$. Note notice that,
    $$\E\Big\{\exp\big(-\frac{\sigma^2}{2}\|{\bf N}\|^2\big)\Big\} = \int \frac{1}{(2\pi)^{d/2}\Big|\sigmat_1+\sigmat_2\Big|^{1/2}}\exp\Big(-\frac{\sigma^2}{2}\|{\bf u}\|^2-\frac{1}{2}{\bf u}^{\top} (\sigmat_1+\sigmat_2)^{-1} {\bf u}\Big) {\mathrm d}{\bf u}.$$

    Note that the exponent on the right side is the same as that of the density of a normal distribution with mean ${\bf 0}$ and variance-covariance matrix $\big(\sigma^2{\bf I}_d+(\sigmat_1+\sigmat_2)^{-1}\big)^{-1}$. Therefore, we have

    $$\E\Big\{\exp\big(-\frac{\sigma^2}{2}\|{\bf N}\|^2\big)\Big\} = \frac{1}{\Big|\sigmat_1+\sigmat_2\Big|^{1/2}\Big|\sigma^2{\bf I}_d+(\sigmat_1+\sigmat_2)^{-1}\Big|^{1/2}} = \Big|\sigma^2(\sigmat_1+\sigmat_2)+ {\bf I}_d\Big|^{-1/2}.$$
    This completes the proof.
\end{proof}

Using Lemma \ref{aux-1} it is easy to derive that $\zeta_\sigma\big(\mathcal{N}(\bm \mu,\bm \Sigma)\big)$ admits the closed form
\[
\zeta_\sigma\big(\mathcal{N}(\bm \mu,\bm \Sigma)\big)
=
\left|\bm I+2\sigma^2\bm \Sigma\right|^{-1/2}
+
\left|\bm I+2\sigma^2\bm \Sigma'\right|^{-1/2}
-
2
\left|\bm I+\sigma^2(\bm \Sigma+\bm \Sigma')\right|^{-1/2}.
\]
Since the two covariance matrices $\bm \Sigma$ and $\bm \Sigma'$ differ only through the block $\bm \Sigma_{XY|Z}$, $\zeta_\sigma(\mathrm{P})$ is completely determined by this partial cross-covariance matrix. In particular, if
\[
\bm \Sigma_{XY|Z}=0, \Leftrightarrow \text{ $\bm \Sigma = \bm \Sigma'$ and hence } \zeta_\sigma(\mathrm{P}) = 0.
\]

More generally, the magnitude of $\zeta_\sigma\big(\mathcal{N}(\bm \mu,\bm \Sigma)\big)$ reflects the strength of the residual dependence between $\bm X$ and $\bm Y$ after conditioning on $\bm Z$. In the Gaussian setting, this dependence is fully characterized by the matrix $\bm \Sigma_{XY|Z}$. Consequently, the behaviour of the proposed test statistic is governed by the size and structure of this partial cross-covariance matrix.

In particular, when the entries of $\bm \Sigma_{XY|Z}$ are large, the two covariance matrices $\bm \Sigma$ and $\bm \Sigma'$ differ substantially and the resulting $\zeta_\sigma\big(\mathcal{N}(\bm \mu,\bm \Sigma)\big)$ is correspondingly large. Conversely, when $\bm \Sigma_{XY|Z}$ is close to zero, the two Gaussian distributions are nearly indistinguishable and the measure becomes small. This observation provides useful intuition for the behaviour of the proposed conditional independence test in the Gaussian setting: the test essentially measures the discrepancy between the joint distribution and its conditionally independent variant, and this discrepancy is controlled by the partial cross-covariance between $\bm X$ and $\bm Y$ given $\bm Z$.

To better understand the role of the partial cross-covariance matrix, define
\[
\bm A=\bm I+2\sigma^2\bm \Sigma_2.
\]
Since \(\bm \Sigma_2\) is a covariance matrix, \(\bm A\) is positive definite and hence invertible. Therefore, using the identity
\( \det(\bm A+\bm B)=\det(\bm A)\det(\bm I+\bm A^{-1}\bm B),\) we obtain
\[
\det(\bm A+2\sigma^2\bm \Delta)
=
\det(\bm A)\det(\bm I+2\sigma^2\bm A^{-1}\bm \Delta)
\]
and
\[
\det(\bm A+\sigma^2\bm \Delta)
=
\det(\bm A)\det(\bm I+\sigma^2\bm A^{-1}\bm \Delta).
\]
Substituting these expressions above gives
\[
\zeta_\sigma(\mathrm P)
=
\frac{1}{\sqrt{\det(\bm A)}}
\Big[
\det(\bm I+2\sigma^2\bm A^{-1}\bm \Delta)^{-1/2}
+
1
-
2\det(\bm I+\sigma^2\bm A^{-1}\bm \Delta)^{-1/2}
\Big].
\]
Hence, the dependence of \(\zeta_\sigma(\mathrm P)\) on the underlying distribution is completely determined by the matrix
\[
\bm A^{-1}\bm \Delta
=
(\bm I+2\sigma^2\bm \Sigma_2)^{-1}\bm \Delta.
\]
If \(\lambda_1,\ldots,\lambda_d\) denote the eigenvalues of \(\bm A^{-1}\bm \Delta\), then
\[
\det(\bm I+t\bm A^{-1}\bm \Delta)
=
\prod_{i=1}^{d}(1+t\lambda_i),
\]
and consequently
\[
\zeta_\sigma(\mathrm P)
=
\frac{1}{\sqrt{\det(\bm A)}}
\left[
\prod_{i=1}^{d}(1+2\sigma^2\lambda_i)^{-1/2}
+
1
-
2\prod_{i=1}^{d}(1+\sigma^2\lambda_i)^{-1/2}
\right].
\]
To further understand the magnitude of \(\zeta_\sigma(\mathrm P)\), consider the regime where the conditional dependence is weak, i.e., \(\|\bm \Delta\|_F\to0\). Using the expansion
\[
\det(\bm I+t\bm M)^{-1/2}
=
1-\frac{t}{2}\operatorname{tr}(\bm M)
+
t^2
\left\{
\frac14\operatorname{tr}(\bm M^2)
+
\frac18\big(\operatorname{tr}(\bm M)\big)^2
\right\}
+
O(\|\bm M\|_F^3),
\]
with \(\bm M=\bm A^{-1}\Delta\), a straightforward calculation shows that %the constant and linear terms cancel, yielding

{\[
\zeta_\sigma(\mathrm P)
=
\frac{\sigma^4}
{4\sqrt{\det(\bm A)}}
\left\{
2\operatorname{tr}\left[(\bm A^{-1}\bm \Delta)^2\right]
+
\big(\operatorname{tr}(\bm A^{-1}\bm \Delta)\big)^2
\right\}
+
O(\|\bm \Delta\|_F^3).
\]}
Since
\(
\big(\operatorname{tr}(\bm A^{-1}\bm \Delta)\big)^2\ge 0,
\)
it follows that
\[
\zeta_\sigma(\mathrm P)
\ge
\frac{\sigma^4}
{2\sqrt{\det(\bm A)}}
\operatorname{tr}(\bm A^{-1}\bm \Delta \bm A^{-1}\bm \Delta)
+
O(\|\bm \Delta\|_F^3).
\]
Let
\(
\bm A=\bm U\bm \Lambda\bm U^\top,
\)
where
\(
\bm \Lambda=\operatorname{diag}(\mu_1,\ldots,\mu_d).
\)
Writing
\(
\widetilde{\bm \Delta}=\bm U^\top \Delta \bm U,
\)
we obtain
\[
\operatorname{tr}(\bm A^{-1}\bm \Delta \bm A^{-1}\bm \Delta)
=
\operatorname{tr}
\left(
\Lambda^{-1}\widetilde{\bm \Delta}
\Lambda^{-1}\widetilde{\bm \Delta}
\right)
=
\sum_{i,j}
\frac{\widetilde{\bm \Delta}_{ij}^2}
{\mu_i\mu_j}.
\]
Therefore,
\[
\operatorname{tr}(\bm A^{-1}\bm \Delta \bm A^{-1}\bm \Delta)
\ge
\frac{1}{\lambda_{\max}(\bm A)^2}
\sum_{i,j}
\widetilde{\bm \Delta}_{ij}^2.
\]

Since orthogonal transformations preserve the Frobenius norm,
\[
\sum_{i,j}
\widetilde{\bm \Delta}_{ij}^2
= \|\widetilde{\bm \Delta}\|_F^2
 = \|\bm \Delta\|_F^2.
\]
Consequently,
\[
\operatorname{tr}(\bm A^{-1}\bm \Delta \bm A^{-1}\bm \Delta)
\ge
\frac{\|\bm \Delta\|_F^2}
{\lambda_{\max}(\bm A)^2}.
\]

Substituting the above inequality into the expansion of
\(\zeta_\sigma(\mathrm P)\) yields
\[
\zeta_\sigma(\mathrm P)
\ge
\frac{\sigma^4}
{2\sqrt{\det(\bm A)}}
\frac{\|\bm \Delta\|_F^2}{\lambda_{\max}(\bm A)^2}
+
O(\|\bm \Delta\|_F^3).
\]

Since $\|\bm \Delta\|_F\rightarrow 0$, the cubic remainder is dominated by the quadratic term, and therefore 
\[
\zeta_\sigma(\mathrm P)
\gtrsim
\|\bm \Delta\|_F^2.
\]
Also, by definition, we have
\(
\|\bm \Delta\|_F^2
 =
2\|\bm \Sigma_{XY|Z}\|_F^2.
\)
Therefore,
\[
\zeta_\sigma(\mathrm P)
\gtrsim
\|\bm \Sigma_{XY|Z}\|_F^2
 \]
whenever \(\|\bm \Sigma_{XY|Z}\|_F\to0\) and the eigenvalues of \(\bm I+2\sigma^2\bm \Sigma_2\) remain uniformly bounded. 
Then, under Gaussian weak-signal regime, the condition of Proposition~\ref{thm:high-dim-consistency} is satisfied whenever
\[
n\|\bm \Sigma_{XY|Z}\|_F^2\to\infty,
\]
and the proposed test achieves high-dimensional consistency.

}

\section{Proofs}
\vspace{0.1in}
\begin{proof}[\bf Proof of Proposition \ref{thm:characterization}:]
    It is easy to see that if $(\bm X, \bm Y, \bm Z) \sim {\Pr}$ is such that $\bm X\indpt \bm Y|\bm Z$, then $\varphi_1=\varphi_2$ and hence $\zeta_W(\Pr)=0$. So, let us prove the only if part. Recall that $\zeta_W(\Pr)$ is given by
    $$\zeta_W(\Pr) = \int \Big|\varphi_1(\bm t,\bm u, \bm v)-\varphi_2(\bm t,\bm u, \bm v)\Big|^2 W(\bm t, \bm u, \bm v) \mathrm d \bm t \mathrm d \bm u \mathrm d \bm v.$$
    Since $W(\cdot, \cdot, \cdot)$ is a non-negative measure, the non-negativity of $\zeta_W(\Pr)$ follows from the non-negativity of the integrand. Now $\zeta_W(\Pr) = 0$ implies $\varphi_1(\bm t,\bm u, \bm v) = \varphi_2(\bm t,\bm u, \bm v)$ over the set $\{(\bm t,\bm u, \bm v)\mid W(\bm t,\bm u, \bm v)>0\}$. Since $W(\cdot,\cdot,\cdot)$ is supported on the entire space, $\varphi_1(\bm t,\bm u, \bm v)=\varphi_2(\bm t,\bm u, \bm v)$ almost everywhere. Since, both $\varphi_1$ and $\varphi_2$ are continuous functions, $\zeta_W(\Pr)=0$ implies
    $\varphi_1(\bm t,\bm u, \bm v)=\varphi_2(\bm t,\bm u, \bm v)$ for all $(\bm t,\bm u, \bm v) \in \R^{d}$, where $d = d_X+d_Y+d_Z$. Hence, when $\zeta_\sigma(\mathrm{P}) = 0$, $(\bm X, \bm Y, \bm Z)$ and $(\bm X^\prime, \bm Y, \bm Z)$ are identically distributed, i.e., $\bm X\indpt \bm Y|\bm Z$. This completes the proof.   
\end{proof}

\begin{proof}[\bf Proof of Proposition \ref{thm:alt-representation}:]
    Note that for any $\bm v_1, \bm v_2$
    \begin{align}
        \int \exp\big\{i\langle {\bf s}, {\bm v}_1-{\bm v}_2\rangle\big\} \frac{1}{(2\pi\sigma^2)^{d/2}} e^{-{\|{\bf s}\|^2}/{2\sigma^2}} \mathrm d{\bf s} = \E\Big\{\exp\big\{i\langle {\bf S}, {\bm v}_1-{\bm v}_2\rangle\big\}\Big\},
        \label{characteristic-representation-v1}
    \end{align}
    where ${\bf S}$ follows a $d$-variate Gaussian distribution with mean zero and variance-covariance matrix $\sigma^2{\bf I}_d$. The right side of equation (\ref{characteristic-representation}) is the characteristic function of ${\bf S}$ evaluated at ${\bm v}_1-{\bm v}_2$. Therefore,
    \begin{align}
        \int \exp\big\{i\langle {\bf s}, {\bm v}_1-{\bm v}_2\rangle\big\} \frac{1}{(2\pi\sigma^2)^{d/2}} e^{-{\|{\bf s}\|^2}/{2\sigma^2}} \mathrm d{\bf s} = \exp\left\{-\frac{\sigma^2}{2}\|{\bm v}_1-{\bm v}_2\|^2\right\}.
        \label{characteristic-representation}
    \end{align}
    
    Now, note that
    $$\Big|\varphi_1(\bm t,\bm u, \bm v)-\varphi_2(\bm t,\bm u, \bm v)\Big|^2 = \Big(\varphi_1(\bm t,\bm u, \bm v)-\varphi_2(\bm t,\bm u, \bm v)\Big)\Big(\varphi_1(-\bm t,-\bm u, -\bm v)-\varphi_2(-\bm t,-\bm u,- \bm v)\Big).$$
    So, expanding the characteristic functions in terms of expectations, one can get 
    \begin{align}
        |\varphi_1(\bm t,\bm u, \bm v)-\varphi_2(\bm t,\bm u, \bm v)|^2 = & \E\Big\{\exp\big\{i\langle {\bf s}, \bm V_1-\bm V_2 \rangle\big\}\Big\}+ \E\Big\{\exp\big\{i\langle {\bf s}, \bm V^{'}_1-\bm V^{'}_2 \rangle\big\}\Big\} \nonumber \\ 
        &-\E\Big\{\exp\big\{i\langle {\bf s}, \bm V_1-\bm V^{'}_2 \rangle\big\}\Big\}-\E\Big\{\exp\big\{i\langle {\bf s}, \bm V_1^{'}-\bm V_2 \rangle\big\}\Big\},
        \label{eq:try}
    \end{align}
    where ${\bf s} = (\bm t, \bm u, \bm v)$. Using \eqref{characteristic-representation}, for any ${\bf V}_1$ and ${\bf V}_2$, one has
    $$ \int \E\Big\{\exp\big\{i\langle {\bf t}, {\bf V}_1-{\bf V}_2 \rangle\big\}\Big\} \frac{1}{(2\pi\sigma^2)^{d/2}} e^{-{\|{\bf t}\|^2}/{2\sigma^2}} {\mathrm d}{\bf t}=
    %Using Fubini's theorem and Lemma \ref{characteristic-representation} we get a representation of the above term as,
    \E \Big\{\exp\big\{-\frac{\sigma^2}{2}\|{\bf V}_1-{\bf V}_2\|^2\big\}\Big\}.$$
    Applying this to all four terms in (\ref{eq:try}), the result follows.
\end{proof}

\begin{proof}[\bf Proof of Proposition \ref{thm:concentration}:]
    As introduced in Section 2.3 one can write the estimator as
    $$\hat\zeta_{n,\sigma} = \frac{2}{n(n-1)}\sum_{1\leq i< j\leq n} g\big((\bm X_i, \bm X_i^\prime, \bm Y_i, \bm Z_i),(\bm X_j, \bm X_j^\prime, \bm Y_j, \bm Z_j)\big),$$
    where
    \begin{align*}
        & g\big((\bm x_1, \bm x_1^\prime, \bm y_1, \bm z_1),(\bm x_2, \bm x_2^\prime, \bm y_2, \bm z_2)\big)\\
        & = \exp\bigg\{-\frac{\sigma^2\|\bm v_1-\bm v_2\|^2}{2}\bigg\}+\exp\bigg\{-\frac{\sigma^2\|\bm v^{'}_1-\bm v^{'}_2\|^2}{2}\bigg\} \\
        & \hspace{0.75in}-\exp\bigg\{-\frac{\sigma^2\|\bm v_1-\bm v^{'}_2\|^2}{2}\bigg\}-\exp\bigg\{-\frac{\sigma^2\|\bm v^{'}_2-\bm v_1\|^2}{2}\bigg\},
    \end{align*}
    with $\bm v_i = (\bm x_i, \bm y_i, \bm z_i)$ and $\bm v^{'}_i = (\bm x_i^\prime, \bm y_i, \bm z_i)$ for $i = 1,2$. Now let $\hat\zeta_{n,\sigma}^{(i)}$ denote our estimator when the $i^{th}$ observation $(\bm X_i, \bm X_i^\prime, \bm Y_i, \bm Z_i)$ is replaced by an independent copy $(\Tilde{\bm X}_i, \Tilde{\bm X}_i^\prime, \Tilde{\bm Y}_i, \Tilde{\bm Z}_i)$ of the same. Note that
    \begin{align*}
         &|\hat\zeta_{n,\sigma}-\hat\zeta_{n,\sigma}^{(i)}|\\
         & \leq \frac{2}{n(n-1)}\Bigg\{\sum_{j=1}^{i-1} \Big|g\big((\bm X_j, \bm X_j^\prime, \bm Y_j, \bm Z_j),(\bm X_i, \bm X_i^\prime, \bm Y_i, \bm Z_i)\big)-g\big((\bm X_j, \bm X_j^\prime, \bm Y_j, \bm Z_j),(\Tilde{\bm X}_i, \Tilde{\bm X}_i^\prime, \Tilde{\bm Y}_i, \Tilde{\bm Z}_i)\big)\Big|\\ & ~~~~~~~~~+ \sum_{j=i+1}^n \Big|g\big((\bm X_i, \bm X_i^\prime, \bm Y_i, \bm Z_i),(\bm X_j, \bm X_j^\prime, \bm Y_j, \bm Z_j)\big)-g\big((\Tilde{\bm X}_i, \Tilde{\bm X}_i^\prime, \Tilde{\bm Y}_i, \Tilde{\bm Z}_i), (\bm X_j, \bm X_j^\prime, \bm Y_j, \bm Z_j)\big)\Big|\Bigg\} \\
    \end{align*}
    Since $|g(\cdot,\cdot)|\leq 2$, this implies 
        $|\hat\zeta_{n,\sigma}-\hat\zeta_{n,\sigma}^{(i)}| \leq \frac{8(n-1)}{n(n-1)} \leq \frac{8}{n}$. So, applying bounded difference inequality \citep[see page 37 in][]{wainwright2019} one has
    \begin{align*}
        \P\{\big|\hat\zeta_{n,\sigma} - \zeta_\sigma(\Pr)\big|> \epsilon\} \leq \exp\Big\{-\frac{2\epsilon^2}{\sum_{i=1}^n \frac{64}{n^2}}\Big\} = \exp\Big\{-\frac{n\epsilon^2}{32}\Big\}.
    \end{align*}
    This completes the proof.
\end{proof}

\begin{proof}[\bf Proof of Proposition \ref{thm:asymp-null}:]
    Note that the estimator $\hat\zeta_{n,\sigma}$ is a U-statistic with the kernel
       \begin{align*}
        g\big((\bm x_1, \bm x_1^\prime, \bm y_1, \bm z_1),(\bm x_2, \bm x_2^\prime, \bm y_2, \bm z_2)\big) = &  \exp\bigg\{-\frac{\sigma^2\|\bm v_1-\bm v_2\|^2}{2}\bigg\}+\exp\bigg\{-\frac{\sigma^2\|\bm v^{'}_1-\bm v^{'}_2\|^2}{2}\bigg\} \\
        & \hspace{0.5in}-\exp\bigg\{-\frac{\sigma^2\|\bm v_1-\bm v^{'}_2\|^2}{2}\bigg\}-\exp\bigg\{-\frac{\sigma^2\|\bm v^{'}_2-\bm v_1\|^2}{2}\bigg\},
    \end{align*}
    of degree 2, where $\bm v_i = (\bm x_i, \bm y_i, \bm z_i)$ and $\bm v^{'}_i = (\bm x_i^\prime, \bm y_i, \bm z_i)$ for $i = 1,2$. The first order Hoeffding projection of $g(.,.)$ is 
    \begin{align}
      g_1\big(\bm x_1,\bm x_1^\prime, \bm y_1, \bm z_1\big) = \E\Big\{K({\bm v}_1,{\bm V}_2)\Big\}+\E\Big\{K(\bm v_1^\prime,{\bm V}_2^\prime)\Big\} -\E\Big\{K({\bm v}_1,{\bm V}_2^\prime)\Big\}-\E\Big\{K({\bm v}_1^\prime,{\bm V}_2)\Big\}
      \label{first-projection}
    \end{align}
    where $K({\bf x},{\bf y})= \exp\{-\sigma^2\|{\bf x}-{\bf y}\|^2/2\}$, $\bm V_1 = (\bm X_1,\bm Y_1, \bm Z_1)$ and ${\bm V}_1^\prime = (\bm X_1^\prime, \bm Y_1, \bm Z_1)$. Under $H_0$, ${\bm V}_1$ and ${\bm V}_1^\prime$ are identically distributed and hence $g_1\big(\bm x_1, \bm x_1^\prime, \bm y_1, \bm z_1\big) = 0$. Therefore, using Theorem 1 from \cite{lee2019u} p.79, one gets that $n\hat\zeta_{n,\sigma}$ converges in distribution to $\sum_{i = 1}^\infty \lambda_i (Z_i^2-1)$ where $\{Z_i\}$ are independent standard normal random variables and $\{\lambda_i\}$ are the eigenvalues of the integral equation 
    $$\int g\big((\bm x_1, \bm x_1^\prime, \bm y_1, \bm z_1),(\bm x_2, \bm x_2^\prime, \bm y_2, \bm z_2)\big) f\big(\bm x_2, \bm x_2^\prime, \bm y_2, \bm z_2\big) \mathrm dF\big(\bm x_2, \bm x_2^\prime, \bm y_2, \bm z_2\big) = \lambda f\big(\bm x_1, \bm x_1^\prime, \bm y_1, \bm z_1\big),$$
    where $F$ is the joint distribution of $(\bm X_1, \bm X_1^\prime, \bm Y_1, \bm Z_1)$. 
    
\end{proof}

\begin{proof}[\bf Proof of Proposition \ref{thm:asymp-alt}:]
    Under $H_1$, the function $g_1\big(\bm x, \bm x^\prime, \bm y, \bm z\big)$ defined in equation \eqref{first-projection} is non-degenerate. Therefore, using Theorem 1 from \cite{lee2019u} p.76, one gets the asymptotic normality of $\sqrt{n}(\hat\zeta_{n,\sigma}-\zeta_\sigma(\Pr))$ with the mean zero and the variance $4\sigma_1^2$, where $\sigma_1^2 = \var\big(g_1\big(\bm X_1, \bm X_1^\prime, \bm Y_1,\bm Z_1\big)\big)$.
\end{proof}

\begin{proof}[\bf Proof of Proposition \ref{prop:level-property}:]
    Define $({\bm X}_1^\prime,\bm Y_1,\bm Z_1),({\bm X}_2^\prime,\bm Y_2, \bm Z_2),\ldots, ({\bm X}_n^\prime,\bm Y_n, \bm Z_n)$ as in Proposition 2.3. Under $H_0$, $({\bm X}_i,{\bm X}_i^\prime, \bm Y_i, \bm Z_i)$ and $({\bm X}_i^\prime, {\bm X}_i, \bm Y_i, \bm Z_i)$ are identically distributed for each $i=1,2,\ldots, n$. Therefore, for any fixed $n$, the joint distribution of $(\hat\zeta_{n,\sigma},\hat c_{1-\alpha})$ is identical to the joint distribution of $(\hat\zeta_{n,\sigma}(\pi),\hat c_{1-\alpha})$, where $\pi$ is an element of $\{0,1\}^n$. Let $\mathcal{D}^\prime = \{({\bm X}_i,{\bm X}_i^\prime, \bm Y_i, \bm Z_i)\mid i=1,2,\ldots, n\}$ denote the augmented data. Then,
    $$\P\{\hat\zeta_{n,\sigma}>\hat c_{1-\alpha}\} = \P\{\hat\zeta_{n,\sigma}(\pi)>\hat c_{1-\alpha}\} = \E\big\{\P\{\hat\zeta_{n,\sigma}(\pi)>\hat c_{1-\alpha}\mid \mathcal{D}^\prime\}\big\}\leq \alpha.$$
    The last inequality follows from the fact that $\hat c_{1-\alpha}$ is the $(1-\alpha)$th quantile of $\hat\zeta_{n,\sigma}(\pi)|\mathcal{D}^\prime$, as given in the resampling algorithm A-C.
\end{proof}

\begin{proof}[\bf Proof of Proposition \ref{prop:cutoff-bound}:]
    Here $({\bm X}_1,\bm Y_1,\bm Z_1),({\bm X}_2,\bm Y_2, \bm Z_2),\ldots, ({\bm X}_n,\bm Y_n, \bm Z_n)$ are independent copies of $(\bm X, \bm Y, \bm Z)\sim \Pr$. Define $({\bm X}_1^\prime,\bm Y_1,\bm Z_1),({\bm X}_2^\prime,\bm Y_2, \bm Z_2),\ldots, ({\bm X}_n^\prime,\bm Y_n, \bm Z_n)$ as in Proposition 2.3, and $\mathcal{D}^\prime$ denote the augmented dataset $\{(\bm X_i,\bm X_i^\prime, \bm Y_i,\bm Z_i)|i=1,2,\ldots, n\}$. Let $\pi=(\pi(1),\pi(2),\ldots,\pi(n))$ be uniformly distributed on the set $\{0,1\}^n$. For $i=1,2,\ldots,n$, define ${\bm U}_i = \pi(i) {\bm X}_i + (1-\pi(i)) {\bm X}_i^\prime$ and ${\bm U}_i^\prime = (1-\pi(i)) {\bm X}_i + \pi(i) {\bm X}_i^\prime$. Since $$g\big(({\bm U}_i,{\bm U}_i^\prime, \bm Y_i, \bm Z_i),({\bm U}_i,{\bm U}_i^\prime, \bm Y_i,\bm Z_i)\big)\geq 0$$
    for all $i=1,2,\ldots,n$, one gets

    \begin{align*}
        n(n-1)\hat\zeta_{n,\sigma}(\pi) = & \sum_{1\leq i\not=j\leq n} g\big(({\bm U}_i,{\bm U}_i^\prime, \bm Y_i, \bm Z_i),({\bm U}_j,{\bm U}_j^\prime, \bm Y_j,\bm Z_j)\big) \\
        & \leq \sum_{1\leq i,j\leq n} g\big(({\bm U}_i,{\bm U}_i^\prime, \bm Y_i, \bm Z_i),({\bm U}_j,{\bm U}_j^\prime, \bm Y_j,\bm Z_j)\big) = n^2 \zeta_\sigma(F_n),
    \end{align*}
    where $F_n$ denotes the empirical probability distribution of $({\bm U}_1,{\bm U}_1^\prime,\bm Y_1, \bm Z_1),\ldots, ({\bm U}_n,{\bm U}_n^\prime, \bm Y_n, \bm Z_n)$. Since, $\zeta_\sigma(F_n)$ is a non-negative random variable, using Markov inequality, for any $\epsilon>0$, one has
    $$\P\big\{\hat\zeta_{n,\sigma}(\pi)>\epsilon\mid \mathcal{D}^\prime\big\} \leq \P\big\{n^2\zeta_\sigma(F_n)>n(n-1)\epsilon\mid \mathcal{D}^\prime\big\} \leq \frac{n^2}{n(n-1)\epsilon} \E\big\{\zeta_\sigma(F_n)\mid \mathcal{D}^\prime\big\}.$$
    Now, taking $\epsilon=  \frac{n}{(n-1)\alpha}\E\big\{\zeta_\sigma(F_n)\mid \mathcal{D}^\prime\big\}$, we get
    $\P\big\{\hat\zeta_{n,\sigma}(\pi)>\epsilon\mid \mathcal{D}^\prime\big\}\leq \alpha.$ Therefore, from the definition of $\hat c_{1-\alpha}$ (see Resampling Algorithm A-C), one has
    $$\hat c_{1-\alpha} \leq \frac{n}{(n-1)\alpha}\E\big\{\zeta_\sigma(F_n)\mid \mathcal{D}^\prime\big\}.$$
    Also, note that
    \begin{align*}
        \E\{\zeta_\sigma(F_n)\mid \mathcal{D}^\prime\} = \frac{1}{n^2}\Bigg[& \sum_{1\leq i\not=j\leq n} \E\Big\{g\big(({\bm U}_i,{\bm U}_i^\prime, \bm Y_i, \bm Z_i),({\bm U}_j,{\bm U}_j^\prime, \bm Y_j,\bm Z_j)\big)\mid \mathcal{D}^\prime\Big\} \\
        & + \left.\sum_{1\leq i\leq n} \E\Big\{g\big(({\bm U}_i,{\bm U}_i^\prime, \bm Y_i, \bm Z_i),({\bm U}_i,{\bm U}_i^\prime, \bm Y_i,\bm Z_i)\big)\mid \mathcal{D}^\prime\Big\}\right].
    \end{align*}
    Now, for any $1\leq i\not=j\leq n$,
    \begin{align*}
      & \E\Big\{\exp\{-\frac{\sigma^2\|({\bm U}_i,{\bm U}_i^\prime, \bm Y_i, \bm Z_i)-({\bm U}_j,{\bm U}_j^\prime, \bm Y_j,\bm Z_j)\|^2}{2}\}\mid \mathcal{D}^\prime\Big\}\\
        & \hspace{0.25in}= \frac{1}{4}\sum_{\pi(1),\pi(2)\in\{0,1\}^2} \exp\Big\{-\frac{\sigma^2\|\pi(1){\bm V}_i+(1-\pi(1)){\bm V}_i^\prime-\pi(2){\bm V}_j-(1-\pi(2)){\bm V}_j^\prime\|^2}{2}\Big\} =\delta_n(i,j) \text{ (say)}.
    \end{align*}
    Similarly, one can also show that 
    \begin{align*}
        \E\Big\{\exp\{-\frac{\sigma^2\|{\bm V}_i^\prime-{\bm V}_j^\prime\|^2}{2}\}\mid \mathcal{D}^\prime\Big\} = & ~\E\Big\{\exp\{-\frac{\sigma^2\|{\bm V}_i-{\bm V}_j^\prime\|^2}{2}\}\mid \mathcal{D}^\prime\Big\}\\
        =&~\E\Big\{\exp\{-\frac{\sigma^2\|{\bm V}_i^\prime-{\bm V}_j\|^2}{2}\}\mid \mathcal{D}^\prime\Big\} = \delta_n(i,j).
    \end{align*}
    Hence,
    $\E\Big\{g\big(({\bm U}_i,{\bm U}_i^\prime, \bm Y_i, \bm Z_i),({\bm U}_j,{\bm U}_j^\prime, \bm Y_j,\bm Z_j)\big)\mid \mathcal{D}^\prime\Big\} = \delta_n(i,j)+\delta_n(i,j)-\delta_n(i,j)-\delta_n(i,j) = 0.$ 
    Similarly, one can also show that for any $1\leq i\leq n$,
    \begin{align*}
        & \E\Big\{g\big(({\bm U}_i,{\bm U}_i^\prime, \bm Y_i, \bm Z_i),({\bm U}_i,{\bm U}_i^\prime, \bm Y_i,\bm Z_i)\big)\mid \mathcal{D}^\prime\Big\}\\
        & = 2\Big(1-\frac{1}{2}\sum_{\pi\in\{0,1\}} \exp\Big\{-\frac{\sigma^2\|\pi(1){\bm V}_i+(1-\pi(1)){\bm V}_i^\prime-(1-\pi(1)){\bm V}_i-\pi(1){\bm V}_i^\prime\|^2}{2}\Big\}\Big)\leq 2
    \end{align*}
    Combining these, one gets $\hat c_{1-\alpha}\leq 2(\alpha (n-1))^{-1}$. This completes the proof.
\end{proof}

\begin{proof}[\bf Proof of Proposition \ref{prop:test-consistency}:]
    Note that Proposition 2.3 ensures that $\hat\zeta_{n,\sigma}$ is a consistent estimator of $\zeta_\sigma(\Pr)$, where $\zeta_\sigma(\Pr)=0$ under $H_0$ and it is positive under $H_1$ (see Proposition 2.6). By Proposition 2.7, $\hat c_{1-\alpha}\stackrel{a.s.}{\rightarrow}0$ as $n\rightarrow\infty$. Combining these, we get under $H_1$, the power of the test ${\P}(\hat\zeta_{n,\sigma}>\hat c_{1-\alpha})$ converges to one as $n$ diverges to infinity. 
\end{proof}

\begin{proof}[\bf Proof of Proposition \ref{thm:monte-carlo-consistency}:]
Given the augmented data set $\mathcal{D}^\prime$, define the distribution functions
$$F(t) = \frac{1}{2^n}\Big\{\sum_{\pi\in \{0,1\}^n}I[\hat\zeta_{n,\sigma}(\pi)\leq t]\Big\}\text{ and }F_B(t) = \frac{1}{B}\left\{\sum_{i=1}^BI[\hat\zeta_{n,\sigma}(\pi_i)\leq t]\right\}.$$
Then,
\begin{equation*}
    \begin{split}
        \big|p_{n}-p_{n,B}\big| & = \bigg|\frac{1}{2^n}\bigg\{\sum_{\pi\in \{0,1\}^n}I[\hat\zeta_{n,\sigma}(\pi)\geq \hat\zeta_{n,\sigma}]\bigg\}-\frac{1}{B+1}\bigg\{\sum_{i=1}^BI[\hat\zeta_{n,\sigma}(\pi_i)\geq \hat\zeta_{n}]+1\bigg\}\bigg|\\
        & = \bigg|\frac{1}{2^n}\bigg\{\sum_{\pi\in \{0,1\}^n}I[\hat\zeta_{n,\sigma}(\pi)< \hat\zeta_{n,\sigma}]\bigg\}-\frac{1}{B+1}\bigg\{\sum_{i=1}^BI[\hat\zeta_{n,\sigma}(\pi_i)< \hat\zeta_{n,\sigma}]\bigg\}\bigg|\\
        &  = \big|F(\hat\zeta_{n}-)-\frac{B}{B+1}F_B(\hat\zeta_{n}-)\Big|\\
        & \leq \Big|F(\hat\zeta_{n,\sigma}-)-F_B(\hat\zeta_{n,\sigma}-)\Big| + \Big|\frac{F_B(\hat\zeta_{n,\sigma}-)}{B+1}\Big| \leq \sup_{t\in\R}\big|F(t)-F_B(t)\big|+\frac{1}{B+1}
    \end{split}
\end{equation*}

Therefore, conditioned on $\mathcal{D}^\prime$, the Dvoretzky-Keifer-Wolfwitz inequality (\cite{massart1990tight}) gives $$\P\{\sup_{t\in\R}|F(t)-F_B(t)|>\epsilon\}\leq 2e^{-2B\epsilon^2}.$$ Now, using simple algebraic calculations one gets our desired result.
\end{proof}

\begin{proof}[\bf Proof of Proposition \ref{prop:lanfd}:]
    It is easy to see that the likelihood ratio of $F_{1-\frac{\beta_n}{\sqrt{n}}} = (1-\beta_n/\sqrt{n})G+\beta_n/\sqrt{n}F$ and $G$ is $\Big(1+\frac{\beta_n}{\sqrt{n}}\big(f({\bf u})/g({\bf u})-1\big)\Big)$. Hence if $({\bm X}_1,\bm Y_1, \bm Z_1),\ldots,({\bm X}_n,\bm Y_n, \bm Z_n)$ are i.i.d. observations from $G$, then the log-likelihood ratio is given by,
    \begin{align*}
        L_N = \log\Big\{\prod_{i=1}^n\frac{dF_{1-\frac{\beta_n}{\sqrt{n}}}}{dG}({\bm X}_i,\bm Y_i, \bm Z_i)\Big\} & = \sum_{i=1}^n \log\Big\{\frac{dF_{1-\frac{\beta_n}{\sqrt{n}}}}{dG}({\bm X}_i,\bm Y_i, \bm Z_i)\Big\}\\
        & = \sum_{i=1}^n \log\Big(1+\frac{\beta_n}{\sqrt{n}}\big(f({\bm X}_i,\bm Y_i, \bm Z_i)/g({\bm X}_i,\bm Y_i, \bm Z_i)-1\big)\Big).
    \end{align*}
    Using the fact that
    $\log(1+y) = y - \frac{y^2}{2}+\frac{1}{2}y^2h(y)$ where $h(\cdot)$ is continuous and $\lim\limits_{y\to0}h(y)=0,$ one gets 
    \begin{align*}
        L_N = & ~\sum_{i=1}^n \frac{\beta_n}{\sqrt{n}}\big(f({\bm X}_i,\bm Y_i, \bm Z_i)/g({\bm X}_i,\bm Y_i, \bm Z_i)-1\big)-\sum_{i=1}^n \frac{\beta_n^2}{2n}\big(f({\bm X}_i,\bm Y_i, \bm Z_i)/g({\bm X}_i,\bm Y_i, \bm Z_i)-1\big)^2\\
        &~~~+\sum_{i=1}^n \frac{\beta_n^2}{2n}\big(f({\bm X}_i,\bm Y_i, \bm Z_i)/g({\bm X}_i,\bm Y_i, \bm Z_i)-1\big)^2h\Big(\frac{\beta_n}{\sqrt{n}}\big(f({\bm X}_i,\bm Y_i, \bm Z_i)/g({\bm X}_i,\bm Y_i, \bm Z_i)-1\big)\Big).
    \end{align*}

    Under the assumptions $\int \big(f(\bm x,\bm y,\bm z)/g(\bm x, \bm y, \bm z)-1\big)^2 g(\bm x, \bm y, \bm z) \mathrm d{\bm x}\mathrm d \bm y \mathrm d \bm z$ is finite and $\beta_n\rightarrow\beta$, as $n$ grows to infinity, one has
    $$\sum_{i=1}^n \frac{\beta_n^2}{n}\big(f({\bm X}_i,\bm Y_i, \bm Z_i)/g({\bm X}_i,\bm Y_i, \bm Z_i)-1\big)^2\stackrel{a.s.}{\to} \beta^2 \E\Big(\big(f({\bm X}_1,\bm Y_1, \bm Z_1)/g({\bm X}_1,\bm Y_1, \bm Z_1)-1\big)^2\Big).$$
    Hence, only need to show that 
    $$\sum_{i=1}^n \frac{\beta_n^2}{n}\big(f({\bm X}_i,\bm Y_i,\bm Z_i)/g({\bm X}_i,\bm Y_i,\bm Z_i)-1\big)^2h\Big(\frac{\beta_n}{\sqrt{n}}\big(f({\bm X}_i,\bm Y_i,\bm Z_i)/g({\bm X}_i,\bm Y_i,\bm Z_i)-1\big)\Big)\stackrel{\P}{\rightarrow} 0.$$
    Notice that
    \begin{align*}
        &\sum_{i=1}^n \frac{\beta_n^2}{n}\big(f({\bm X}_i,\bm Y_i, \bm Z_i)/g({\bm X}_i,\bm Y_i, \bm Z_i)-1\big)^2h\Big(\frac{\beta_n}{\sqrt{n}}\big(f({\bm X}_i,\bm Y_i, \bm Z_i)/g({\bm X}_i,\bm Y_i, \bm Z_i)-1\big)\Big)\\
        & \leq \max_{1\leq i\leq n}|h\Big(\frac{\beta_n}{\sqrt{n}}\big(f({\bm X}_i,\bm Y_i, \bm Z_i)/g({\bm X}_i,\bm Y_i, \bm Z_i)-1\big)\Big)|\sum_{i=1}^n \frac{\beta_n^2}{n}\big(f({\bm X}_i,\bm Y_i, \bm Z_i)/g({\bm X}_i,\bm Y_i, \bm Z_i)-1\big)^2.
    \end{align*}
    Therefore, it suffices to show that $\max_{1\leq i\leq N}|h\Big(\frac{\beta_n}{\sqrt{n}}\big(f({\bm X}_i,\bm Y_i, \bm Z_i)/g({\bm X}_i,\bm Y_i, \bm Z_i)-1\big)\Big)|$ converges to zero in probability, which follows if $\max_{1\leq i\leq N}|\frac{\beta_n}{\sqrt{n}}\big(f({\bm X}_i,\bm Y_i, \bm Z_i)/g({\bm X}_i,\bm Y_i, \bm Z_i)-1\big)|$ converges to zero in probability (as $\lim_{y\to 0} h(y) = 0$ and it is continuous). Note that,
    \begin{align*}
        & \P\Big\{\max_{1\leq i\leq n}\big|\frac{1}{\sqrt{n}}\big(f({\bm X}_i,\bm Y_i, \bm Z_i)/g({\bm X}_i,\bm Y_i, \bm Z_i)-1\big)\big|>\epsilon\Big\}\\
        & \leq\sum_{i=1}^n \P\Big\{\big|\frac{1}{\sqrt{n}}\big(f({\bm X}_i,\bm Y_i, \bm Z_i)/g({\bm X}_i,\bm Y_i, \bm Z_i)-1\big)\big|>\epsilon\Big\}\\
        & = n \P\Big\{\big|\frac{1}{\sqrt{n}}\big(f({\bm X}_1,\bm Y_1, \bm Z_1)/g({\bm X}_1,\bm Y_1, \bm Z_1)-1\big)\big|>\epsilon\Big\}\\
        & = n \E\Big\{I\big(\big|\frac{1}{\sqrt{n}}\big(f({\bm X}_1,\bm Y_1, \bm Z_1)/g({\bm X}_1,\bm Y_1, \bm Z_1)-1\big)\big|>\epsilon\big)\Big\}\\
        & \leq n \E\Big\{\frac{\big(f({\bm X}_1,\bm Y_1, \bm Z_1)/g({\bm X}_1,\bm Y_1, \bm Z_1)-1\big)^2}{n\epsilon^2}I\big|\frac{1}{\sqrt{n}}\big(f({\bm X}_1,\bm Y_1, \bm Z_1)/g({\bm X}_1,\bm Y_1, \bm Z_1)-1\big)\big|>\epsilon\Big\}\\
        & \leq \frac{1}{\epsilon^2}\E\Big\{\big(f({\bm X}_1,\bm Y_1, \bm Z_1)/g({\bm X}_1,\bm Y_1, \bm Z_1)-1\big)^2I\big(\big|\frac{1}{\sqrt{n}}\big(f({\bm X}_1,\bm Y_1, \bm Z_1)/g({\bm X}_1,\bm Y_1, \bm Z_1)-1\big)\big|>\epsilon\big)\Big\}.
    \end{align*}
    Since $I\big(\big|\frac{1}{\sqrt{n}}\big(f({\bm X}_1,\bm Y_1, \bm Z_1)/g({\bm X}_1,\bm Y_1, \bm Z_1)-1\big)\big|>\epsilon\big)$ converges to zero in probability, the right-hand side converges to zero by the Dominated Convergence Theorem. Hence, one has
    \begin{align*}
        \log\Big\{\prod_{i=1}^n\frac{dF_{1-\beta_n/\sqrt{n}}}{dG}({\bm X}_i,\bm Y_i, \bm Z_i)\Big\} = &~\frac{\beta_n}{\sqrt{n}}\sum_{i=1}^n \Big(f({\bm X}_i,\bm Y_i, \bm Z_i)/g({\bm X}_i,\bm Y_i, \bm Z_i)1\Big)\\
        &~~~~ -\frac{\beta_n^2}{2}\E\Big\{f({\bm X}_1,\bm Y_1, \bm Z_1)/g({\bm X}_1,\bm Y_1, \bm Z_1)-1\Big\}^2 + o_P(1).
    \end{align*}
    This completes the proof.
\end{proof}

\begin{proof}[\bf Proof of Theorem \ref{thm:local-limit}:]
    Let $\bm V_1 = ({\bm X}_1,\bm Y_1, \bm Z_1),\ldots,\bm V_n = ({\bm X}_n,\bm Y_n, \bm Z_n)$ i.i.d. observations and $\bm V_1^\prime = ({\bm X}_1^\prime,\bm Y_1, \bm Z_1),\ldots,\bm V_n^\prime = ({\bm X}_n,\bm Y_n, \bm Z_n)$ be their respective null exchangeable pairs. To prove this theorem, one only needs to find the limit distribution of $\sqrt{n}\big(\frac{1}{n}\sum_{i=1}^n h({\bm X}_i,{\bm X}_i^\prime,\bm Y_i, \bm Z_i) - \E\big\{h({\bm X}_1,{\bm X}_1^\prime, \bm Y_1, \bm Z_1)\big\}\big)$ for some square-integrable function $h$ when $\bm V_1\sim F_{1-\beta_nn^{-1/2}}= \big(1-\frac{\beta_n}{\sqrt{n}}\big) G + \frac{\beta_n}{\sqrt{n}} F$, where $F_{1-\beta_nn^{-1/2}}$ satisfies the condition of Proposition 3.1. Using the bivariate central limit theorem, one can say that as $n$ diverges to infinity, the joint distribution of 
    $$\sqrt{n}\big(\frac{1}{n}\sum_{i=1}^n h({\bm X}_i,{\bm X}_i^\prime, \bm Y_i, \bm Z_i) - \E\big\{h({\bm X}_1,{\bm X}_1^\prime, \bm Y_1, \bm Z_1)\big\}\big)$$
    and $$\frac{\beta_n}{\sqrt{n}}\sum_{i=1}^n \big(\frac{f({\bm X}_i, \bm Y_i, \bm Z_i)}{g({\bm X}_i, \bm Y_i, \bm Z_i)}-1\big)-\frac{\beta_n^2}{2}\E\Big\{\frac{f({\bm X}_1, \bm Y_1, \bm Z_1)}{g({\bm X}_1, \bm Y_1,\bm Z_1)}-1\Big\}^2$$
    converges to a bivariate normal distribution with mean\\
    $$\mu = \begin{pmatrix}0\\ -\frac{\beta^2}{2}\E\Big\{\frac{f({\bm X}_1, \bm Y_1, \bm Z_1)}{g({\bm X}_1, \bm Y_1, \bm Z_1)}-1\Big\}^2\end{pmatrix}$$ 
    and variance
    $$\Sigma = 
    \Biggl(
    \begin{tabular}{cc}
    $\var\big( h({\bm X}_1,{\bm X}_1^\prime, \bm Y_1, \bm Z_1)\big)$ & $\tau$\\ $\tau$  & $\beta^2\E\big\{\frac{f({\bm X}_1, \bm Y_1, \bm Z_1)}{g({\bm X}_1, \bm Y_1,\bm Z_1)}-1\big\}^2$
    \end{tabular}
     \Biggr),$$
    where 
    \begin{align*}
    \tau  & = \E\Big\{\big\{h({\bm X}_1,{\bm X}_1^\prime, \bm Y_1,\bm Z_1) - \E\{h({\bm X}_1,{\bm X}_1^\prime, \bm Y_1,\bm Z_1)\}\big\}\beta\big\{\frac{f({\bm X}_1, \bm Y_1, \bm Z_1)}{g({\bm X}_1, \bm Y_1,\bm Z_1)}-1\big\}\Big\}\\
    & = \beta\int \big\{h\big({\bm x}, \bm x^\prime, \bm y, \bm z)) - \E\{h({\bm X}_1,{\bm X}_1^\prime, \bm Y_1, \bm Z_1)\}\big\}\big(f({\bm x},\bm y, \bm z)-g({\bm x}, \bm y, \bm z)\big) \mathrm d{\bm x} \mathrm d\mu_{\bm z}({\bm x^\prime}) \mathrm d{\bm y}  \mathrm d{\bm z},    
    \end{align*}
    $\mu_{\bm z}$ being the distribution of $\bm X|\bm Z=\bm z$. Now using Le Cam's third lemma \citep[see][]{van2000asymptotic}, as $n$ diverges to infinity, under $F_{1-\beta_n/\sqrt{n}}$, one has
    $$\sqrt{n}\big(\frac{1}{n}\sum_{i=1}^n h({\bm X}_i,{\bm X}_i^\prime, \bm Y_i,\bm Z_i) - \E\big\{h({\bm X}_1,{\bm X}_1^\prime, \bm Y_1,\bm Z_1)\big\}\big)\stackrel{D}{\to} N\Big(\tau, \var\big(h({\bm X}_1,{\bm X}_1^\prime, \bm Y_1, \bm Z_1)\big)\Big).$$
   Using standard theory of degenerate U-statistics, contiguity theory and approximation arguments \citep[see][]{lee2019u}, one gets
     \begin{align*}
       n\hat\zeta_{n,\sigma}\stackrel{D}{\to} & \sum_{k=1}^\infty \lambda_k \left(\big(U_k+\beta~\E_F\{\psi_k({\bm X}_1,{\bm X}_1^\prime, \bm Y_1, \bm Z_1)\}\big)\Big)^2-1\right)
    \end{align*}
    under $F_{1-\beta_n/\sqrt{n}}$, where $\{U_i\}$ is a sequence of i.i.d. normal random variables, $\lambda_k$s and $\psi_k$s are the eigenvalues and eigenfunctions of the integral equation 
$$\E\left\{g\big((\bm x, \bm x^\prime, \bm y, \bm z),(\bm X, \bm X^\prime, \bm Y, \bm Z)\big)\psi(\bm X, \bm X^\prime, \bm Y, \bm Z)\right\} = \lambda \psi(\bm x, \bm x^\prime, \bm y, \bm z),$$
where $g(\cdot, \cdot)$ is as defined in (5) and $\bm V = (\bm X, \bm Y, \bm Z)\sim G$ and $\bm V^\prime$ is the null exchangeable pair of $\bm V$. This completes the proof.
\end{proof}

{
\begin{lemma}
    Let $\bm V_1 = ({\bm X}_1,\bm Y_1, \bm Z_1),\ldots,\bm V_n = ({\bm X}_n,\bm Y_n, \bm Z_n)$ be independent copies of $\bm V \sim \Pr$ and $\bm V_1^\prime, \ldots, \bm V_n^\prime$ be their null exchangeable pairs. Let $\mathcal{D}^\prime = \{(\bm X_i, \bm X_i^\prime, \bm Y_i, \bm Z_i)\}_{1\le i\leq n}$ be the augmented data set. For $i=1,2,\ldots,n$, define $\bm U_i = \pi_i \bm X_i + (1-\pi_i)\bm X_i^\prime$ and $\bm U_i^\prime = \pi_i \bm X_i^\prime + (1-\pi_i)\bm X_i$, where $\pi_1,\ldots, \pi_n\stackrel{iid}{\sim} \textnormal{Bernoulli}(0.5)$. Let $f(\cdot)$ be a function such that $f(\bm x, \bm x^\prime, \bm y, \bm z) = -f(\bm x^\prime, \bm x, \bm y, \bm z)$ and $\E[f(\bm X_1,\bm X_1^\prime, \bm Y_1, \bm Z_1)^2]$ is finite. Then, the conditional distribution of $n^{-1/2}\sum_{i=1}^n f(\bm U_i,\bm U_i^\prime, \bm Y_i, \bm Z_i)$ given $\mathcal{D}^\prime$ converges weakly to a $\mathcal{N}_1(0,\E[f^2(\bm U_1,\bm U_1^\prime, \bm Y_1, \bm Z_1)])$ in probability as $n$ grows to infinity. 
    \label{conditional-CLT}
\end{lemma}

\begin{proof}[\bf Proof of Lemma \ref{conditional-CLT}:]
    Since $\E[f(\bm U_1, \bm U_1^\prime, \bm Y_1, \bm Z_1)^2]$ is finite, both $\E[f(\bm X_1, \bm X_1^\prime, \bm Y_1, \bm Z_1)^2]$ and \linebreak $\E[f(\bm X_1^\prime, \bm X_1, \bm Y_1, \bm Z_1)^2]$ are also finite. By $f(\bm x,\bm x^\prime, \bm y, \bm z) = - f(\bm x^\prime, \bm x, \bm y, \bm z)$, one also has $\E[f(\bm U_1, \bm U_1^\prime, \bm Y_1, \bm Z_1)\mid \mathcal{D}^\prime]$ = 0. Now, note that conditioned on $\mathcal{D}^\prime$, $\{f(\bm U_j,\bm U_j^\prime, \bm Y_j, \bm Z_j)\}_{1\leq j\leq n}$ form a triangular array of independent random variables. Therefore, one can use arguments similar to the Lindeberg's CLT to prove this result. So, first define
    \begin{align*}
        S_n^2 & = \sum_{j=1}^n \var[f(\bm U_j, \bm U_j^\prime, \bm Y_j, \bm Z_j)\big|\mathcal{D}^\prime]\\
        & = \sum_{j=1}^n \E[f(\bm U_j, \bm U_j^\prime, \bm Y_j, \bm Z_j)^2\big|\mathcal{D}^\prime]= \sum_{j=1}^n \frac{1}{2} f(\bm X_j, \bm X_j^\prime, \bm Y_j, \bm Z_j)^2 + \sum_{j = 1}^n \frac{1}{2} f(\bm X_j^\prime, \bm X_j, \bm Y_j, \bm Z_j)^2.
    \end{align*}
    So, by the strong law of large numbers, one gets
    \begin{align*}
        \frac{S_n^2}{n} \stackrel{a.s.}{\rightarrow}\frac{1}{2}\E[f(\bm X_1, \bm X_1^\prime, \bm Y_1, \bm Z_1)^2] + \frac{1}{2} \E[f(\bm X_1^\prime, \bm X_1, \bm Y_1, \bm Z_1)^2] = \E[f(\bm U_1, \bm U_1^\prime, \bm Y_1, \bm Z_1)^2] \mbox{    as } n \rightarrow \infty.
    \end{align*}
     Now, using the inequality 
    %\begin{align*}
       $\Big|e^{itx} - \big(1 + itx - \frac{1}{2} t^2 x^2\big)\Big| \leq \min\Big\{\big|tx\big|^2, \big|tx\big|^3\Big\}$,
    %\end{align*}
    one gets
    \begin{align*}
        \Big|\E\big\{e^{it\frac{1}{\sqrt{n}}f({\bm U}_j,{\bm U}_j^\prime, \bm Y_j, \bm Z_j)} \big|\mathcal{D} \big\} - &\big(1-\frac{t^2}{2n} \E\big\{f^2(\bm U_j,\bm U_j^\prime, \bm Y_j, \bm Z_j)\big|\mathcal{D}^\prime\big\}\big)\Big|\\
        & \leq \E\Bigg\{\min\Big\{\big|\frac{t}{\sqrt{n}}f({\bm U}_j,{\bm U}_j^\prime, \bm Y_j, \bm Z_j)\big|^2, \big|\frac{t}{\sqrt{n}}f({\bm U}_j,{\bm U}_j^\prime, \bm Y_j, \bm Z_j)\big|^3\Big\} \big|\mathcal{D}^\prime\Bigg\}.
    \end{align*}
    Note that the expectation on the right hand side of the above inequality is finite. Now take any arbitrarily small $\epsilon>0$ and note that
    \begin{align*}
        & \E\Bigg\{\min\Big\{\big|\frac{t}{\sqrt{n}}f({\bm U}_j,{\bm U}_j^\prime, \bm Y_j, \bm Z_j)\big|^2, \big|\frac{t}{\sqrt{n}}f({\bm U}_j,{\bm U}_j^\prime, \bm Y_j, \bm Z_j)\big|^3\Big\} \big|\mathcal{D}^\prime\Bigg\}\\
        & = \E\Bigg\{\min\Big\{\big|\frac{t}{\sqrt{n}}f({\bm U}_j,{\bm U}_j^\prime, \bm Y_j, \bm Z_j)\big|^2, \big|\frac{t}{\sqrt{n}}f({\bm U}_j,{\bm U}_j^\prime, \bm Y_j, \bm Z_j)\big|^3\Big\} I[|f({\bm U}_j,{\bm U}_j^\prime, \bm Y_j, \bm Z_j)|<\epsilon S_n] \big|\mathcal{D}^\prime\Bigg\}\\
        & ~~~~~ + \E\Bigg\{\min\Big\{\big|\frac{t}{\sqrt{n}}f({\bm U}_j,{\bm U}_j^\prime, \bm Y_j, \bm Z_j)\big|^2, \big|\frac{t}{\sqrt{n}}f({\bm U}_j,{\bm U}_j^\prime, \bm Y_j, \bm Z_j)\big|^3\Big\} I[|f({\bm U}_j,{\bm U}_j^\prime, \bm Y_j, \bm Z_j)|\geq \epsilon S_n] \big|\mathcal{D}^\prime\Bigg\}\\
        & \leq \frac{\epsilon S_n |t|^3}{n^{3/2}} \E\Big\{ f({\bm U}_j,{\bm U}_j^\prime, \bm Y_j, \bm Z_j)^2\big|\mathcal{D}^\prime\Big\} \\
        & ~~~~~~~~~~~~ + \frac{t^2}{2n}\Big\{f({\bm X}_j,{\bm X}_j^\prime, \bm Y_j, \bm Z_j)^2 I[|f({\bm X}_j,{\bm X}_j^\prime, \bm Y_j, \bm Z_j)|\geq \epsilon S_n]\\
        & ~~~~~~~~~~~~~~~~~~~~~~~ + f({\bm X}_j^\prime,{\bm X}_j, \bm Y_j, \bm Z_j)^2 I[|f({\bm X}_j^\prime,{\bm X}_j, \bm Y_j, \bm Z_j)|\geq \epsilon S_n]\Big\}.
    \end{align*}
    Then, by the assumption $\E[f({\bm U}_j,{\bm U}_j^\prime, \bm Y_j, \bm Z_j)^2]<\infty$, one gets
    \begin{align*}
        & \sum_{j=1}^n \Big|\E[e^{it \frac{1}{\sqrt{n}}f({\bm U}_j,{\bm U}_j^\prime, \bm Y_j, \bm Z_j)}\big|\mathcal{D}^\prime] - \big(1-\frac{t^2}{2n} \E[f({\bm U}_j,{\bm U}_j^\prime, \bm Y_j, \bm Z_j)^2\big|\mathcal{D}^\prime]\big)\Big|\\
        & \leq \frac{\epsilon S_n |t|^3}{\sqrt{n}} \Big\{\frac{1}{2n}\sum_{j=1}^n f({\bm X}_j,{\bm X}_j^\prime, \bm Y_j, \bm Z_j)^2 + \frac{1}{2n}\sum_{j=1}^n f({\bm X}_j^\prime,{\bm X}_j, \bm Y_j, \bm Z_j)^2\Big\}\\
        &~~~~~~ + \frac{t^2}{2} \Big\{\frac{1}{n} \sum_{j=1}^n f({\bm X}_j,{\bm X}_j^\prime, \bm Y_j, \bm Z_j)^2 I[|f({\bm X}_j,{\bm X}_j^\prime, \bm Y_j, \bm Z_j)|\geq \epsilon S_n]\\
        &~~~~~~~~~~~~~~+ \frac{1}{n} \sum_{j=1}^n f({\bm X}_j^\prime,{\bm X}_j, \bm Y_j, \bm Z_j)^2 I[|f({\bm X}_j^\prime,{\bm X}_j, \bm Y_j, \bm Z_j)|\geq \epsilon S_n]\Big\}.
    \end{align*}
    Now, by the strong law of large numbers, one gets
    \begin{align*}
        & \frac{\epsilon S_n |t|^3}{\sqrt{n}} \Big\{\frac{1}{2n}\sum_{j=1}^n f({\bm X}_j,{\bm X}_j^\prime, \bm Y_j, \bm Z_j)^2 + \frac{1}{2n}\sum_{j=1}^n f({\bm X}_j^\prime,{\bm X}_j, \bm Y_j, \bm Z_j)^2\Big\}\\
        & \stackrel{a.s.}{\rightarrow} \epsilon |t|^3 \sqrt{\E[f({\bm U}_1, {\bm U}_1^\prime, \bm Y_1, \bm Z_1)^2]} \Big\{\frac{1}{2} \E[f({\bm X}_1,{\bm X}_1^\prime, \bm Y_1, \bm Z_1)^2] + \frac{1}{2} \E[f({\bm X}_1^\prime,{\bm X}_1, \bm Y_1, \bm Z_1)^2]\Big\}\\
        & = \epsilon |t|^3 \E[f({\bm U}_1, {\bm U}_1^\prime, \bm Y_1, \bm Z_1)^2]^{3/2},
    \end{align*}
    which is arbitrarily small (since $\epsilon$ is arbitrarily small). Again, using DCT, one has
    \begin{align*}
        & \E\left[\frac{1}{n} \sum_{j=1}^n f({\bm X}_j,{\bm X}_j^\prime, \bm Y_j, \bm Z_j)^2 I[|f({\bm X}_j,{\bm X}_j^\prime, \bm Y_j, \bm Z_j)|\geq \epsilon S_n] \right]\\
        & = \E\left[f({\bm X}_1,{\bm X}_1^\prime, \bm Y_j, \bm Z_j)^2 I[|f({\bm X}_1,{\bm X}_1^\prime, \bm Y_1, \bm Z_1)|\geq \epsilon S_n] \right] \rightarrow 0
    \end{align*}
    as $n$ grows to infinity. Therefore, $\frac{1}{n} \sum_{j=1}^n f({\bm X}_j,{\bm X}_j^\prime, \bm Y_j, \bm Z_j)^2 I[|f({\bm X}_j,{\bm X}_j^\prime, \bm Y_j, \bm Z_j)|\geq \epsilon S_n]$ converges to zero in probability as $n$ grows to infinity. Similarly, one can also show that $$\frac{1}{n} \sum_{j=1}^n f({\bm X}_j^\prime,{\bm X}_j, \bm Y_j, \bm Z_j)^2 I[|f({\bm X}_j^\prime,{\bm X}_j, \bm Y_j, \bm Z_j)|\geq \epsilon S_n]\stackrel{P}{\rightarrow} 0$$ 
    as $n$ grows to infinity. Therefore, by repeated application of the triangle inequality, one gets
    \begin{align*}
        & \left|\prod_{j=1}^n \E\Big[e^{it \frac{1}{\sqrt{n}} f({\bm U}_j,{\bm U}_j^\prime, \bm Y_j, \bm Z_j) }\big|\mathcal{D}^\prime\Big] - \prod_{j=1}^n \big(1-\frac{t^2}{2n} \E\big[f({\bm U}_j, {\bm U}_j^\prime, \bm Y_j, \bm Z_j)^2\big|\mathcal{D}^\prime\big]\big)\right|\\
        & \leq \sum_{j=1}^n \Big|\E\Big[e^{it \frac{1}{\sqrt{n}} f({\bm U}_j,{\bm U}_j^\prime, \bm Y_j, \bm Z_j) }\big|\mathcal{D}^\prime\Big] - \big(1-\frac{t^2}{2n} \E\big[f({\bm U}_j, {\bm U}_j^\prime, \bm Y_j, \bm Z_j)^2\big|\mathcal{D}^\prime\big]\big)\Big| \stackrel{P}{\rightarrow} 0,
    \end{align*}
    as $n$ grows to infinity. Before proceeding further let us note that for any $\delta>0$, 
    \begin{align*}
        &\Big|\frac{1}{n} \E\big[f({\bm U}_j,{\bm U}_j^\prime, \bm Y_j, \bm Z_j)^2\big|\mathcal{D}^\prime\big]\Big| \\
        & \leq \delta^2 +\frac{1}{n}\E\big[f({\bm U}_j, {\bm U}_j^\prime, \bm Y_j, \bm Z_j)^2 I[|f({\bm U}_j, {\bm U}_j^\prime, \bm Y_j, \bm Z_j)|\geq \delta\sqrt{n}]\big|\mathcal{D}^\prime\big]\\
        & = \delta^2 + \frac{1}{2n}f({\bm X}_j, {\bm X}_j^\prime, \bm Y_j, \bm Z_j)^2 I[|f({\bm X}_j, {\bm X}_j^\prime, \bm Y_j, \bm Z_j)|\geq \delta\sqrt{n}]\\
        &~~~~~~~~~~~~+\frac{1}{2n}f({\bm X}_j^\prime, {\bm X}_j, \bm Y_j, \bm Z_j)^2 I[|f({\bm X}_j^\prime, {\bm X}_j, \bm Y_j, \bm Z_j)|\geq \delta\sqrt{n}].
    \end{align*}
    \begin{align*}
        \text{So,  } &
    \max_{1\leq j\leq n} \Big|\frac{1}{n} \E\big[f({\bm U}_j, {\bm U}_j^\prime, \bm Y_j, \bm Z_j)^2\big|\mathcal{D}^\prime\big]\Big|\\
        & \leq \delta^2 + \frac{1}{2n}\sum_{j=1}^n f({\bm X}_j, {\bm X}_j^\prime, \bm Y_j, \bm Z_j)^2 I[|f({\bm X}_j, {\bm X}_j^\prime, \bm Y_j, \bm Z_j)|\geq \delta\sqrt{n}] \\
        &~~~~~~~~+ \frac{1}{2n}\sum_{j=1}^n f({\bm X}_j^\prime, {\bm X}_j, \bm Y_j, \bm Z_j)^2 I[|f({\bm X}_j^\prime, {\bm X}_j, \bm Y_j, \bm Z_j)|\geq \delta\sqrt{n}].
    \end{align*}
    Since, $\delta^2$ is of arbitrary and the other term is of order $o_p(1)$, one gets 
    $$\max_{1\leq j\leq n} \Big|\frac{1}{n} \E\big[f({\bm U}_j, {\bm U}_j^\prime, \bm Y_j, \bm Z_j)^2\big|\mathcal{D}^\prime\big]\Big|\stackrel{P}{\rightarrow} 0$$ as $n \rightarrow \infty$. Also note that
    \begin{align*}
        & \left|\prod_{j=1}^n \exp\Big\{-\frac{t^2}{2n} \big(\frac{1}{2} f({\bm X}_j, {\bm X}_j^\prime, \bm Y_j, \bm Z_j)^2 + \frac{1}{2} f({\bm X}_j^\prime, {\bm X}_j, \bm Y_j, \bm Z_j)^2\big)\Big\}\right. \\
        &\hspace{1in}- \left.\prod_{j=1}^n \Big\{1-\frac{t^2}{2n} \big(\frac{1}{2} f({\bm X}_j, {\bm X}_j^\prime, \bm Y_j, \bm Z_j)^2 + \frac{1}{2} f({\bm X}_j^\prime, {\bm X}_j, \bm Y_j, \bm Z_j)^2\big)\Big\}\right|\\
        & \leq \sum_{j=1}^n \left|\exp\Big\{-\frac{t^2}{2n} \big(\frac{1}{2} f({\bm X}_j, {\bm X}_j^\prime, \bm Y_j, \bm Z_j)^2 + \frac{1}{2} f({\bm X}_j^\prime, {\bm X}_j, \bm Y_j, \bm Z_j)^2\big)\Big\}\right. \\
        & \hspace{1in}\left.- \Big\{1-\frac{t^2}{2n} \big(\frac{1}{2} f({\bm X}_j, {\bm X}_j^\prime, \bm Y_j, \bm Z_j)^2 + \frac{1}{2} f({\bm X}_j^\prime, {\bm X}_j, \bm Y_j, \bm Z_j)^2\big)\Big\}\right|\\
        & \leq \sum_{j=1}^n \left|\frac{t^2}{4n} \big( f({\bm X}_j, {\bm X}_j^\prime, \bm Y_j, \bm Z_j)^2 + f({\bm X}_j^\prime, {\bm X}_j, \bm Y_j, \bm Z_j)^2\big)\right|^2 \\
        &~~~~~~~\times\exp\Big\{\frac{t^2}{4n} \big( f({\bm X}_j, {\bm X}_j^\prime, \bm Y_j, \bm Z_j)^2 + f({\bm X}_j^\prime, {\bm X}_j, \bm Y_j, \bm Z_j)^2\big)\Big\}\\
        &\leq \exp\Big\{\max_{1\leq j\leq n}\frac{t^2}{4n} \big( f({\bm X}_j, {\bm X}_j^\prime, \bm Y_j, \bm Z_j)^2 + f({\bm X}_j^\prime, {\bm X}_j, \bm Y_j, \bm Z_j)^2\big)\Big\}\\
        &~~~~~~~\times\max_{1\leq j\leq n}\frac{t^2}{4n} \big( f({\bm X}_j, {\bm X}_j^\prime, \bm Y_j, \bm Z_j)^2 + f({\bm X}_j^\prime, {\bm X}_j, \bm Y_j, \bm Z_j)^2\big)\\
        & ~~~~~~~~~~~~~~~~~~~~~\times\sum_{j=1}^n \frac{t^2}{4n} \big( f({\bm X}_j, {\bm X}_j^\prime, \bm Y_j, \bm Z_j)^2 + f({\bm X}_j^\prime, {\bm X}_j, \bm Y_j, \bm Z_j)^2\big).
    \end{align*}
    Clearly, this converges to $0$ as $n$ grows to infinity. Therefore,
    \begin{align}\label{eq:conditional-CLT}
        \prod_{j=1}^n \E\Big[e^{it \frac{1}{\sqrt{n}} f({\bm U}_j,{\bm U}_j^\prime, \bm Y_j, \bm Z_j) }\big|\mathcal{D}^\prime\Big] & = \exp\Big\{-\frac{t^2}{2} \frac{1}{n}\sum_{j=1}^n \big(f({\bm X}_j,{\bm X}_j^\prime, \bm Y_j, \bm Z_j)^2+f({\bm X}_j^\prime,{\bm X}_j, \bm Y_j, \bm Z_j)^2\big)\Big\} + o_P(1)\nonumber\\
        & = \exp\Big\{-\frac{t^2}{2} \E[f({\bm U}_1,{\bm U}_1^\prime, \bm Y_1, \bm Z_1)^2]\Big\} + o_P(1)
    \end{align}
    This completes the proof. 
\end{proof}
}

\begin{lemma}
    Let $F$ and $G$ be two probability distributions on $\R^d$ and assume $F_\delta = (1-\delta)F+\delta G$. Then, for any $\delta \in(0,1)$, we have 
    $$\zeta_\sigma(F_\delta) = (1-\delta)^2 \zeta_\sigma(F)+\delta^2 \zeta_\sigma(G) + 2\delta(1-\delta) \zeta_\sigma^\prime (G,F)$$
    where 
    $\zeta_\sigma^\prime(G,F) = \E\{K({\bm V}_1,{\bm V}_2)\}+\E\{K({\bm V}_1^\prime, {\bm V}_2^\prime)\}-\E\{K({\bm V}_1,{\bm V}_2^\prime)\}-\E\{K({\bm V}_1^\prime,{\bm V}_2)\},$ 
    ${\bm V}_1\sim G$ and ${\bf Y}_2\sim F$ are independent and ${\bm V}_1^\prime$ and ${\bm V}_2^\prime$ are their respective null exchangeable pairs and $K({\bf x},{\bf y}) = \exp\{-\frac{\sigma^2}{2}\|{\bf x}-{\bf y}\|^2\}$.
    \label{lem:relation-contamination}
\end{lemma}

\begin{proof}[\bf Proof of Lemma \ref{lem:relation-contamination}:]
    Recall that $\zeta_\sigma(\Pr)$ can be expressed as 
    $$\zeta_\sigma(\Pr) = \E\{K({\bm V}_1,{\bm V}_2)\}+\E\{K({\bm V}_1^\prime, {\bm V}_2^\prime)\}-2\E\{K({\bm V}_1,{\bm V}_2^\prime)\},$$
    where ${\bm V}_1 = (\bm X_1,\bm Y_1, \bm Z_1),{\bm V}_2 = (\bm X_2,\bm Y_2, \bm Z_2)$ are independent observations from $\Pr$, ${\bm V}_1^\prime, {\bm V}_2^\prime$ are their respective null exchangeable pairs, and $K({\bf x},{\bf y}) = \exp\big\{-\frac{\sigma^2}{2}\|{\bf x}-{\bf y}\|^2\big\}$. Now if {$F_\delta = (1-\delta) F+ \delta G$ }($0<\delta<1$), then
    \begin{align*}
        & \E\{K({\bm V}_1,{\bm V}_2)\} = \int K({\bm v}_1,{\bm v}_2) {\mathrm d}F_\delta({\bm v}_1) {\mathrm d}F_\delta({\bm v}_2)\\
        & = (1-\delta)^2 \int K({\bm v}_1,{\bm v}_2) {\mathrm d}F({\bm v}_1) {\mathrm d}F({\bm v}_2) + 2\delta (1-\delta) \int K({\bm v}_1,{\bm v}_2) {\mathrm d}G({\bm v}_1) {\mathrm d}F({\bm v}_2)\\
        & \hspace{2in}+\delta^2 \int K({\bm v}_1,{\bm v}_2) {\mathrm d}G({\bm v}_1) {\mathrm d}G({\bm v}_2).
    \end{align*}
    If $\mu_+$ denotes the distribution of $\bm V^\prime$ when $\bm V\sim F_\delta$ then $\mu_+ = (1-\delta)F^\prime + \delta G^\prime$, where $F^\prime$ the distribution of the null exchangeable pair of a sample from $F$ and $G^\prime$ is the same for a sample from $G$. Then
    \begin{align*}
        \E\{K({\bm V}_1^\prime,{\bm V}_2^\prime)\} = &\int K(\bm v^\prime_1,{\bm v^\prime}_2) \mathrm d\mu_+({\bm v^\prime}_1) \mathrm d\mu_+({\bm v^\prime}_2) \\
         = & ~(1-\delta)^2 \int K(\bm v^\prime_1,\bm v^\prime_2) \mathrm dF^\prime({\bm v^\prime}_1)\mathrm dF^\prime({\bm v^\prime}_2) + \delta^2 \int K(\bm v^\prime_1,\bm v^\prime_2) \mathrm dG^\prime({\bm v^\prime}_1)\mathrm dG^\prime({\bm v^\prime}_2) \\
        &+ 2\delta (1-\delta) \int K(\bm v^\prime_1,\bm v^\prime_2) \mathrm dG^\prime({\bm v^\prime}_1)\mathrm dF^\prime({\bm v^\prime}_2) 
    \end{align*}
    \begin{align*}
        \mbox{and }& \E\{K({\bm V}_1,{\bm V}_2^\prime)\} = \int K(\bm v_1,\bm v^\prime_2) \mathrm dF_{\delta}(\bm v_1)\mathrm d\mu_+({\bm v^\prime}_1) \\
        & = (1-\delta)^2 \int K(\bm v_1,\bm v^\prime_2) \mathrm dF({\bm v}_1)\mathrm dF^\prime({\bm v^\prime}_2) +  \delta (1-\delta) \int K({\bm v}_1,{\bm v^\prime}_2)\mathrm dG({\bm v}_1)\mathrm dF^\prime({\bm v^\prime}_2)\\
        & ~~~~~~~~+ \delta (1-\delta) \int K({\bm v}_1,{\bm v^\prime}_2) \mathrm dF({\bm v}_1)\mathrm dG^\prime({\bm v^\prime}_2) +\delta^2 \int K({\bm v}_1,{\bm v^\prime}_2) \mathrm dG({\bm v}_1) \mathrm dG^\prime({\bm v^\prime}_2).
    \end{align*}
    Therefore, if {$F_\delta = (1-\delta) F + \delta G$}, then,
    $$\zeta_\sigma({F_\delta}) = (1-\delta)^2 \zeta_\sigma(F)+\delta^2 \zeta_\sigma(G) + 2\delta(1-\delta) \zeta_\sigma^\prime (G,F)$$
    where 
    $$\zeta_\sigma^\prime(G,F) = \E\{K({\bm V}_1,{\bm V}_2)\}+\E\{K({\bm V}_1^\prime, {\bm V}_2^\prime)\}-\E\{K({\bm V}_1,{\bm V}_2^\prime)\}-\E\{K({\bm V}_1^\prime,{\bm V}_2)\}$$
    where ${\bm V}_1\sim G$ and ${\bf Y}_2\sim F$ are independent and ${\bm V}_1^\prime$ and ${\bm V}_2^\prime$ are their respective null exchangeable pairs. 
\end{proof}

\begin{remark}
    Note that if $\zeta_\sigma(G) = 0$ then $\bm V_1\stackrel{D}{=}\bm V_2^\prime$, which leads to $\zeta_\sigma^\prime(G,F) = 0$. Hence, if $\zeta(G)=0$, then $\zeta_\sigma(F_\delta) = (1-\delta)^2\zeta_\sigma(F)$.
\end{remark}

\begin{proof}[\bf Proof of Theorem \ref{thm:local-limit-per}:]
       Let $\bm V_1 = ({\bm X}_1, \bm Y_1, \bm Z_1),\ldots, \bm V_n = ({\bm X}_n, \bm Y_n, \bm Z_n) \stackrel{iid}{\sim}F_{1-\beta_n n^{-1/2}} = (1-\frac{\beta_n}{\sqrt{n}}) G+ \frac{\beta_n}{\sqrt{n}} F$, where $\zeta(G) = 0$, $\zeta(F)>0$ and $\beta_n\rightarrow\beta\in[0,\infty)$. Let $\bm V_1^\prime, \ldots, \bm V_n^\prime$ be their respective null exchangeable pairs. Also, assume that $\pi_1,\pi_2,\ldots,\pi_n\stackrel{iid}{\sim}\textnormal{Bernoulli}(0.5)$. The resampled test statistic $\hat\zeta_{n,\sigma}(\pi)$ can be written as
    $$\hat\zeta_{n,\sigma}(\pi)=\frac{1}{n(n-1)}\sum_{1\leq i\not=j\leq n} g\big(({\bm U}_i,{\bm U}_i^\prime),({\bm U}_j,{\bm U}_j^\prime)\big),$$
    where ${\bm U}_i = \pi_i {\bm X}_i+(1-\pi_i){\bm X}_i^\prime$, ${\bm U}_i^\prime = (1-\pi_i){\bm X}_i+\pi_i {\bm X}_i^\prime$ and $g(\cdot, \cdot)$ as in (5) in the main article. First, we will prove that under $F_{1-\beta_n/\sqrt{n}}$ where $\beta_n\rightarrow\beta\in[0,\infty)$, we have
    $$n\hat\zeta_{n,\sigma}(\pi) \stackrel{D}{\rightarrow}\sum_{i=1}^\infty\lambda_i (Z_i^2-1)$$
    weakly in probability, where $Z_i$s are i.i.d. standard normal random variables independent of $\bm V_i$s and $\lambda_i$s are the eigenvalues of the integral equation 
    $$\E\left\{g\big((\bm x, \bm x^\prime, \bm y, \bm z),(\bm X, \bm X^\prime, \bm Y, \bm Z)\big)\right\} = \lambda \psi(\bm x, \bm x^\prime, \bm y, \bm z),$$
    when $(\bm X,\bm Y,\bm Z)\sim G$ and $(\bm X^\prime, \bm Y,\bm Z)$ is its null exchangeable pair.
    
    Denote, $\mathcal{D}^\prime = \{(\bm X_i, \bm X_i^\prime, \bm Y_i, \bm Z_i)\}_{1\leq i\leq n}$ as the augmented data set as before. So, conditioned on $\mathcal{D}^\prime$, the randomness within $\hat\zeta_{n,\sigma}(\pi)$ comes from $\pi_1,\ldots, \pi_n$ only. Then, one gets

    \begin{align*}
        & \E\Big[\exp\big\{-\frac{\|(\bm U_1,\bm U_1^\prime, \bm Y_1, \bm Z_1) - (\bm U_2,\bm U_2^\prime, \bm Y_2, \bm Z_2)\|^2}{2d}\big\}\big|(\bm U_1, \bm U_1^\prime, \bm Y_1, \bm Z_1), \mathcal{D}^\prime\Big]\\
        & = \frac{1}{2} \sum_{\pi\in\{0,1\}} \left[\exp\big\{-\frac{\|(\bm U_1,\bm U_1^\prime, \bm Y_1, \bm Z_1) - \pi (\bm X_2, \bm X_2^\prime, \bm Y_2, \bm Z_2) - (1-\pi) (\bm X_2^\prime, \bm X_2, \bm Y_2, \bm Z_2)\|^2} {2d}\big\} \right] = \eta \text{ (say)}.
    \end{align*}
    (Note that here the randomness is due to $\pi_2$ only). Similarly, one can also get
    \begin{align*}
        & \E\Big[\exp\big\{-\frac{\|(\bm U_1,\bm U_1^\prime, \bm Y_1, \bm Z_1) - (\bm U_2^\prime,\bm U_2, \bm Y_2, \bm Z_2)\|^2}{2d}\big\}\big|(\bm U_1, \bm U_1^\prime, \bm Y_1, \bm Z_1), \mathcal{D}^\prime\Big]\\
        & = \E\Big[\exp\big\{-\frac{\|(\bm U_1^\prime,\bm U_1, \bm Y_1, \bm Z_1) - (\bm U_2,\bm U_2^\prime, \bm Y_2, \bm Z_2)\|^2}{2d}\big\}\big|(\bm U_1, \bm U_1^\prime, \bm Y_1, \bm Z_1), \mathcal{D}^\prime\Big]\\
        & = \E\Big[\exp\big\{-\frac{\|(\bm U_1^\prime,\bm U_1, \bm Y_1, \bm Z_1) - (\bm U_2^\prime,\bm U_2, \bm Y_2, \bm Z_2)\|^2}{2d}\big\}\big|(\bm U_1, \bm U_1^\prime, \bm Y_1, \bm Z_1), \mathcal{D}^\prime\Big] = \eta.
    \end{align*}
    Hence, one gets 
    \begin{align}
        \E[g\big((\bm U_1,\bm U_1^\prime, \bm Y_1, \bm Z_1),(\bm U_2,\bm U_2^\prime, \bm Y_1, \bm Z_1)\big)\big|(\bm U_1, \bm U_1^\prime, \bm Y_1, \bm Z_1),\mathcal{D}^\prime] = 0.\label{eq:conditional-degeneracy}
    \end{align}
    Therefore, the limiting conditional distribution of $n\hat\zeta_{n,\sigma}(\pi)$ should be evaluated as in the case of unconditional degenerate U-statistics. Using the eigen decomposition of $g(\cdot,\cdot)$ with respect to $G$, one gets
    $$g\big((\bm x_1, \bm x_1^\prime,\bm y_1, \bm z_1),(\bm x_2, \bm x_2^\prime, \bm y_2,\bm z_2)\big) = \sum_{i=1}^\infty \lambda_i \psi_i(\bm x_1,\bm x_1^\prime, \bm y_1,\bm z_1) \psi_i(\bm x_2,\bm x_2^\prime, \bm y_2,\bm z_2),$$
    where $\lambda_i$ s and $\psi_i$s are as in Proposition 2.3. For any $K\in\mathbb{N}$, define $$g_K\big((\bm x_1,\bm x_1^\prime,\bm y_1,\bm z_1),(\bm x_2,\bm x_2^\prime,\bm y_2,\bm z_2)\big) = \sum_{i=1}^K \lambda_i \psi_i(\bm x_1,\bm x_1^\prime, \bm y_1,\bm z_1) \psi_i(\bm x_2,\bm x_2^\prime, \bm y_2,\bm z_2).$$ 
    Let $U_n(f) := \frac{1}{(n-1)} \sum_{1\leq i \not = j\leq n} f\big((\bm U_i,\bm U_i^\prime, \bm Y_i, \bm Z_i),(\bm U_j,\bm U_j^\prime, \bm Y_j, \bm Z_j)\big)$. Then,
    \begin{align*}
        & \E\big[\big(U_n(g-g_K)\big)^2\big|\mathcal{D}^\prime\big]\\
        & = \frac{1}{(n-1)^2} \sum_{1\leq i_1\not=j_1\leq n} \sum_{1\leq i_2\not=j_2\leq n}\E\left[(g-g_K)\big((\bm U_{i_1},\bm U_{i_1}^\prime, \bm Y_{i_1}, \bm Z_{i_1}),(\bm U_{j_1},\bm U_{j_1}^\prime, \bm Y_{j_1}, \bm Z_{j_1})\big)\right.\\
        &\hspace{2.5in}\left.(g-g_K)\big((\bm U_{i_2},\bm U_{i_2}^\prime, \bm Y_{i_2}, \bm Z_{i_2}),(\bm U_{j_2},\bm U_{j_2}^\prime, \bm Y_{j_2}, \bm Z_{j_2})\big)\big| \mathcal{D}^\prime\right] \\
        & =  \frac{2}{(n-1)^2} \sum_{1\leq i_1\not= j_1\leq n} \E\Big[\Big((g-g_K)\big((\bm U_{i_1},\bm U_{i_1}^\prime, \bm Y_{i_1}, \bm Z_{i_1}),(\bm U_{j_1},\bm U_{j_1}^\prime, \bm Y_{j_1}, \bm Z_{j_1})\big)\Big)^2\big| \mathcal{D}^\prime\Big] ~~~~\text{(by \eqref{eq:conditional-degeneracy})}\\
        & = \frac{2}{(n-1)^2} \sum_{1\leq i_1\not= j_1\leq n} \frac{1}{4}\sum_{\pi_{i_1}, \pi_{j_1}\in\{0,1\}}\Big((g-g_K)\big((\pi_{i_1}\bm W_{i_1} + (1-\pi_{i_1})\bm W_{i_1}^\prime,\pi_{i_1}\bm W_{i_1}^\prime + (1-\pi_{i_1})\bm W_{i_1}), \\
        &\hspace{2.7in} (\pi_{j_1}\bm W_{j_1} + (1-\pi_{j_1})\bm W_{j_1}^\prime,\pi_{j_1}\bm W_{j_1}^\prime + (1-\pi_{j_1})\bm W_{j_1})\big)\Big)^2\\
        & = \frac{1}{2} [\bm \Sigma_I + \bm \Sigma_{II} + \bm \Sigma_{III} + \bm \Sigma_{IV}],   
    \end{align*}
    where $\bm W_j = (\bm X_j, \bm X_j^\prime, \bm Y_j, \bm Z_j)$ and $\bm W_j^\prime = (\bm X_j^\prime, \bm X_j, \bm Y_j, \bm Z_j)$ and
    \begin{align*}
        \bm \Sigma_I & = \frac{1}{(n-1)^2}  \sum_{1\leq i_1\not= j_1\leq n} \Big((g-g_K)\big(\bm W_{i_1},\bm W_{j_1}\Big)^2,\\
        \bm \Sigma_{II} & = \frac{1}{(n-1)^2}  \sum_{1\leq i_1\not= j_1\leq n} \big((g-g_K)\big(\bm W_{i_1}^\prime ,\bm W_{j_1}\big)\big)^2,\\
        \bm \Sigma_{III} & = \frac{1}{(n-1)^2}  \sum_{1\leq i_1\not= j_1\leq n} \big((g-g_K)\big(\bm W_{i_1},\bm W_{j_1}^\prime\big)^2, \textnormal{ and}\\
        \bm \Sigma_{IV} & = \frac{1}{(n-1)^2}  \sum_{1\leq i_1\not= j_1\leq n} \big((g-g_K)\big(\bm W_{i_1}^\prime,\bm W_{j_1}^\prime\big)^2.
    \end{align*}
    Clearly, by the strong consistency of U-statistics and contiguity of $F_{1-\beta_n/\sqrt{n}}$ with $G$, one gets
    $$\E\big[\big(U_n(g-g_K)\big)^2\big|\mathcal{D}^\prime\big] \stackrel{a.s.}{\rightarrow}2 \E\left[\Big((g-g_K)\big((\bm U_{1} ,\bm U_{1}^\prime, \bm Y_1, \bm Z_1),(\bm U_{2},\bm U_{2}^\prime, \bm Y_2, \bm Z_2\big)\Big)^2\right]$$
    as $n$ grows to infinity. The limiting quantity converges to zero as $K$ diverges to infinity. This ensures the closeness of the limiting distribution of $U_n(g)$ and $U_n(g_K)$ when $K$ is large.
    
    Now, let us focus on the limiting distribution of $U_n(g_K)$. First, let us find the joint limiting distribution of $\frac{1}{\sqrt{n}}\sum_{i=1}^n\Big(\psi_1\big(\bm U_i,\bm U_i^\prime, \bm Y_i, \bm Z_i\big), \ldots, \psi_K\big((\bm U_i,\bm U_i^\prime, \bm Y_i, \bm Z_i)\big)\Big)$ conditioned on $\mathcal{D}^\prime$, where $\psi_i$s are the eigenvectors as defined before. By the Cramer-Wold device, it is enough to focus on the conditional limit distribution of
    $$\frac{1}{\sqrt{n}}\sum_{i=1}^n \left[t_1\psi_1\big((\bm U_1,\bm U_1^\prime, \bm Y_1, \bm Z_1)\big) + \ldots + t_K\psi_K\big((\bm U_K,\bm U_K^\prime, \bm Y_K, \bm Z_K)\big)\right]$$
    for some real $t_1,\ldots,t_K$. Note that since $$g\big(({\bm x}_1,{\bm x}_1^\prime, \bm y_1, \bm z_1),({\bm x}_2,{\bm x}_2^\prime, \bm y_2, \bm z_2)\big) = -g\big(({\bm x}_1^\prime,{\bm x}_1, \bm y_1, \bm z_1),({\bm x}_2,{\bm x}_2^\prime, \bm y_2, \bm z_2)\big),$$ one gets
    \begin{align}
    \label{eq:anti-symm-phi}
        \psi_k({\bm x}_1, {\bm x}_1^\prime, \bm y_1, \bm z_1) & = \int g\big(({\bm x}_1,{\bm x}_1^\prime, \bm y_1, \bm z_1),({\bm x}_2,{\bm x}_2^\prime, \bm y_2, \bm z_2)\big) \psi_i({\bm x}_2,{\bm x}_2^\prime, \bm y_2, \bm z_2) \mathrm d \nu_+({\bm x}_2,{\bm x}_2^\prime, \bm y_2, \bm z_2)\nonumber\\
        & = - \int g\big(({\bm x}_1^\prime,{\bm x}_1, \bm y_1, \bm z_1),({\bm x}_2,{\bm x}_2^\prime, \bm y_2, \bm z_2)\big) \mathrm d \nu_+({\bm x}_2,{\bm x}_2^\prime, \bm y_2, \bm z_2) = -\psi_i({\bm x}_1, {\bm x}_1^\prime, \bm y_1, \bm z_1),
    \end{align}
    for $i=1,2,\ldots,K$, where $\nu_+$ denotes the joint distribution of $(\bm X_1, \bm X_1^\prime, \bm Y_1, \bm Z_1)$ when $(\bm X_1, \bm Y_1, \bm Z_1)\sim G$ and $(\bm X_1^\prime,\bm Y_1, \bm Z_1)$ is its null exchangeable pair. So, 
    \begin{align*}
        &\E[\psi_i(\bm U_1, \bm U_1^\prime, \bm Y_1, \bm Z_1)\big| \mathcal{D}] = \frac{1}{2}\big(\psi_i(\bm X_1, \bm X_1^\prime, \bm Y_1, \bm Z_1) + \psi_i(\bm X_1^\prime, \bm X_1, \bm Y_1, \bm Z_1)\big) \\
        & = \frac{1}{2}\big(\psi_i(\bm X_1, \bm X_1^\prime, \bm Y_1, \bm Z_1) - \psi_i(\bm X_1, \bm X_1^\prime), \bm Y_1, \bm Z_1\big) = 0, \text{ 
 and}\\
    %\end{align*} 
    %\begin{align*}
       & \E[\psi_i(\bm U_1, \bm U_1^\prime,\bm Y_1, \bm Z_1)^2]\\
       & = \Big(1-\frac{\beta_n}{\sqrt{n}}\Big) \int \psi_i(\xvec,\xvec^\prime, \bm y, \bm z)^2 \mathrm d\nu_+ (\xvec, \xvec^\prime, \bm y, \bm z) + \frac{\beta_n}{\sqrt{n}} \int \psi_i(\xvec,\xvec^\prime, \bm y, \bm z)^2 \mathrm d\nu_+^\ast (\xvec, \xvec^\prime, \bm y, \bm z)\\
        & \hspace{2.0in} \rightarrow \int \psi_i(\xvec,\xvec^\prime, \bm y, \bm z)^2 \mathrm d\nu_+ (\xvec, \xvec^\prime, \bm y, \bm z) = 1, \textnormal{ as $n\rightarrow \infty$,}.
    \end{align*} 
    where $\nu_+$ and $\nu_+^\ast$ are the joint distribution of $({\bm X}_1,{\bm X}_1^\prime, \bm Y, \bm Z)$ when $({\bm X}_1, \bm Y_1, \bm Z_1)\sim G$ and $({\bm X}_1, \bm Y_1, \bm Z_1)\sim F$, respectively and $(\bm X_1^\prime, \bm Y_1, \bm Z_1)$ is its null exchangeable pair. Then, following similar arguments as in Lemma \ref{conditional-CLT}, under $F_{1-\beta_n/\sqrt{n}}$, one gets
    \begin{align*}
        \prod_{j=1}^n \E\Big[e^{it \frac{1}{\sqrt{n}} \psi_i({\bm U}_j,{\bm U}_j^\prime, \bm Y_j, \bm Z_j) }\big|\mathcal{D}^\prime\Big] & = \exp\Big\{-\frac{t^2}{2} \frac{1}{n}\sum_{j=1}^n \big(\psi_i({\bm X}_j,{\bm X}_j^\prime, \bm Y_j, \bm Z_j)^2+\psi_i({\bm X}_j^\prime,{\bm X}_j, \bm Y_j, \bm Z_j)^2\big)\Big\} + o_P(1)\\
        & = \exp\Big\{-\frac{t^2}{2} \E[\varphi({\bm U}_1,{\bm U}_1^\prime, \bm Y_1, \bm Z_1)^2]\Big\} + o_P(1).
    \end{align*}
    The probability convergence of $\frac{1}{n}\sum_{j=1}^n \big(\psi_i({\bm X}_j,{\bm X}_j^\prime, \bm Y_j, \bm Z_j)^2+\psi_i({\bm X}_j^\prime,{\bm X}_j, \bm Y_j, \bm Z_j)^2\big)$ to $\E[\psi_i({\bm U}_1,{\bm U}_1^\prime, \bm Y_1, \bm Z_1)^2]$ also follows by the contiguity of $F_{1-\beta_n/\sqrt{n}}$ and $G$. Therefore, using simple algebraic arguments, one gets that conditioned on $\mathcal{D}^\prime$, 
    \begin{align*}
        &\frac{1}{\sqrt{n}}\sum_{i=1}^n \left[t_1\psi_1\big(\bm U_i,\bm U_i^\prime, \bm Y_i, \bm Z_i\big) + \ldots + t_K\psi_K\big(\bm U_i,\bm U_i^\prime, \bm Y_i, \bm Z_i\big)\right]\\ &\hspace{2.0in} \stackrel{D}{\rightarrow} \mathcal{N}_1\big(0, \sum_{\ell=1}^K t_\ell^2\E[\psi_\ell(\bm U_1, \bm U_1^\prime, \bm Y_1, \bm Z_1)^2]\big)
         =  \mathcal{N}_1\big(0, \sum_{\ell=1}^K t_\ell^2\big)
    \end{align*}
    in probability as $n$ grows to infinity. By simple application of weak law of large numbers, it is easy to see that conditioned on $\mathcal{D}^\prime$, for any $\ell\in \mathbb{N}$, $\frac{1}{n}\sum_{j=1}^n \psi_j(\bm U_j, \bm U_j^\prime, \bm Y_j, \bm Z_j)^2$ converges in probability to $\E[\psi_j(\bm U_1, \bm U_1^\prime)^2] = 1$ as $n$ grows to infinity. Now, by continuous mapping theorem one can conclude that
    $$U_n(g_K) = \sum_{\ell=1}^K \lambda_\ell \Big( \big( \frac{1}{\sqrt{n}} \sum_{j = 1}^n \psi_\ell(\bm U_j, \bm U_j^\prime) \big)^2  - \frac{1}{n} \sum_{j=1}^n \psi_{\ell}(\bm U_j, \bm U_j^\prime)^2\Big)\stackrel{D}{\rightarrow} \sum_{\ell=1}^K \lambda_\ell (Z_\ell ^2 - 1)$$
    as $n$ grows to infinity, where $Z_\ell$s are i.i.d. standard normal random variables independent of $\bm V_i$s and $\lambda_i$s are the eigenvalues of the integral equation
    $$\E\left\{g\big((\bm x, \bm x^\prime, \bm y, \bm z),(\bm X, \bm X^\prime, \bm Y, \bm Z)\big)\psi(\bm X, \bm X^\prime, \bm Y, \bm Z)\right\} = \lambda \psi(\bm x, \bm x^\prime, \bm y, \bm z),$$
where $g(\cdot, \cdot)$ is as defined in (5) in the main article and $\bm V = (\bm X, \bm Y, \bm Z)\sim G$ and $\bm V^\prime$ is the null exchangeable pair of $\bm V$. The rest of the arguments can be established using approximation arguments using characteristic functions of $U_n(g_K), U_n(g), \sum_{i=1}^K\lambda_i(Z_i^2-1)$ and $\sum_{i=1}^\infty\lambda_i(Z_i^2-1)$, as in Lemma 1 from \cite{banerjee2024twosam}. We also refer the interested readers to Chapter 12 of \cite{van2000asymptotic} or page 79 of \cite{lee2019u} for similar arguments. Now,

    \begin{itemize}
        \item[(a)] By Theorem~2.4, we have  
\[
n\hat\zeta_{n,\sigma} \xrightarrow{d} \sum_{i=1}^\infty \lambda_i (Z_i^2 - 1),
\]
where the random variables $\{Z_i\}$ are i.i.d. standard normal random variables. Moreover, under the condition $\beta_n \to 0$, we have established that the conditional distribution of $n\hat\zeta_\sigma(\pi)$ converges weakly, in probability, to the same limiting distribution $\sum_{i=1}^\infty \lambda_i ({Z_i^\prime}^2 - 1)$, where $Z_i^\prime$s are independent of $Z_i$s, but $\lambda_i$s are the same sequence of eigenvalues.

It follows that for any $\alpha \in (0,1)$, the conditional quantiles $\hat c_{1-\alpha}$ of the resampling distribution converge to the corresponding theoretical quantiles of the null distribution $\sum_{i=1}^\infty \lambda_i (Z_i^2 - 1)$. Hence,
\[
P\!\left(\hat\zeta_{n,\sigma} > \hat c_{1-\alpha}\right) \;\rightarrow\; \alpha,
\qquad n \to \infty.
\]

\item[(b)] By Theorem~3.1, under $F_{1-\beta_n/\sqrt{n}}$, where $\beta_n\rightarrow\beta\in(0,\infty)$, we have  
\[
n\hat\zeta_{n,\sigma} \xrightarrow{d} \sum_{i=1}^\infty \lambda_i \big((Z_i+\beta\E[\psi_i(\bm X_1,\bm X_1^\prime, \bm Y_1, \bm Z_1)])^2 - 1),
\]
where the random variables $\{Z_i\}$ are i.i.d. standard normal random variables and $\lambda_i$s are the eigenvalues of the integral equation 
$$\E\left\{g\big((\bm x, \bm x^\prime, \bm y, \bm z),(\bm X, \bm X^\prime, \bm Y, \bm Z)\big)\psi(\bm X, \bm X^\prime, \bm Y, \bm Z)\right\} = \lambda \psi(\bm x, \bm x^\prime, \bm y, \bm z),$$
where $g(\cdot, \cdot)$ is as defined (5) in the main article and $\bm V = (\bm X, \bm Y, \bm Z)\sim G$ and $\bm V^\prime$ is the null exchangeable pair of $\bm V$. Moreover, under the condition $\beta_n \to \beta$, we have established that the conditional distribution of $n\hat\zeta_\sigma(\pi)$ converges weakly, in probability, to $\sum_{i=1}^\infty \lambda_i ({Z_i^\prime}^2 - 1)$, where $Z_i^\prime$s are independent of $Z_i$s and $\lambda_i$s are as before.

It follows that for any $\alpha \in (0,1)$, the conditional quantiles $\hat c_{1-\alpha}$ of the resampling distribution converge to the corresponding theoretical quantiles $z_{1-\alpha}$ of the null distribution $\sum_{i=1}^\infty \lambda_i (Z_i^2 - 1)$. Hence,
\[
P\!\left(\hat\zeta_{n,\sigma} > \hat c_{1-\alpha}\right) \;\rightarrow\; P\left[\sum_{i=1}^\infty \lambda_i (Z_i+\beta^2\E[\psi_i(\bm X_1,\bm X_1^\prime, \bm Y_1, \bm Z_1)] - 1) >z_{1-\alpha}\right] = \mathcal{L}(\beta),
\qquad n \to \infty.
\]

\item[(c)] Using Lemma \ref{lem:relation-contamination} we get $\zeta_\sigma(F_{1-\beta_nn^{-1/2}}) = \frac{\beta_n^2}{n}\zeta_\sigma(F)$. Then by \eqref{eq:type-II-bound}, we get that under $F_{1-\beta_n/\sqrt{n}}$
\begin{align*}
P[\hat\zeta_{n,\sigma}<\hat c_{1-\alpha}] & \leq \frac{{n\choose 2}^{-1}\left[4(n-2) \zeta_\sigma(F_{1-\beta_nn^{-1/2}}) + 4\right]}{\big(\zeta_\sigma(F_{1-\beta_n/\sqrt{n}})-2((n-1)\alpha)^{-1}\big)^2}\\
& = \frac{{n\choose 2}^{-1}\left[4(n-2) \frac{\beta_n^2}{n} \zeta_\sigma(F) + 4\right]}{\big(\frac{\beta_n^2}{n}\zeta_\sigma(F)-2((n-1)\alpha)^{-1}\big)^2} = O(\beta_n^{-2})~~~\text{(using Lemma \ref{lem:relation-contamination})}.
\end{align*}
So, if $\beta_n\rightarrow\infty$ as $n$ diverges to infinity, $P[\hat\zeta_{n,\sigma}<\hat c_{1-\alpha}]$ (the Type-II error rate) converges to zero, i.e., the power of the test converges to one. %Therefore, if $\beta_n\rightarrow\infty$, we get that the type II error rate of the test converges to zero. In other words, the power of the test converges to one.
\end{itemize}
This completes the proof.

\end{proof}

\begin{proof}[\bf Proof of Proposition \ref{prop:baic-prop}:]
    Recall that $\hat\zeta_{n,\sigma}$ is a U-statistics with the kernel
    $$g\big(({\bm x}_1,{\bm x}_1^\prime, \bm y_1, \bm z_1),({\bm x}_2,{\bm x}_2^\prime, \bm y_2, \bm z_2)\big) = K({\bm v}_1,{\bm v}_2) + K({\bm v}_1^\prime,{\bm v}_2^\prime) - K({\bm v}_1,{\bm v}_2^\prime) - K({\bm v}_1^\prime,{\bm v}_2),$$
    where $K({\bf x},{\bf y}) = \exp\{-\sigma^2\|{\bf x}-{\bf y}\|^2/2\}, \bm v_i = (\bm x_i, \bm y_i, \bm z_i)$ and $\bm v_i^\prime = (\bm x_i^\prime, \bm y_i, \bm z_i), i = 1,2$. Then, by the theory of U-statistics \citep[see p.12 in ][]{lee2019u} the unbiasedness of the estimator follows trivially and, one has
    \begin{align*}
        \var(\hat\zeta_{n,\sigma}) & = {n\choose 2}^{-1}\left[{2\choose 1}{n-2\choose 1} \var\Big(g_1\big({\bm X}_1,{\bm X}_1^\prime, \bm Y_1, \bm Z_1\big)\Big) \right.\\
        & ~~~~~~~~~~~~~~~~~~~~~~ \left.+ {2\choose 2}{n-2\choose 0} \var\Big(g\big(({\bm X}_1,{\bm X}_1^\prime, \bm Y_1, \bm Z_1),({\bm X}_2,{\bm X}_2^\prime, \bm Y_2, \bm Z_2)\big)\Big)\right]
    \end{align*}
    where,
    \begin{equation*}
    \begin{split}
        & g_1({\bm x}_1,{\bm x}_1^\prime, \bm y_1, \bm z_1)\\
        & = \E\{\exp(-\frac{\sigma^2}{2}\|{\bm v}_1-{\bm V}_2\|^2)\}-\E\{\exp(-\frac{\sigma^2}{2}\|{\bm v}_1-{\bm V}_2^\prime\|^2)\}\\
        &~~~~~~~~+\E\{\exp(-\frac{\sigma^2}{2}\|{\bm v}_1^\prime-{\bm V}_2^\prime\|^2)\}-\E\{\exp(-\frac{\sigma^2}{2}\|{\bm v}_1^\prime-{\bm V}_2\|^2)\},
    \end{split}
\end{equation*}    
$\bm v_1, \bm v_2$ are as defined before and $\bm V_1, \bm V_1^\prime$ are the associated random vector.
Since, $|g(\cdot,\cdot)|\leq 2$, $\var\Big(g\big(({\bm X}_1,{\bm X}_1^\prime, \bm Y_1, \bm Z_1),({\bm X}_2,{\bm X}_2^\prime, \bm Y_2, \bm Z_2)\big)\Big)$ is bounded by $4$. Whereas to bound the first term note that
    \begin{align*}
        g_1({\bm x}_1,{\bm x}_1^\prime, \bm y_1, \bm z_1) & = \E\{\int \exp(i\langle {\bf t}, {\bm v}_1-{\bm V}_2 \rangle) \phi_0({\bf t}) \mathrm d{\bf t}\} - \E\{\int \exp(i\langle {\bf t}, {\bm v}_1-{\bm V}_2^\prime \rangle) \phi_0({\bf t}) \mathrm d{\bf t}\}\\
        &\hspace{0.75in}-\E\{\int \exp(i\langle {\bf t}, {\bm v}_1^\prime-{\bm V}_2 \rangle) \phi_0({\bf t}) \mathrm d{\bf t}\}+\E\{\int \exp(i\langle {\bf t}, {\bm v}_1^\prime-{\bm V}_2^\prime \rangle) \phi_0({\bf t}) \mathrm d{\bf t}\}\tag*{\text{(where $\phi_0({\bf t})$ denotes the density of $N({\bf 0}, \sigma^2{\bf I}_d)$)}}\\
        & = \int \big(\exp(i\langle {\bf t}, {\bm v}_1 \rangle) - \exp(i\langle {\bf t}, {\bm v}_1^\prime \rangle)\big)\big(\E\{\exp(-i\langle {\bf t}, {\bm V}_2 \rangle)\} - \E\{\exp(-i\langle {\bf t}, {\bm V}_2^\prime \rangle)\}\big) \phi_0({\bf t}) \mathrm d{\bf t}\\
        & = \int (\exp(i\langle {\bf t}, x_1 \rangle) - \exp(i\langle {\bf t}, {\bm v}_1^\prime \rangle))(\varphi_{1}(-{\bf t}) - \varphi_{2}(-{\bf t})) \phi_0({\bf t}) \mathrm d{\bf t}\\
        & = \int (\exp(i\langle {\bf t}, {\bm v}_1 \rangle) - \exp(i\langle {\bf t}, {\bm v}_1^\prime \rangle))\overline{(\varphi_{1}({\bf t}) - \varphi_{2}({\bf t}))} \phi_0({\bf t}) {\mathrm d}{\bf t}.
     \end{align*}
Then using Cauchy-Schwartz inequality, one has
$$g_1^2({\bm x}_1,{\bm x}_1^\prime, \bm y_1, \bm z_1) \leq \int |\exp(i\langle {\bf t}, {\bm v}_1 \rangle) - \exp(i\langle {\bf t}, {\bm v}_1^\prime \rangle)|^2 \phi_0({\bf t}) {\mathrm d}{\bf t} \times \int |\varphi_{1}({\bf t}) - \varphi_{2}({\bf t})|^2 \phi_0({\bf t}) {\mathrm d}{\bf t} \leq 2 \zeta_\sigma(\Pr).$$

This gives
$$\var(\hat\zeta_{n,\sigma}) \leq {n\choose 2}^{-1}\left[4(n-2) \zeta_\sigma(\Pr) + 4\right].$$
This completes the proof.    
\end{proof}

\begin{proof}[\bf Proof of Theorem \ref{thm:minimax-lower-bound}:]
    We prove this theorem using a standard application of the Neyman-Pearson lemma. Let $Q_1$ and $Q_2$ be the joint distribution of the sample $\{(\bm X_i, \bm X_i^\prime,\bm Y_i, \bm Z_i)\}_{i=1}^n$ under the null and alternative hypotheses respectively. Then we can lower bound the minimax risk $R_{n}(\epsilon)$ as follows
    $$R_n(\epsilon)\geq 1-\alpha - d_{TV}(Q_1,Q_2)\geq 1-\alpha-\sqrt{\frac{1}{2}KL(Q_1,Q_2)}.$$
    The first inequality follows using the fact that $\P_{Q_1}\{\phi=1\}\leq \alpha$ and
    $$\P_{Q_1}\{\phi=1\} + \P_{Q_2}\{\phi=0\} = 1-(\P_{Q_1}\{\phi=0\}-\P_{Q_2}\{\phi=0\}) \geq 1-d_{TV}(Q_1,Q_2)$$
    where $d_{TV}$ denotes the total variation distance between $Q_1$ and $Q_2$. The second inequality follows from Pinsker's inequality \citep[see.][]{tsybakov2009introduction}. Suppose $G$ is a distribution that admits a p.d.f. $g$ with $\zeta_\sigma(G)=0$ and $F$ is a distribution with p.d.f. $f$ such that $\zeta_\sigma(F) = \gamma_0>0$. Assume that $\int (f({\bf u})/g({\bf u})-1)^2 g({\bf u})\mathrm d{\bf u} = \gamma_1<\infty$. Then setting $Q_1 = \prod_{i=1}^n \Big(\big(1-\frac{\delta}{\sqrt{n}}\big)G+\frac{\delta}{\sqrt{n}} F\Big)$ for some $\delta>0$ and $Q_2 = G^n$, we have
    \begin{align*}
        KL(Q_1,Q_2) & = \int \log\prod_{i=1}^n\left\{1+\frac{\delta}{\sqrt{n}}\big(\frac{f({\bf u}_i)}{g({\bf u}_i)}-1\big)\right\} \prod_{i=1}^n d\Big(\big(1-\frac{\delta}{\sqrt{n}}\big)G+\frac{\delta}{\sqrt{n}} F\Big)({\bf u}_i)\\
        & = n \int \log\left\{1+\frac{\delta}{\sqrt{n}}\big(\frac{f({\bf u}_1)}{g({\bf u}_1)}-1\big)\right\} d\Big(\big(1-\frac{\delta}{\sqrt{n}}\big)G+\frac{\delta}{\sqrt{n}} F\Big)({\bf u}_1)\\
        & = n \big(1-\frac{\delta}{\sqrt{n}}\big) \int \log\left\{1+\frac{\delta}{\sqrt{n}}\big(\frac{f({\bf u}_1)}{g({\bf u}_1)}-1\big)\right\} g({\bf u}_1) \mathrm d {\bf u}_1\\
        &~~~~~~~~~~~+n\frac{\delta}{\sqrt{n}} \int \log\left\{1+\frac{\delta}{\sqrt{n}}\big(\frac{f({\bf u}_1)}{g({\bf u}_1)}-1\big)\right\} f({\bf u}_1) \mathrm d{\bf u}_1
    \end{align*}
    Using the inequality $\log(1+y)\leq y$ we get,
    \begin{align*}
        KL(Q_1,Q_2) & \leq n\left[ \frac{\delta}{\sqrt{n}} (1-\frac{\delta}{\sqrt{n}}) \int \big(\frac{f({\bf u})}{g({\bf u})}-1\big) g({\bf u}) d{\bf u} + \frac{\delta^2}{n} \int \big(\frac{f({\bf u})}{g({\bf u})}-1\big) f({\bf u}) \mathrm d{\bf u} \right]\\
        & = \delta^2 \left[\int \frac{f^2({\bf u})}{g({\bf u})} - 1\right] \mathrm d {\bf u}= \delta^2 \int \left(\frac{f({\bf u})}{g({\bf u})}-1\right)^2 g({\bf u}) {\mathrm d}{\bf u} = \delta^2\gamma_1.
    \end{align*}
    Also by Lemma \ref{lem:relation-contamination} we have,
    $$\zeta_\sigma\left(\Big(\big(1-\frac{\delta}{\sqrt{n}}\big)G+\frac{\delta}{\sqrt{n}} F\Big)\right) = \frac{\delta^2}{n} \zeta_\sigma(F) = \frac{\delta^2}{n} \gamma_0. $$
    Now for any $0<\beta<1-\alpha$ if we choose $\delta = \sqrt{2/\gamma_1}(1-\alpha-\beta)$ we get, 
    $$\zeta_\sigma\left(\Big(\big(1-\frac{\delta}{\sqrt{n}}\big)G+\frac{\delta}{\sqrt{n}} F\Big)\right) = \frac{1}{n} \left(\frac{2\gamma_0(1-\alpha-\beta)^2}{\gamma_1}\right). $$
    Now define, $c(\alpha,\beta) = \big(2\gamma_0(1-\alpha-\beta)^2\big)/\gamma_1$. Then, $\big((1-\frac{\delta}{\sqrt{n}})G+\frac{\delta}{\sqrt{n}} F\big)\in \mathcal{F}(cn^{-1}) = \{F\mid \zeta_\sigma(F)>cn^{-1}\}$ for all $0<c<c(\alpha,\beta)$. For this choice of alternative, we also have $R_n(cn^{-1})\geq \beta$ for all $0<c<c(\alpha,\beta)$. Since $\beta$ and $c(\alpha,\beta)$ does not depend on $n$, this trivially satisfies the condition $\liminf_{n\to\infty} R_n(cn^{-1}) \geq \beta$ for all $0<c<c(\alpha,\beta)$.
    
\end{proof}

\begin{proof}[\bf Proof of Theorem \ref{thm:minimax-upper-bound}:]
    Here we want to show that for every $0<\alpha<1$ and $0<\beta<1-\alpha$ there exists a constant $C(\alpha,\beta)>0$ such that
    $$\limsup_{n\to\infty}\sup_{F\in \mathcal{F}(cn^{-1})} \P_{F^n}\{\hat\zeta_{n,\sigma}\leq c_{1-\alpha}\} \leq \beta$$
    for all $c>C(\alpha,\beta)$. Now take any $\Pr\in\mathcal{F}(cn^{-1})$ with $c>4/\alpha$ (i.e. $\zeta_\sigma({\Pr})>4/n\alpha$). Using the fact $c_{1-\alpha}\leq 2((n-1)\alpha)^{-1}$ and Chebyshev's inequality, we have
    \begin{align*}
        \P_{F^n}\{\hat\zeta_{n,\sigma}\leq c_{1-\alpha}\} \leq \P_{F^n}\{\hat\zeta_{n,\sigma}\leq 2((n-1)\alpha)^{-1}\} & \leq \P_{F^n}\{\zeta_\sigma(\Pr)-\hat\zeta_{n,\sigma}\geq \zeta_\sigma(\Pr) - 2((n-1)\alpha)^{-1}\}\\
        & \leq \frac{\var\big(\hat\zeta_{n,\sigma}\big)}{\big(\zeta_\sigma(\Pr) - 2((n-1)\alpha)^{-1}\big)^2},
    \end{align*}
    which holds since $\zeta_\sigma(\Pr)-2((n-1)\alpha)^{-1}>4(n\alpha)^{-1}-2((n-1)\alpha)^{-1}=%\frac{4n-4-2n}{\alpha n(n-1)} 
    \frac{2n-4}{n(n-1)\alpha}>0$ for all $n\geq 2$. Now,
    \begin{align}
    \label{eq:chebyshev-bound}
        \frac{\var\big(\hat\zeta_{n,\sigma}\big)}{\big(\zeta_\sigma(\Pr) - 2((n-1)\alpha)^{-1}\big)^2} \leq \frac{{n\choose 2}^{-1}\left[4(n-2) \zeta_\sigma(\Pr) + 4\right]}{\big(\zeta_\sigma(\Pr)-2((n-1)\alpha)^{-1}\big)^2}~~~~~~ \mbox{(follows from Lemma \ref{lem:relation-contamination})}
    \end{align}
    which implies that
    \begin{align*}
        \limsup_{n\to\infty}\sup_{F\in\mathcal{F}(cn^{-1})} \P_{F^n}\{\hat\zeta_{n,\sigma}\leq c_{1-\alpha}\} \leq \frac{4c + 4}{\big(c-2/\alpha\big)^2}.
    \end{align*}
    It is easy to see that the upper bound is a monotonically decreasing function of $c$ for $c>4/\alpha$ and it converges to $0$ as $c$ increases. Hence for any $\beta<1-\alpha$, there exists a $r(\alpha,\beta)$ such that the upper bound is smaller than $\beta$ whenever $c>r(\alpha,\beta)$. Now set $C(\alpha,\beta) = \max\{r(\alpha,\beta),4/\alpha\}$. Then for any $c>C(\alpha,\beta)$ the maximum type II error rate of our test is upper bounded by $\beta$. This completes the proof of this theorem.

\end{proof}

\begin{proof}[\bf Proof of Proposition \ref{thm:high-dim-consistency}:]
    Take any distribution $\Pr$. Using the fact $\hat c_{1-\alpha}\leq 2((n-1)\alpha)^{-1}$ and Chebyshev's inequality, one has
    \begin{align*}
        \P_{\Pr^n}\{\hat\zeta_{n,\sigma}\leq \hat c_{1-\alpha}\} & \leq \P_{\Pr^n}\{\hat\zeta_{n,\sigma}\leq 2((n-1)\alpha)^{-1}\}\nonumber\\
        & \leq \P_{\Pr^n}\{\zeta_\sigma(\Pr)-\hat\zeta_{n,\sigma}\geq \zeta_\sigma(\Pr) - 2((n-1)\alpha)^{-1}\}\nonumber\\
        & \leq \frac{\var\big(\hat\zeta_{n,\sigma}\big)}{\big(\zeta_\sigma(\Pr) - 2((n-1)\alpha)^{-1}\big)^2},
    \end{align*}
    which holds since $\zeta_\sigma(\Pr)-2((n-1)\alpha)^{-1}>4(n\alpha)^{-1}-2((n-1)\alpha)^{-1}=%\frac{4n-4-2n}{\alpha n(n-1)} 
    \frac{2n-4}{n(n-1)\alpha}>0$ for all $n\geq 2$. Now,
    \begin{align}
        \P_{\Pr^n}\{\hat\zeta_{n,\sigma}\leq \hat c_{1-\alpha}\} & \leq \frac{\var\big(\hat\zeta_{n,\sigma}\big)}{\big(\zeta_\sigma(\Pr) - 2((n-1)\alpha)^{-1}\big)^2}\nonumber\\
        & \leq \frac{{n\choose 2}^{-1}\left[4(n-2) \zeta_\sigma(\Pr) + 4\right]}{\big(\zeta_\sigma(\Pr)-2((n-1)\alpha)^{-1}\big)^2}~~~~~~ \text{(follows from Proposition 3.2)}\label{eq:type-II-bound}
    \end{align}
    Now, if the distribution $\Pr$ is such that $n\zeta_\sigma(\Pr)$ diverges to infinity, then from equation \eqref{eq:type-II-bound} one gets $\lim\P\{\hat\zeta_{n,\sigma}\leq \hat c_{1-\alpha}\} = 0$. Hence, under the above condition, the power of our test converges to one. 
\end{proof}

{\begin{proof}[\bf Proof of Theorem \ref{thm:excess-type-I}:]

Recall the definition of $p_{n,B}:$
    $$p_{n,B} = \frac{1+\sum_{i=1}^B \mathrm I\big\{\Tilde{\zeta}_{n,\sigma}(\pi_i)\geq \hat\zeta_{n,\sigma}\big\}}{1+B},$$
where $\hat\zeta_{n,\sigma}$ is the estimator of $\zeta_{\sigma}(\mathrm P)$ and $\Tilde{\zeta}_{n,\sigma}(\pi)$ is the resample analog of $\hat\zeta_{n,\sigma}$. For the following arguments we use the notations
\begin{align}
    \hat\zeta_{n,\sigma} = \hat\zeta_{n,\sigma}\big(\mathcal{X}, \mathcal{X}^\prime, \mathcal{Y}, \mathcal{Z}\big)\text{ and }\Tilde{\zeta}_{n,\sigma}(\Pi) = \hat\zeta_{n,\sigma}\big(\Pi\mathcal{X} + (\bm 1-\Pi)\mathcal{X}^\prime, (\bm 1 - \Pi)\mathcal{X} + \Pi\mathcal{X}^\prime, \mathcal{Y}, \mathcal{Z}\big),
\end{align}
where $\mathcal{X} = \{\bm X_1,\bm X_2,\ldots, \bm X_n\}, \mathcal{Y} = \{\bm Y_1,\bm Y_2,\ldots, \bm Y_n\}, \mathcal{Z} = \{\bm Z_1,\bm Z_2,\ldots, \bm Z_n\}$ and $\bm 1 = (1,1,\ldots,1)$, $n$-dimensional vector with each entry being $1$. This helps to put an emphasis on the dependence of the estimators on the samples.

Now assume $\bm X_i^{(1)}\sim Q_{\bm X|\bm Z_i}^{(n)}, i=1,2,\ldots, n$ and define $\mathcal{X}^{(1)}=\{\bm X_1^{(1)},\bm X_{2}^{(1)},\ldots, \bm X_n^{(1)}\}$. Then the randomized p-value based on this approximate sampling scheme is
$$p_{n,B} = \frac{1+\sum_{i=1}^B \mathrm I\big\{\hat\zeta_{n,\sigma}\big(\pi_i\mathcal{X} + (\bm 1-\pi_i)\mathcal{X}^{(1)}, (\bm 1 - \pi_i)\mathcal{X} + \pi_i\mathcal{X}^{(1)}, \mathcal{Y}, \mathcal{Z}\big) \geq \hat\zeta_{n,\sigma}\big(\mathcal{X}, \mathcal{X}^{(1)}, \mathcal{Y}, \mathcal{Z}\big)\big\} }{1+B}.$$
Now define,
$$A_\alpha = \left\{(\mathcal{x},\mathcal{x}^{(1)}): \frac{1+\sum_{i=1}^B \mathrm I\big\{\hat\zeta_{n,\sigma}\big(\pi_i\mathcal{x} + (\bm 1-\pi_i)\mathcal{x}^{(1)}, (\bm 1 - \pi_i)\mathcal{x} + \pi_i\mathcal{x}^{(1)}, \mathcal{Y}, \mathcal{Z}\big) \geq \hat\zeta_{n,\sigma}\big(\mathcal{x}, \mathcal{x}^{(1)}, \mathcal{Y}, \mathcal{Z}\big)\big\} }{1+B}\leq \alpha\right\}.$$
Note that  under the assumption $\bm X_i\indpt \bm Y_i|\bm Z_i$ and $\bm X_i|\bm Z_i \stackrel{D}{=}\bm X_i^{\prime}|\bm Z_i$ for all $i=1,2,\ldots,n$, we have
$$\P\left[\big(\mathcal{X},\mathcal{X}^{\prime}\big)\in A_{\alpha}|\mathcal{Y},\mathcal{Z}\right] = \frac{\lfloor(B+1)\alpha\rfloor}{(B+1)}.$$
Now,
\begin{align*}
    & \P\left[p_{n,B}\leq \alpha|\mathcal{Y}, \mathcal{Z}\right]\\
    & = \P\left[\big(\mathcal{X},\mathcal{X}^{(1)}\big)\in A_{\alpha}|\mathcal{Y},\mathcal{Z}\right]\\
    & \leq \P\left[\big(\mathcal{X},\mathcal{X}^{\prime}\big)\in A_{\alpha}|\mathcal{Y},\mathcal{Z}\right] + d_{\mathrm{TV}}\left(\mathcal{L}\big((\mathcal{X},\mathcal{X}^{(1)})|\mathcal{Y},\mathcal{Z}\big),\mathcal{L}\big((\mathcal{X},\mathcal{X}^{\prime})|\mathcal{Y},\mathcal{Z}\big)\right)\\
    & = \frac{\lfloor(B+1)\alpha\rfloor}{(B+1)} + d_{\mathrm{TV}}\left(\mathcal{L}\big(\mathcal{X}^{(1)}|\mathcal{Z}\big),\mathcal{L}\big(\mathcal{X}^{\prime}|\mathcal{Z}\big)\right)\\
    & = \frac{\lfloor(B+1)\alpha\rfloor}{(B+1)} + d_{\mathrm{TV}}\Bigg(\prod_{i=1}^n Q_{\bm X|\bm Z_i}^{(n)}, \prod_{i=1}^n P_{\bm X|\bm Z_i}\Bigg).
\end{align*}
This completes the proof.
\end{proof}

\begin{proof}[\bf Proof of Proposition \ref{prop:contiguity-consistency}:]
    Note that under the condition 
    $$d_{\mathrm{TV}}\big(\prod_{i=1}^n Q_{\bm X|\bm Z_i}^{(n)}, \prod_{i=1}^n P_{\bm X|\bm Z_i}\big) = o_P(1),$$
    a simple application of Dominated Convergence Theorem states that
    \begin{align*}
        & d_{\mathrm{TV}}\big(\prod_{i=1}^n Q_{(\bm X_i,\bm X_1^{(1)}\bm Y_1,\bm Z_1)}^{(n)}, \prod_{i=1}^n P_{\bm X_i, \bm X_i^\prime, \bm Y_i, \bm Z_i}\big)\\
        & \leq \E_{\{(\bm Y_i,\bm Z_i)\}_{i=1}^n\}}\Big(d_{\mathrm{TV}}\big(\prod_{i=1}^n P_{X|\bm Y_i, \bm Z_i}\otimes Q_{\bm X|\bm Z_i}^{(n)}, \prod_{i=1}^n P_{X|\bm Y_i, \bm Z_i}\otimes P_{\bm X|\bm Z_i}\big)\Big)\\
        & \leq \E_{\{(\bm Y_i,\bm Z_i)\}_{i=1}^n\}}\Big(d_{\mathrm{TV}}\big(\prod_{i=1}^n Q_{\bm X|\bm Z_i}^{(n)}, \prod_{i=1}^n P_{\bm X|\bm Z_i}\big)\Big)\rightarrow 0
    \end{align*}
    as $n$ diverges to infinity. Hence, the joint distribution of the sample $\{(\bm X_i, \bm X_i^{(1)}, \bm Y_i, \bm Z_i)\}_{i=1}^n$ and that of $\{(\bm X_i, \bm X_i^{\prime}, \bm Y_i, \bm Z_i)\}_{i=1}^n$ are mutually contiguous. Since, under the exact sampling scheme we have $\hat\zeta_{n,\sigma}\stackrel{\P}{\rightarrow}\zeta_\sigma(\mathrm P)$, by the definition of contiguity we have
    $$\hat\zeta_{n,\sigma}\stackrel{\P}{\rightarrow}\zeta_\sigma(\mathrm P)$$
    as $n\rightarrow\infty$, under approximate null exchangeable pair sampling scheme.
    
\end{proof}
}

\end{document}